\documentclass[10pt, a4paper]{amsart}
\usepackage{times}
\usepackage{amssymb}
\usepackage[dvipsnames]{xcolor}
\usepackage{amsmath}
\usepackage{amssymb}
\usepackage{amsthm}
\usepackage{bbm}
\usepackage{bm}
\usepackage{mathrsfs}
\usepackage{dsfont}
\usepackage{pgfplots}
\usepackage{esint}
\usepackage{enumitem}
\setenumerate{label={\rm (\alph{*})}}
\usepackage[left=2.75cm,right=2.75cm,top=2.5cm,bottom=2cm]{geometry}
\usepackage[colorlinks=true, linktocpage=true, linkcolor=red!70!black, citecolor=green!50!black]{hyperref}
\usepackage{graphicx}
\usepackage{tikz-cd}
\usetikzlibrary{decorations.pathreplacing}
\usepgfplotslibrary{fillbetween}
\usetikzlibrary{patterns}
\usetikzlibrary{arrows}
\pgfplotsset{compat=1.16} 
\usepackage{scalerel}

\newcommand{\wedgeq}{\scaleto{\mathbf{\bigwedge}}{10pt}}

\numberwithin{equation}{section}

\newtheorem{theorem}{Theorem}[section]
\newtheorem{lemma}[theorem]{Lemma}
\newtheorem{proposition}[theorem]{Proposition}

\newtheorem{corollary}[theorem]{Corollary}

\newtheorem{definition}[theorem]{Definition}

\theoremstyle{remark}

\newtheorem{remark}[theorem]{Remark}
\newtheorem{example}[theorem]{Example}

\numberwithin{equation}{section}

\normalsize

\def\XXint#1#2#3{{\setbox0=\hbox{$#1{#2#3}{\int}$ }
\vcenter{\hbox{$#2#3$ }}\kern-.6\wd0}}

\def\dashint{\Xint-}

\renewcommand{\geq}{\geqslant}

\newcommand{\di}{\operatorname{div}}
\newcommand{\Lip}{\operatorname{Lip}}

\newcommand{\dif}{\operatorname{d}\!}

\newcommand{\spt}{\operatorname{spt}}
\newcommand{\N}{\mathbb{N}}

\newcommand{\R}{\mathbb{R}}

\newcommand{\locc}{\operatorname{loc}}

\newcommand{\dista}{\operatorname{dist}}

\newcommand{\diameter}{\operatorname{diam}}

\newcommand{\ball}{\operatorname{B}}

\newcommand{\sobo}{\operatorname{W}}
\newcommand{\lebe}{\operatorname{L}}

\newcommand{\hold}{\operatorname{C}}

\newcommand{\D}{\operatorname{D}\!}
\newcommand{\curl}{\operatorname{curl}}
\renewcommand{\leq}{\leqslant}

\newcommand{\id}{\operatorname{Id}}

\newcommand{\A}{\mathbb{A}}

\newcommand{\besov}{\operatorname{B}}

\newcommand{\mres}{\!\mathbin{\vrule height 1.6ex depth 0pt width
0.13ex\vrule height 0.13ex depth 0pt width 1.1ex}\!}

\renewcommand{\dashint}{\fint}

\DeclareMathOperator{\diver}{div}
\newcommand{\Acal}{\mathscr{A}}

\newcommand{\Dcal}{\mathcal{D}}

\newcommand{\Ocal}{\mathcal{O}}
\newcommand{\Wcal}{\mathcal{W}}
\newcommand{\dist}{\mathrm{dist}}
\newcommand{\E}{\mathscr{E}}
\newcommand{\Rcal}{\mathcal{R}}
\newcommand{\Scal}{\mathcal{S}}

\renewcommand{\phi}{\varphi}
\newcommand{\dHaus}{~\textup{d}\mathscr{H}}
\newcommand{\dmu}{~\textup{d}\mu}
\newcommand{\dx}{~\textup{d}x}
\newcommand{\dy}{~\textup{d}y}

\DeclareMathOperator{\conv}{conv}

\newcommand{\franz}{\varepsilon_{\Omega}}
\renewcommand{\epsilon}{\varepsilon}
\definecolor{Gump}{rgb}{0,0.6,0.4}
\definecolor{Blue2}{rgb}{0.4,0.4,1}
\definecolor{Red2}{rgb}{1,0.4,0.4}

\usepackage{tikz-cd}
\begin{document}

\title[Divergence free extensions]{Extensions of divergence-free \\ fields in $\lebe^{1}$-based function spaces}

\thanks{}
\author[F. Gmeineder]{Franz Gmeineder}
\address{F.G.: Department of Mathematics and Statistics, University of Konstanz, Universit\"{a}sstra\ss e 10, 78457 Konstanz, Germany. E-mail: \texttt{franz.gmeineder@uni-konstanz.de}}
\author[S.~Schiffer]{Stefan Schiffer}
\address{S.Sc.: Max-Planck-Insitut f\"ur Mathematik in den Naturwissenschaften, Inselstra{\ss}e 22, 04103 Leipzig, Germany. E-mail: \texttt{schiffer@mis.mpg.de}}
\date{\today}

\begin{abstract}
We establish the first extension results for divergence-free (or solenoidal) elements of $\mathrm{L}^{1}$-based function spaces. Here, the key point is to preserve the solenoidality constraint while simultaneously keeping the underlying $\mathrm{L}^{1}$-boundedness. While previous results as in \textsc{Kato} et al. \cite{HKMT} for $\mathrm{L}^{p}$-based function spaces, $1<p<\infty$, rely on PDE approaches, basic principles from harmonic analysis rule out such strategies in the $\mathrm{L}^{1}$-context. By means of a novel method adapted to the divergence-free constraint via differential forms, we establish the existence of such extension operators in the $\mathrm{L}^{1}$-based situation. This applies both to the case of convex domains, where a global extensions can be achieved, as well as to the Lipschitz case, where a local extension can be achieved. Being applicable to $1<p<\infty$ too, our method provides a unifying approach to the cases $p\in\{1,\infty\}$ and $1<p<\infty$. Specifically, covering the exponents $p\in\{1,\infty\}$, this answers a borderline case left open by \textsc{Kato} et al. \cite{HKMT} in the affirmative. By use of explicit examples, the assumptions on the underlying domains are shown to be almost optimal. 
\end{abstract}

\maketitle

\tableofcontents
\section{Introduction}\label{sec:intro}
\subsection{Aims and scope}
A classical task in the theory of function spaces is to extend functions, a priori defined on a proper open subset $\Omega$ of $\R^{n}$, to the entire $\R^{n}$ while preserving certain properties. This allows to view such elements as restrictions of globally defined functions with  specific features, and then opens the gateway to techniques which are confined to the entire space situation.

Whereas classical results on weakly differentiable functions often deal with such properties being smoothness and integrability, much less is known on the preservation of differential constraints. Here, a key operator which arises in numerous applications and which displays the pivotal object in the present paper, is the \emph{divergence} $\di$. Given an open set $\Omega\subset\R^{n}$ of a certain geometry and regularity, and some normed function space $(X(\Omega),\|\cdot\|_{X(\Omega)})$ of vector fields $v\colon\Omega\to\R^{n}$, one might thus be interested in the following question: Does there exist a bounded linear operator $\mathscr{E}\colon X(\Omega)\to X(\R^{n})$ such that 
\begin{align}\label{eq:mainequestion}
\mathrm{div}(v)=0\;\text{in}\;\mathscr{D}'(\Omega)\Longrightarrow \mathrm{div}(\mathscr{E}v)=0\;\text{in}\;\mathscr{D}'(\R^{n})?
\end{align}
If this is possible, we shall refer to $\mathscr{E}$ as a \emph{solenoidality perserving} extension operator on $X(\Omega)$. Divergence-free or solenoidal  fields are ubiquitous in applications: For instance, in fluid mechanics they model incompressible fluid flows.

It is clear that, depending on the method (such as, e.g., PDE-based approaches to be revisited in Section \ref{sec:context} below), only certain function spaces can be tackled successfully. Specifically, the only known results are confined to $\lebe^{p}$-based function spaces for $1<p<\infty$, having left the borderline case $p=1$ open so far. In the present paper, we aim to close this gap and provide a satisfactory answer to the above extension problem for $\lebe^{1}$-based function spaces $X$. Our approach, which also applies to the case of $\lebe^{p}$-based function spaces for $1<p<\infty$, thus allows for a unifying treatment of the above extension problem. Before we embark on the precise description of our results in Section \ref{thm:mainthms}, we briefly pause to review the results available so far.

\subsection{Context and previous results}\label{sec:context}
In order to contextualise our results, we begin by revisiting the extension result \cite{HKMT} by \textsc{Kato} et al.  for the exponent range $1<p<\infty$. 
 
\textbullet \; Let $\Omega,\Omega'\subset\R^{n}$ be two bounded, convex domains with smooth boundaries such that  $\overline{\Omega}\subset\Omega'$ and $\Omega' \setminus \Omega$ is connected. In order to extend a distributionally solenoidal field $u\in\lebe^{p}(\Omega;\R^{n})$  
to $\Omega'$ while preserving solenoidality in $\mathscr{D}'(\Omega')$, one considers the Neumann problem for $(-\Delta)$:
\begin{align}\label{eq:neumann}
\begin{cases}
-\Delta v = 0&\;\text{in}\;\Omega'\setminus\overline{\Omega},\\ 
\partial_{\nu}v = u|_{\partial\Omega}\cdot\nu&\;\text{on}\;\partial\Omega, \\
\partial_{\nu}v =  0&\;\text{on}\;\partial\Omega'.
\end{cases}
\end{align}
Here, $\nu$ is the outer unit normal to $\partial\Omega$ and $u|_{\partial\Omega}$ is understood in the usual  boundary trace sense. Then one then has that $u|_{\partial\Omega}\cdot\nu\in\mathrm{B}_{p,p}^{-1/p}(\partial\Omega)$, and \eqref{eq:neumann} can be solved in $\sobo^{1,p}(\Omega' \setminus\Omega;\R^{n})$ (see, e.g., \textsc{Triebel} \cite[Thm. 4.3.3]{Triebel}). One then defines
\begin{align}\label{eq:KatoExt}
\overline{u} := \begin{cases} 
u&\;\text{in}\;\Omega,\\ 
\nabla v &\;\text{in}\;\Omega'\setminus\overline{\Omega},\\ 
 0&\;\text{otherwise},
\end{cases}
\end{align}
whereby $\overline{u}$ still satisfies $\mathrm{div}(\overline{u})=0$ in $\mathscr{D}'(\R^{n})$ and is compactly supported. The corresponding operator $\mathscr{E}\colon u \mapsto \overline{u}$ then boundedly maps
\begin{align*}
\lebe_{\di}^{p}(\Omega;\R^{n}):=\{u\in\lebe^{p}(\Omega;\R^{n})\colon\;\mathrm{div}(u)=0\;\text{in}\;\mathscr{D}'(\Omega)\}\to \lebe_{\di}^{p}(\Omega';\R^{n}). 
\end{align*}
An analogous strategy works for Sobolev spaces $\sobo^{m,p}(\Omega;\R^{n})$ with $m\in\mathbb{N}$, see \cite{HKMT}. 

\textbullet \; Even though it is, to the best of our knowledge, not written down explicitly in the literature, one may construct an extension of divergence-free maps by use of potentials. For simplicity, we assume that $n=3$ and that  $\Omega\subset\R^{n}$ is open and convex. Moreover, let  $u\in\lebe^{p}(\Omega;\R^{3})$ satisfy $\mathrm{div}(u)=0$ in $\mathscr{D}'(\Omega)$. In this situation, one considers the so-called \emph{regularised Poincar\'{e}-type operator} $R_{x_{0}}$. For $x_{0}\in\Omega$, this is an averaged version of the classical potential operator 
\begin{align}
R(x)= -(x-x_{0})\times \int_{0}^{1}t\,u(x_{0}+t(x-x_{0}))\dif t,\qquad x\in\Omega.
\end{align}
As established by \textsc{Costabel \& McIntosh} \cite{Costabel}, $R_{x_{0}}\colon\sobo^{m,p}(\Omega;\R^{3})\to\sobo^{m+1,p}(\Omega;\R^{3})$ boundedly (where $m\in\mathbb{N}_{0}$, $1<p<\infty$ and $\sobo^{0,p}:=\lebe^{p}$) and $u=\curl(R_{x_{0}}u)$. The highly non-trivial point here is that $R_{x_{0}}$ in fact yields a $\sobo^{m+1,p}$-solution of the heavily underdetermined curl equation. This gives  us access to extension operators in Sobolev spaces: Letting $E$ be a bounded linear extension operator $E\colon\sobo^{m+1,p}(\Omega;\R^{3})\to\sobo^{m+1,p}(\R^{3};\R^{3})$, one then may put $\mathscr{E}u:=\curl(ER_{x_{0}}u)$. 

The reader might notice that both extension operators are a priori only defined on divergence-free fields. In \eqref{eq:neumann}ff., this is reflected by the $\besov_{p,p}^{-1/p}$-regularity of $\partial_{\nu}v$. This regularity assertion  requires $\mathrm{div}(u)=0$ due to Stokes' theorem. Similarly, in the second approach, we can only assert that $u=\curl(R_{x_{0}}u)$ provided $\mathrm{div}(u)=0$. If $\partial\Omega$ moreover is sufficiently regular, one might perform a Helmholtz decomposition of $\lebe^{p}(\Omega;\R^{n})$ and separately extend the divergence- and curl-free parts. In this way, we obtain an affirmative answer to \eqref{eq:mainequestion}. 

In addition, the topological assumptions on $\Omega$ can be relaxed for both approaches. However, a solenodaility preserving extension to the entire $\R^{n}$ is not even possible for general, smoothly bounded domains. For instance, if $\R^{n}\setminus\Omega$ has multiple connected components, \eqref{eq:neumann} needs not be solvable. This is a consequence of Gauss' or Stokes' theorem, respectively. In particular, one can only come up with an extension to a neighbourhood of $\partial\Omega$ or pass to a suitable quotient space -- see \cite{HKMT} and Section \ref{sec:topcha} below. 
\smallskip

While such topological obstructions for extensions are inherent in the operator $\mathrm{div}$ through Stokes' theorem, the available approaches \textbf{do not apply to the endpoint cases $p=1$ or $p=\infty$}. For instance, it is well-known that the trace operator $\mathrm{tr}$ maps $\sobo^{1,1}(\Omega;\R^{n})$ surjectively to $\lebe^{1}(\partial\Omega;\R^{n})$, and that the Neumann problem for the Laplacian with $\lebe^{1}$-boundary data is ill-posed. Hence it is difficult, if not impossible, to employ PDE-based extension operators such as \eqref{eq:neumann}, \eqref{eq:KatoExt} in the $\lebe^{1}$-setting. On the other hand, approaches based on the explicit construction of potentials such as the regularised Poincar\'{e}-type operator, equally do not generalise to $p=1$. Here, the reason is that $R_{x_{0}}$ is a weakly singular integral operator with kernel $k(x,y)$ roughly of size $|k(x,y)|\sim |x-y|^{1-n}$, and does not map $\sobo^{m,1}(\Omega;\R^{n})\to\sobo^{m+1,1}(\Omega;\R^{n})$ boundedly. 

Yet, for geometrically very simple domains such as cubes $Q$, it is easy to obtain divergence-free extensions at least into a neighbourhood of $Q$:  
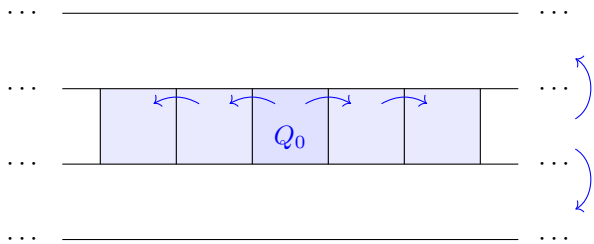
\begin{figure}[t]
\begin{tikzpicture}
\draw[-] (-3,1) -- (3,1);
\draw[-] (-3,0) -- (3,0);
\draw[-] (-3,-1) -- (3,-1);
\draw[-] (-3,-2) -- (3,-2);
\draw[-] (-2.5,0) -- (-2.5,-1);
\draw[-] (-1.5,0) -- (-1.5,-1);
\draw[-] (-0.5,0) -- (-0.5,-1);
\draw[-] (0.5,0) -- (0.5,-1);
\draw[-] (1.5,0) -- (1.5,-1);
\draw[-] (2.5,0) -- (2.5,-1);
\draw[fill=blue!60!white, opacity=0.2] (-0.5,-1) -- (0.5,-1) -- (0.5,0) -- (-0.5,0) -- (-0.5,-1);
\node[blue] at (0,-0.65) {$Q_{0}$};
\draw[fill=blue!40!white, opacity=0.2] (-1.5,-1) -- (-0.5,-1) -- (-0.5,0) -- (-1.5,0) -- (-1.5,-1);
\draw[fill=blue!40!white, opacity=0.2] (-2.5,-1) -- (-1.5,-1) -- (-1.5,0) -- (-2.5,0) -- (-2.5,-1);
\draw[fill=blue!40!white, opacity=0.2] (0.5,-1) -- (1.5,-1) -- (1.5,0) -- (0.5,0) -- (0.5,-1);
\draw[fill=blue!40!white, opacity=0.2] (1.5,-1) -- (2.5,-1) -- (2.5,0) -- (1.5,0) -- (1.5,-1);
\draw[->,blue] (0.2,-0.2) [out=30, in = 150] to (0.8,-0.2);
\draw[->,blue] (1.2,-0.2) [out=30, in = 150] to (1.8,-0.2);
\draw[->,blue] (-0.2,-0.2) [out=150, in = 30] to (-0.8,-0.2);
\draw[->,blue] (-1.2,-0.2) [out=150, in = 30] to (-1.8,-0.2);
\node at (3.5,1) {$\cdots$};
\node at (3.5,0) {$\cdots$};
\node at (3.5,-1) {$\cdots$};
\node at (3.5,-2) {$\cdots$};
\node at (-3.5,1) {$\cdots$};
\node at (-3.5,0) {$\cdots$};
\node at (-3.5,-1) {$\cdots$};
\node at (-3.5,-2) {$\cdots$};
\draw[->,blue] (3.75,-0.4) [out=30, in =330] to (3.75,0.4);
\draw[->,blue] (3.75,-0.8) [out=330, in =30] to (3.75,-1.6);
\end{tikzpicture}
\caption{Extensions by reflections in the flat case.}\label{fig:reflect}
\end{figure}

\begin{example}[Cubes]\label{ex:cube}
Consider the unit cube $Q_{0}=(0,1)^n$. We can extend a divergence-free function $u \colon Q \to \R^n$ as follows: First, we define an extension to $(0,2) \times (0,1)^{n-1}$ for $x_1 \in (1,2)$ and $x'\in(0,1)^{n-1}$
\[
u_1(x_1,x') = u_1(2-x_i),~u_i(x_1,x') = - u_i(2-x_1,x'), \quad i\in\{2,\ldots,n\}.
\]
A successive application of this procedure yields a divergence-free function on the strip $\R\times (0,1)^{n-1}$. Performing similar reflections in the remaining coordinate directions then yields an extension of $u$ to full space, see Figure \ref{fig:reflect}. However, proceeding in this way only gives an extension $\lebe^{1}(Q_{0};\R^{n})\to\lebe_{\locc}^{1}(\R^{n};\R^{n})$ which preserves solenodaility; note that localisation by use of a smooth cut-off destroys the divergence constraint.
\end{example}
Approaches as in \eqref{ex:cube} only work for flat boundaries, and reductions to the flat case seem difficult if not impossible either. Indeed,  locally flattening a curved boundary by nonlinear maps $\Phi_{j}$ equally destroys the requisite solenoidality of the 'flattened' maps $u\circ\Phi_{j}$. 

We shall also shortly mention that solving a boundary value problem \begin{equation}\label{intro:BoundaryValue}
\begin{cases}
\diver u = 0&\;\text{in}\;\R^{n}\setminus\overline{\Omega},\\ 
u \cdot \nu = g &\;\text{on}\;\partial\Omega, 
\end{cases}
\end{equation}
directly \emph{without} taking the detour via \eqref{eq:neumann}, this matter basically is as difficult as our extension: First we need to identify a suitable trace space and then also solve the equation. While a theory via divergence-measure fields seems to be quite promising for $p=\infty$, \cite{PT,CLT}, the case $p=1$ is, to the best of the authors' knowledge, relatively unventured.
We shall also mention that present approaches via solving boundary value problems or reflection do not improve its regularity anywhere: At each point $y \in Q^{\complement}$ we only get $u \in \lebe^1(B_{\varepsilon}(y))$ and nothing better. This is a contrast to the extension we construct: Away from the boundary, our extension is smooth.

While seemingly being only a technical, functional analytic result, the result comes in quite handy in problems dealing with incompressible fluids. In particular, the extension result of \cite{HKMT} is for instance applied to studying incompressible flow on domains, e.g. \cite{MMW,GKM,DRLS}, when divergence-free functions \emph{without} zero boundary conditions appear; we also refer to \cite{Kromer,Kreisbeck} for some applications that are not directly connected to fluids. We give some further applications of our extension result, that are inspired by \cite{BVS,GS21,Schiffer2} in Section \ref{sec:applications} and especially specific to the borderline case of $\lebe^1$-extensions.

Finally, the treatment of the constraint $\diver u=0$, albeit being one of the physically most relevant, is only the tip of the iceberg. Indeed, one may replace the constraint $\diver u=0$ by a general differential constraint $\A u=0$. While only investigating the divergence as a differential operator in the present paper, we shortly comment further on this general case in the final Section \ref{sec:applications}.

\section{Main results and strategy of proof}
\subsection{First main theorem: $\lebe^1$-extension}\label{thm:mainthms}
We now come to the description of our main results. Similarly as in \textsc{Kato} et al. \cite{HKMT}, we distinguish between the low regularity context of extending divergence-free $\lebe^{1}$-fields and the higher regularity context of extending divergence-free $\sobo^{m,1}$-fields. Section \ref{sec:strategy} will then provide a discussion of the underlying strategy of proof, and in how far it differs from that of \cite{HKMT}. For all of the following, we denote for a given open subset $\Omega\subset\R^{n}$ and a normed function space $(X(\Omega),\|\cdot\|_{X(\Omega)})$ with $X(\Omega)\subset\lebe_{\locc}^{1}(\Omega;\R^{n})$ 
\begin{align}
X_{\diver}(\Omega):= \{u\in X(\Omega)\colon\;\mathrm{div}(u)=0\;\;\text{in}\;\;\mathscr{D}'(\Omega)\}, 
\end{align}
where $\mathscr{D}'(\Omega)$ denotes the space of distributions on $\Omega$ as usual.
The main theorem regarding the lower regularity case is as follows.
\begin{theorem}[$\lebe^{1}$-case]\label{thm:W1}
Let $\Omega \subset \R^n$ be a bounded Lipschitz domain. Then there exists $\Omega'$ with $\Omega \Subset \Omega'$ and a bounded and linear extension operator $\E_{\lebe^1} \colon \lebe^1(\Omega;\R^n) \to \lebe^1(\Omega';\R^n)$ such that $u \in \lebe^1_{\diver}(\Omega) \Rightarrow \E_{\lebe^1}u \in \lebe^1_{\diver}(\Omega')$.
\end{theorem} 
The existence of such a $\Omega'$ is attributed to the topological challenges, cf. Example \ref{ex:topology} below. We may also formulate a global version of Theorem \ref{thm:W1} on a suitable quotient space, Theorem \ref{thm:global:extension}. 

Theorem \ref{thm:W1} will be established in Sections \ref{sec:L1} and \ref{sec:Lip}. Specifically, Section \ref{sec:L1} deals with the case of \emph{convex} sets $\Omega\subset\R^{n}$, whereas the general case of Lipschitz domains is the subject of Section \ref{sec:Lip}. The key reason for this distinction is two-fold, and is due to both our method, topological obstructions and their interplay to be described next. 

First, from a methodological perspective, the extension operator underlying Theorem \ref{thm:W1} hinges on nonlinear reflection procedure inspired by that of \textsc{Jones} \cite{Jones}. However, in order to map divergence-free fields to divergence-free fields, our modified construction works by reducing the properties of the extension operator to simplices inside $\Omega$. For the latter, we may consequently employ Stokes' theorem in order to obtain the requisite mapping property on divergence-free fields, also cf. \cite{Schiffer,BGS}. A direct approach of reflecting an $\lebe^p$ vector field and then correcting (as we do in the $\sobo^{1,p}$-case) probably does not work, also cf. Remark \ref{rem:ontheapproach}.

\subsubsection{Geometric challenges}
If $\Omega$ is \emph{convex}, flat simplices with vertices in $\Omega$ are always contained in $\Omega$ and are therefore well-defined. This is not necessarily so if $\Omega$ has only a Lipschitz boundary, and in this case, more care is needed to define the underlying simplices. In order to focus on the key construction, which is most transparent for \emph{convex} domains, we thus give in Section \ref{sec:L1} the detailed proof of Theorem \ref{thm:W1} for this class of domains. Here, it is most convenient to reformulate the above extension problem in the language of differential forms. Based on a construction strongly inspired by Stokes' theorem, we then first introduce the requisite extension operator on sufficiently smooth maps, cf. Definition \ref{def:extension:l1}. The general passage to the low regularity regime of $\lebe^{1}$-maps is then accomplished by a density result. This density result, in turn, requires the extension operator on smooth functions that will be at our disposal at this stage of the proof. In view of this strategy, the underlying geometric set-up in the convex situation is explained in Section \ref{sec:geomsetupconvex}, and a more detailed outline of the proof in this case is displayed at the beginning of Section \ref{sec:proofL1divcon}. As a key point, the proof to be completed in Section \ref{sec:density}, yields that distributionally divergence-free fields defined on a convex domain can be extended to distributionally divergence-free fields defined on the entire $\R^{n}$, and even compact supports can be achieved for the extensions.

\subsubsection{Topological challenges}\label{sec:topcha}
Second, when aiming for an extension result on Lipschitz domains, we face topological obstructions which are are invisible in the convex case. In general, if $\Omega\subset\R^{n}$ merely has Lipschitz boundary, we can only expect to extend divergence-free fields on $\Omega$ to divergence-free fields on a neighbourhood of $\Omega$. 
Recall the ansatz from \eqref{eq:neumann}. If $\Omega' \setminus\Omega$ has several connected components $A_{1},...,A_{N}$, the system \eqref{eq:neumann} or \eqref{intro:BoundaryValue} is solvable if and only if 
\begin{equation} \label{intro:bdryint}
 \int_{\partial A_i} u|_{\partial\Omega} \cdot \nu \dif\mathscr{H}^{n-1} =0.
 \end{equation}
 This is a consequence of the Gauss theorem and $\mathrm{div}(u)=0$ in $\mathscr{D}'(\Omega)$, and is automatically satisfied if $\R^{n}\setminus\Omega$ has only one connected component. Hence, for domains which are diffeomorphic to a ball, there are no topological obstructions since the system \eqref{eq:neumann} is solvable on the complement. However, for domains with holes (e.g. an annulus) and hereafter a different homology,  \eqref{intro:bdryint} does not hold and hence solenoidal extensions are impossible in general:
 \begin{example}[Topological obstructions] \label{ex:topology}
Consider the annulus $\ball_2(0) \setminus \overline{\ball}_1(0)$, the complement of its closure having the two connected components $A_1 =\ball_1(0)$ and $A_2= \R^{n}\setminus\overline{\ball}_2(0)$. 
Then \eqref{intro:bdryint} is not satisfied for any function, for instance not for $g(x) = \tfrac{x}{\vert x \vert^n}$, which is divergence-free in the annulus. As is pointed out in \cite{HKMT} (also see Theorem \ref{thm:global:extension} below), this function cannot be extended to solenoidal field on $\R^{n}$. Moreover, this type of fields is basically the only example with this property: For instance, in the case of the annulus we can extend any function satisfying \eqref{intro:bdryint} and for any other function $u$ we can find $\lambda \in \R$ such that $u - \lambda g$ obeys \eqref{intro:bdryint}.
In other words, we may find an extension only for elements in a \emph{quotient space}.
\end{example}
This impossibility of a global extension is reflected in our construction as follows: As discussed above, the definition of simplices required for the implementation of Stokes' theorem is non-trivial in the situation of Lipschitz domains. Yet, employing a collar type construction in the sense of algebraic topology or differential geometry (see, e.g. \textsc{Hatcher} \cite{Hatcher} or \textsc{Lee} \cite{Lee}), we are able to introduce the requisite (now in general curvilinear) simplices \emph{close to the boundary $\partial\Omega$}. By analogy with the convex case, the extension operator is constructed by transferring information from the interior to the exterior of $\Omega$ at a distance uniformly proportional to $\partial\Omega$. Since the curvilinear simplices are only well-defined close to the boundary, it is then clear that our extension operator can only extend divergence-free fields in $\Omega$ to divergence-free fields \emph{in a neighbourhood} of $\Omega$. 

Hence, the definition of suitable simplices inside $\Omega$ and close to $\partial\Omega$ appears not only as a technical but conceptual point too. As mentioned above, we here make use of a collar-type construction and a suitable definition of curvilinear simplices on (smooth) embedded manifolds. While such constructions are well-known to hold for smoothly bounded domains, this is not so for Lipschitz domains. To accomplish the passage from smooth to Lipschitz domains, we make use of a remarkable result due to \textsc{Ball \& Zarnescu} \cite{BZ} (cf. Lemma \ref{lemma:BZ}) which asserts that any bounded Lipschitz domain is Bi-Lipschitz homeomorphic to a smoothly bounded domain. This, in turn, leads to a suitable definition of simplices close to the boundary in the Lipschitz case, cf. Section \ref{sec:def:simplex}. Based on an analogy with the easier, yet more transparent case of convex domains displayed in Section \ref{sec:L1}, the slightly more involved simplex construction then allows to establish Theorem \ref{thm:W1} by parallel means in Section \ref{sec:Lip}. 

\subsection{Second Main Theorem: $\sobo^{1,1}$ extension} We now turn to a setting of higher regularity, namely solenoidal fields in Sobolev spaces. For brevity, we restrict our mathematical analysis to the case $m=1$, i.e. $\sobo^{1,1}$-cases. While the strategy for extension differs between $\lebe^1$- and $\sobo^{1,1}$ regularity, more derivatives only introduce some technical intricacies (cf. Section \ref{sec:higherorder}). To keep the presentation concise, we therefore only deal with $\sobo^{1,1}$-case.

\begin{theorem}[$\sobo^{1,1}$-case]\label{thm:L1}
Let $\Omega \subset \R^n$ be a bounded Lipschitz domain. Then there exists $\Omega'$ with $\Omega \Subset \Omega'$ and a bounded and linear extension operator $\E_{\sobo^{1,1}} \colon \sobo^{1,1}(\Omega;\R^n) \to \sobo^{1,1}(\Omega';\R^n)$ such that $u \in \sobo^{1,1}_{\diver}(\Omega) \Rightarrow \E_{\sobo^{1,1}}u \in \sobo^{1,1}_{\diver}(\Omega')$.
\end{theorem} 
While we focuse on the case $p=1$, our theorems also works for any other integrability exponent.
\begin{remark}[$1<p<\infty$]
While the case $p=1$ is the most interesting, we also may show (cf. Corollary \ref{Coro:Lpextension}) that $\E_{\lebe^1}$ and $\E_{\sobo^{1,1}}$ provide an extension operator for $p=\infty$ and therefore extension operators for all $\lebe^p$ and $\sobo^{1,p}$ for all $1\leq p \leq \infty$. Moreover, while the $\lebe^1$-extension and the $\sobo^{1,1}$-extension are different on a conceptual level, the extension constructed for $\sobo^{1,1}$ should, in principle, be generalisable to all $\sobo^{m,1}$ for $m\geq 1$.
\end{remark}

\subsection{On the geometry of the domain}
As alluded to in Section \ref{sec:context}, for exponents $1<p<\infty$, an extension of divergence-free functions for Lipschitz bounded sets is constructed by \cite{HKMT}. For the endpoint cases, we are only aware of a few examples, where a $\lebe^1$-extension can be explicitly constructed, for instance by reflection for rectangular domains (cf. \cite{GS21}). If the boundary is not flat, then it is quite hard to construct an extension that preserves solenoidality (flattening the boundary destroys this condition); for general domains we therefore need a more flexible approach.

Furthermore, we show in Section \ref{sec:regularityofOmega} that the requirement of $\Omega$ having Lipschitz boundary is quite close to the optimum. In particular, we show that domains with exterior and interior cusps- in particular these are $\hold^{\alpha}$-domains- do not always allow for extension operators that preserve solenoidality. 

\noindent Last, we tacitly assume that $\Omega$ is a bounded domain, i.e. it is connected. While for several technical issues, the boundedness assumptions \emph{cannot} be dropped, we may indeed show everything also for open Lipschitz sets with several connected components. As the connected components are well-separated we may simply construct an extension for the connected components independently. 

The following theorem outlines some geometrical issues of the extension. It is by no means optimal (i.e. with a more elaborate construction we can find better H\"older exponents), but still shows what goes wrong.
\begin{theorem}[Geometric limitations] \label{thm:geometriclim}
Let $1\leq p  \leq \infty$ and let $\gamma < \gamma(p)<1$. Then there is a domain $\Omega$ that is homeomorphic to a ball with boundary of class $\hold^{\gamma}$, such that there is an infinite dimensional vector space of functions $u \in \lebe^p(\Omega;\R^n)$ such that
\begin{enumerate}
    \item $\diver u=0$ in $\mathcal{D}'(\Omega)$;
    \item for any $\Omega'$ with $\Omega \subset \subset \Omega'$ there is \textbf{no} $\tilde{u} \in \lebe^p(\Omega';\R^n)$ with $u = \tilde{u}$ in $\Omega$ and $\diver u=0$ in $\mathcal{D}'(\Omega)$.
\end{enumerate}
\end{theorem}
We highlight that Theorem \ref{thm:geometriclim} is formulated for the low regularity case only- for solenoidal $\sobo^{1,p}$ it is clear that $\Omega$ must be regular enough to allow an unconstrained extension (i.e. Jones' extension), which directly excludes $\hold^{\alpha}$ domains. It is, however, not at all clear whether the same restrictions hold for the constrained and the unconstrained extension.

\subsection{Strategy of proof}\label{sec:strategy}
In view of divergence-free extensions on $\lebe^{1}$-based function spaces, we now explain the strategy employed below in detail. As will be visible from the discussion, here it is necessary to distinguish between Sobolev regularity, $\sobo^{1,1}$, and $\lebe^1$. As the strategy for higher regularity is more straightforward, we begin with the $\sobo^{1,1}$-case.
\subsubsection{Sobolev regularity}\label{sec:W1ext}

Here we assume that $u\in\sobo^{1,1}(\Omega;\R^{n})$ satisfies $\mathrm{div}(u)=0$ $\mathscr{L}^{n}$-a.e. in $\Omega$. In this case, there exists a bounded linear extension operator $\mathscr{E}_{0}\colon \sobo^{1,1}(\Omega;\R^{n})\to\sobo^{1,1}(\R^{n};\R^{n})$. In general, $\mathscr{E}_{0}$ cannot be assumed to map divergence-free maps to divergence-free maps, and so we introduce a corrector $\mathscr{R}$ with the following properties: 
\begin{enumerate} 
\item\label{item:strat1} $\mathscr{R}\colon\sobo^{1,1}(\Omega;\R^{n})\to(\hold^{\infty}\cap\sobo_{0}^{1,1})(\R^{n}\setminus\overline{\Omega};\R^{n})$ is a bounded linear operator with respect to the $\sobo^{1,1}$-norm: There exists $c>0$ such that we have 
\begin{align*}
\|\mathscr{R}u\|_{\sobo^{1,1}(\R^{n}\setminus\overline{\Omega})}\leq c\|u\|_{\sobo^{1,1}(\Omega)}\qquad\text{for all}\;u\in\sobo^{1,1}(\Omega;\R^{n}). 
\end{align*}
\item\label{item:strat2} Whenever $u\in\sobo^{1,1}(\Omega;\R^{n})$ satisfies $\mathrm{div}(u)=0$ $\mathscr{L}^{n}$-a.e. in $\Omega$, then for $\mathscr{L}^{n}$-a.e. $x\in\R^{n}\setminus\overline{\Omega}$ we have 
\begin{align*}
\mathrm{div}(\mathscr{R}u)(x)=-\mathrm{div}(\mathscr{E}_{0}u)(x).
\end{align*}
\end{enumerate} 
Based on these properties, we then define 
\begin{align*}
\mathscr{E}:=\mathscr{E}_{0}+\mathscr{R}. 
\end{align*}
Property \ref{item:strat1} then implies that $\mathscr{E}$ maps $\sobo^{1,1}(\Omega;\R^{n})\to\sobo^{1,1}(\R^{n};\R^{n})$ boundedly together with $\mathscr{E}u|_{\Omega}=u$. In particular, $\mathrm{div}(\mathscr{E}u)\in\lebe^{1}(\R^{n})$ whenever $u\in\sobo^{1,1}(\Omega;\R^{n})$, and property \ref{item:strat2} then yields the implication 
\begin{align*}
\mathrm{div}(u)=0\;\text{$\mathscr{L}^{n}$-a.e. in $\Omega$} \quad \Longrightarrow 
\quad \mathrm{div}(\mathscr{E}u)=0\;\text{in}\;\lebe^{1}(\R^{n}).
\end{align*}
The strategy to obtain extensions in even higher regularity is the same. In particular, one may introduce an extension operator $\E_0$ on $\sobo^{m,1}$ and also add corrector terms to preserve solenoidality, cf. Section \ref{sec:higherorder}

\subsubsection{Lebesgue spaces}\label{sec:L1ext} 
It is clear that, when lowering the regularity from $\sobo^{m,1}(\Omega;\R^{n})$ to $\lebe^{1}(\Omega;\R^{n})$, it is not possible to use an analogous approach as for the case $m \geq 1$ described above. To explain this point in detail, we start by choosing a bounded linear extension operator $\mathscr{E}_{0}\colon\lebe^{1}(\Omega;\R^{n})\to\lebe^{1}(\R^{n};\R^{n})$. At this stage, $\mathscr{E}_{0}$ might be taken to be the trivial extension (that is zero on $\Omega^{\complement}$).

Systematically replacing $\sobo^{1,1}$ or $\sobo_{0}^{1,1}$ in \ref{item:strat1}, there exists a linear extension operator $\widetilde{\mathscr{R}}\colon\lebe^{1}(\Omega;\R^{n})\to(\hold^{\infty}\cap\lebe^{1})(\R^{n}\setminus\overline{\Omega};\R^{n})$ which is bounded with respect to the $\lebe^{1}$-norm. Such an extension operator cannot be obtained by extending a given function $u\in\lebe^{1}(\Omega;\R^{n})$ trivially to the entire $\R^{n}$, but for instance by \textsc{Jones}' reflection method. However, the corresponding analogue of \ref{item:strat2}, namely 
\begin{itemize}
\item[(b')] Whenever $u\in\lebe^{1}(\Omega;\R^{n})$ satisfies $\mathrm{div}(u)=0$ in $\mathscr{D}'(\Omega)$, then we have that the identity
\begin{align*}
\mathrm{div}(\mathscr{R}u)=-\mathrm{div}(\mathscr{E}_{0}u)\qquad\text{in}\;\mathscr{D}'(\R^{n}\setminus\overline{\Omega})
\end{align*}
\end{itemize}
is not useful anymore: If we analogously introduce $\mathscr{E}=\mathscr{E}_{0}+\mathscr{R}$, we do not necessarily have $\mathrm{div}(\mathscr{E}u)\in\lebe^{1}(\R^{n})$. Hence, $\mathrm{div}(u)=0$ in $\mathscr{D}'(\Omega)$ and $\mathrm{div}(\mathscr{E}u)=0$ in $\mathscr{D}'(\R^{n}\setminus\overline{\Omega})$  do \emph{not} combine to $\mathrm{div}(\mathscr{E}u)=0$ in $\mathscr{D}'(\R^{n})$. In particular, $\diver \mathscr{E}u$ might, for instance, be a measure that is supported on $\partial \Omega$. Moreover, as only the \emph{normal trace} is defined in some negative Sobolev space, verifying a condition that $u$ and $\E u$ coincide on the boundary appears to be very challenging.

In order to overcome these issues, we will directly introduce a divergence-free $\lebe^{1}$-extension operator \\  $\mathscr{F}\colon \lebe^1(\Omega;\R^n) \to \lebe^1(\R^n;\R^n)$ with the following properties:
\begin{enumerate}
    \item\label{item:strat3} $\mathscr{F}\colon\lebe^{1}(\Omega;\R^{n})\to \lebe^{1}(\Omega;\R^{n})$ is a bounded linear operator, and we have 
\begin{align*}    
    \mathscr{F}u|_{\R^{n}\setminus\overline{\Omega}}\in (\lebe^{1}\cap \hold^{\infty})(\R^{n}\setminus\overline{\Omega};\R^n)\qquad\text{for all}\;u\in\lebe^{1}(\Omega;\R^{n}).
    \end{align*}
    \item\label{item:strat4} Whenever $u\in\lebe^{1}(\Omega;\R^{n})$ satisfies $\mathrm{div}(u)=0$ in $\mathscr{D}'(\Omega)$, then we have 
\begin{align*}
\mathrm{div}(\mathscr{F}u)=0\qquad\text{pointwisely in $\R^{n}\setminus\overline{\Omega}$}.
\end{align*}
    \item \label{crucialstep} Whenever $u\in\lebe^{1}(\Omega;\R^{n})$ satisfies $\mathrm{div}(u)=0$ in $\mathscr{D}'(\Omega)$, then $\diver(\mathscr{F} u)\in \mathscr{D}'(\R^n)$ is actually in $\lebe^{1}$, that is, 
\begin{align*}    
\mathrm{div}(u)=0\;\text{in}\;\mathscr{D}'(\Omega) \Longrightarrow     \diver(\mathscr{F}u)\in \lebe^1(\R^{n}).
\end{align*}
\end{enumerate}
In proving \ref{item:strat3}--\ref{item:strat4}, we first construct  the corresponding extension operator acting on vector fields $u\in \hold^{1}(\overline{\Omega};\R^{n})\cap\lebe^{1}(\Omega;\R^{n})$. Since all estimates for $\mathscr{F}u$ will only involve the $\lebe^{1}$-norms of the fields $u$, the requisite extension operator  $\mathscr{F}\colon\lebe^{1}(\Omega;\R^{n})\to\lebe^{1}(\R^{n};\R^{n})$ is then obtained by density. This step particularly requires density of $\hold_{\mathrm{div}}^{1}(\overline{\Omega};\R^{n})$ in $\lebe_{\mathrm{div}}^{1}(\Omega;\R^{n})$ with respect to the $\lebe^{1}$-norm. 

This density result needs to be approached with care: First, at this stage of the proof we cannot simply employ the corresponding extension operator $\lebe_{\mathrm{div}}^{1}(\Omega;\R^{n})\to\lebe_{\mathrm{div}}^{1}(\R^{n};\R^{n})$ and mollify. Namely, it is precisely this extension operator which we aim to construct, and which is \emph{not} yet at our disposal. Second, approaches which make use of localisation, subsequent mollification and finally patching together the single pieces (as is customary e.g. in proving density of $\hold^{1}(\overline{\Omega})$ in $\sobo^{1,p}(\Omega)$ for $1\leq p<\infty$, cf. \cite[Chpt. 5.3.3]{Evans}) are equally difficult to be implemented in the present setting. Here, the underlying obstruction is that, even though the single localised pieces are divergence-free and smooth up to the boundary, patching them together by means of a smooth partition of unity destroys solenoidality. As described in detail in Section \ref{sec:density}, it is here that we utilise the divergence-free extension operator for $\hold^{1}(\overline{\Omega};\R^{n})$ which at this stage \emph{is} available. 
\subsubsection{On the Lipschitz regularity of the boundary}
While it is imaginable for an extension operator to exist for a more general class of domains, we only focus on Lipschitz domains. The reason is two-fold. On the one hand, as the reader will notice, the approach underlying the proofs of Theorem \ref{thm:L1} and \ref{thm:W1} makes use of decomposition techniques available for $(\varepsilon,\infty)$-domains, hence we need to assume at least this regularity. On the other hand, to define the requisite extension operator for smooth fields we come up suitable definition of smooth (curvilinear) simplices on the boundary $\partial\Omega$ via an identification of the Lipschitz set with a set with smooth boundary, cf. \textsc{Ball \& Zarnescu}, \cite{BZ}. We refer to Section \ref{sec:simplices} for more motivation. Combining both these steps, one may write down some abstract regularity assumption replacing Lipschitz regularity- it is, however, not so clear whether this is really a much tighter assumption that Lipschitz regularity. As it becomes apparent by Theorem \ref{thm:geometriclim}/Theorem \ref{thm:counterexamples}, the assumption of Lischitz regularity is reasonably sharp.

\subsection{Structure of the paper}
We now briefly comment on the organisation of the remainder of the paper. After fixing notation and gathering auxiliary tools such as Whitney coverings in Section \ref{sec:prelims}, we address in Section \ref{sec:simplices} the definition of simplices as required for the divergence-free extension for Lipschitz domains later on. Section \ref{sec:L1} then is devoted to the extension operator for convex domains. Working from the construction in this case, we then give the proof of the extension result for Lipschitz domains in Section \ref{sec:Lip}.  In Section \ref{sec:sobolev}, we address the higher regularity regime and deal with extensions of divergence-free Sobolev functions. Section \ref{sec:further} is concerned with lowering the Lipschitz regularity of the boundaries and comparing the underlying findings with the classical results in the Sobolev context. The paper is concluded with a discussion of various applications in Section \ref{sec:applications}.
{\small  
\subsection*{Acknowledgment} 
The authors are thankful to Oliver C. Schn\"{u}rer for discussions on collar-type constructions. The first author gratefully acknowledges financial support by the Hector foundation. The second author is thankful to the Department of Mathematics and Statistics at the University of Konstanz for financial support and the kind hospitality during a research stay in May 2023, where parts of the paper were concluded. 
}
\section{Preliminaries}\label{sec:prelims}
In this section, we fix notation, gather background results on Whitney coverings and collect additional auxiliary tools which will prove useful in the main part of the paper.
\subsection{General notation}
Throughout, $\mathscr{L}^{n}$ and $\mathscr{H}^{k}$ with $k\in\{0,...,n\}$ denote the $n$-dimensional Lebesgue and the $k$-dimensional Hausdorff measures, respectively. For a Radon measure $\mu$ on $\R^{n}$, a $\mu$-measurable set $U\subset\R^{n}$ with $\mu(U)\in (0,\infty)$ and a $\mu$-measurable function $f\colon U\to V$ with some finite dimensional inner product space $V$, we employ the average integral notation 
\begin{align}\label{eq:dashedintegral}
\dashint_{U}f\dif \mu := \frac{1}{\mu(U)}\int_{U}f\dif\mu.
\end{align}
If $\mu(U)=0$, we define the left-hand side of \eqref{eq:dashedintegral} to be zero. Moreover, given a $\mu$-measurable set $A\subset\R^{n}$, we denote by $\mu\mres A$ the restriction of $\mu$ to $A$. 

For clarity of the integration variable and the measure, even when integrating over differential forms, we still write $\dif x$ or $\dif \mu(x)$. 

For one or more vectors $v_1,...,v_r$ in some inner product space $V$ we denote by $\langle v_1,\ldots,v_r \rangle$ the span of those vectors.
\smallskip

As usual, all distances are taken with respect to the Euclidean norm. For a set $U\subset\R^{n}$, we define the \emph{signed distance function} which is negative on the complement of $U$ by 
\begin{align}\label{eq:signeddistance}
d(x,\partial U):=\begin{cases} 
-\mathrm{dist}(x,\partial U)&\;\text{if}\;x\in \R^{n}\setminus U,\\ 
\mathrm{dist}(x,\partial U)&\;\text{if}\;x\in U.
\end{cases} 
\end{align}
If $Q\subset\R^{n}$ is a cube, we denote by $\ell(Q)$ its side length. Given $\lambda>0$, we define $\lambda Q$ as the cube with the same center and orientation as $Q$ but $\lambda$-times its length. For a given set $U\subset\R^{n}$, we moreover denote by $\conv(U)$ the convex hull of $U$. Given points $x_{1},...,x_{k}\in \R^{n}$, we then simply write $\conv(x_{1},...,x_{k}):=\conv(\{x_{1},...,x_{k}\})$. Lastly, given $x_{0}\in\R^{n}$ and $r>0$, we denote by $\ball_{r}(x_{0}):=\{x\in\R^{n}\colon\;|x-x_{0}|<r\}$ the Euclidean ball with radius $r$ centered at $x$ and, if $U\subset\R^{n}$ is a set, $\ball_{r}(U):=\{x\in\R^{n}\colon\;\dista(x,U)<r\}$.

In order to formulate several of the key estimates in the main part in a coordinate-free manner, it is useful to employ differential forms. Here we formally distinguish between $\R^{n}$ and its dual $(\R^{n})^{*}$. To fix notation, we denote by $\wedgeq^{k}(\R^{n})$ and $\wedgeq^{k}(\R^{n})^{*}$ the exterior product of $k$-copies of $\R^n$, and $(\R^n)^{\ast}$, respectively. As a consequence, $\wedgeq^{k}(\R^{n})^{*}$ corresponds to the set of $k$-multilinear, anti-symmetric functions  $(\R^{n})^{k}\to \R$. Identifying a vector field $v=(v_{1},...,v_{n})\colon \Omega\to\R^{n}$ with the corresponding $(n-1)$-form 
\begin{align}\label{eq:differentialconvention1}
\omega_{v}:=\sum_{j=1}^{n}(-1)^{j-1}v_{j}(\dif x_{1}\wedge ... \dif x_{j-1}\wedge \dif x_{j+1}\wedge ... \wedge \dif x_{n}), 
\end{align}
the exterior derivative of $\omega_{v}$ then satisfies 
\begin{align}\label{eq:differentialconvention2}
\dif\omega_{v} = \mathrm{div}(v)(\dif x_{1}\wedge ... \wedge \dif x_{n})
\end{align}
and hence can be identified with $\mathrm{div}(v)$.
For $1 \leq p \leq \infty$ we further define the function space
\[
\lebe^p_{\diver}(\Omega) := \left\{ u \in \lebe^p(\Omega;\R^n) \colon \diver u=0 \text{ in } \mathcal{D}'(\Omega) \right\}
\]
as $\lebe^p$-functions that have vanishing distributional divergence. In the same fashion for the identified $(n-1)$-forms we can define 
\[
\lebe^p_{\dif}(\Omega) := \left\{ u \in \lebe^p(\Omega;\wedgeq^{k}((\R^{n})^{*})) \colon \dif u=0 \text{ in } \mathcal{D}'(\Omega) \right\}.
\]

Consider a differentiable $(n-1)$-form $u \colon \Omega \to \wedgeq^{k}((\R^{n})^{*})$. Therefore, we may see the full derivative $Du$ as an element of ${\wedgeq}^1((\R^{n})^{*}) \otimes {\wedgeq}^{n-1}((\R^{n})^{*})$.  Consider an element $v \in \wedgeq^{k}(\R^n)$. By $D u \cdot v$ we understand the pairing between the first coordinate of $Du$ with $v$, i.e. for a tensor $\omega_1 \otimes \omega_2 \in \wedgeq^{1}((\R^n)^{\ast}) \otimes \wedgeq^{n-1}((\R^n)^{\ast})$ and $v \in \wedgeq^{k}(\R^n)$ we have
\begin{equation} \label{eq:Skmeaning}
(\omega_1 \otimes \omega_2) \cdot v = (\omega_1(v) \otimes \omega_2) \in \wedgeq^{k-1}(\R^n) \otimes {\wedgeq}^{n-1}(\R^n)^{\ast}.
\end{equation}
Moreover, we have $z \in \R^n = \wedgeq^{1}(\R^n)$, i.e. we may associate
\[
\left[(\omega_1 \otimes \omega_2) \cdot v\right]z \quad \text{with} \quad ( \omega_1(v) \wedge z) \otimes \omega_2 \in \wedgeq^{k}(\R^n) \otimes {\wedgeq}^{n-1}((\R^n)^{\ast})
\]
and then map elements of $ {\wedgeq}^{k}(\R^n) \otimes {\wedgeq}^{n-1}((\R^n)^{\ast})$ 
to elements in ${\wedgeq}^{n-1-k}((\R^n)^{\ast})$ in the natural way, i.e. 
\begin{equation} \label{eq:inner}
\left[(\omega_1 \otimes \omega_2) \cdot v\right]z := \omega_2 \left[\omega_1(v) \wedge z\right] \in  {\wedgeq}^{n-1-k}((\R^n)^{\ast}).
\end{equation}

 \subsection{Whitney and double Whitney coverings}\label{sec:Whitney}
Let $\Omega\subset\R^{n}$ be open. In this situation, $\Omega$ admits a \emph{Whitney covering} (cf. \cite[Chapter 6.1]{Stein}), by which we understand a countable collection $\mathcal{W}=(Q_{j})$ of closed dyadic cubes $Q_{j}\subset\Omega$ with the following properties: 
\begin{enumerate}[label=(W\arabic*)]
    \item\label{item:Whitney1} $\bigcup_{Q\in\mathcal{W}}Q=\Omega$, and for any $Q\in\mathcal{W}$ we have $\frac{7}{6}Q\Subset\Omega$.
    \item\label{item:Whitney2} There exists a number $\mathtt{N}=\mathtt{N}(n)\in\mathbb{N}$ such that for any $Q\in\mathcal{W}$ there exist at most $\mathtt{N}$ cubes $Q'\in\mathcal{W}$ such that $\frac{7}{6}Q\cap \frac{7}{6}Q'\neq\emptyset$. 
    \item\label{item:Whitney3} There exists a constant $c=c(n)>0$ such that we have 
    \begin{align*}
    \frac{1}{c}\ell(Q) \leq \dista(\tfrac{7}{6}Q,\partial\Omega)\leq \dista(Q,\partial\Omega) \leq c\ell(Q)\qquad\text{for all}\;Q\in\mathcal{W}. 
    \end{align*}
\end{enumerate}
\begin{figure}
    \centering \includegraphics[width=0.6 \textwidth]{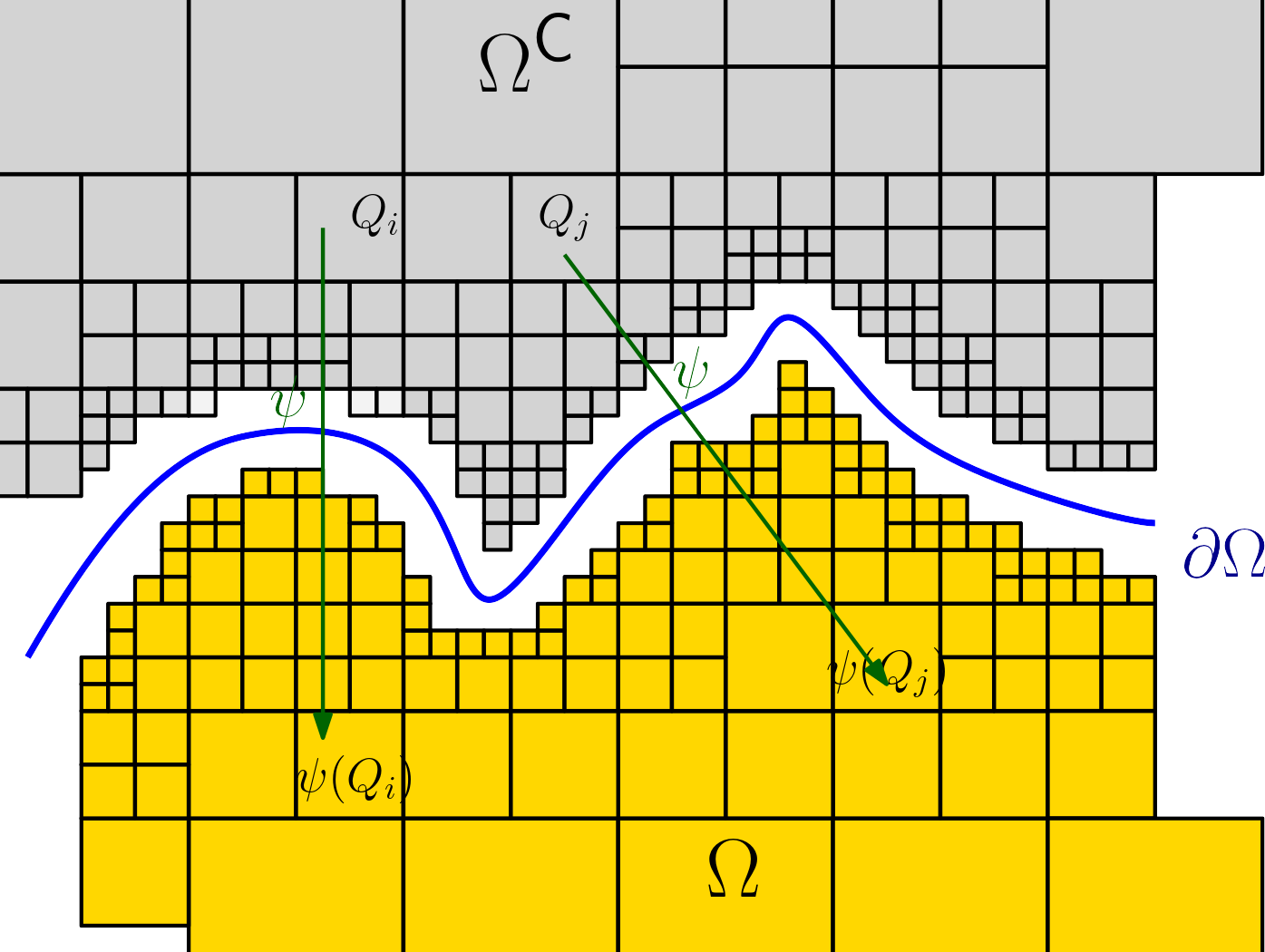}
    \caption{A double Whitney cover and the map $\psi$ between the covers: Even though the cubes $Q_i$ and $Q_j$ are quite close, $\psi(Q_i)$ and $\psi(Q_j)$ are not close together. Indeed, their distance is only controlled in terms of Lemma \ref{cor:Jones}. Observe that if the Lipschitz norm of the graph-like boundary gets large, the distance of a cube $Q_i$ to $\psi(Q_j)$ of similar size also increases.}
    \label{fig:Whitney}
\end{figure}

\textsc{Jones}' construction for extensions of Sobolev functions \cite{Jones} relies on the construction of \emph{interior} and \emph{exterior} Whitney coverings for both $\Omega$ and $\R^{n}\setminus\overline{\Omega}$, respectively, and a suitable identification map between the dyadic cubes in the corresponding coverings.

This is reflected by Lemma \ref{cor:Jones} below, for which we introduce the following notation: For a bounded Lipschitz domain $\Omega$, recall that we can find a finite cover $\mathcal{O}=\{O_{1},...,O_{m}\}$ of $\partial\Omega$ such that, upon rotating and translating, we can write 
\begin{align}\label{eq:nowandthen}
\Omega\cap O_{i} = \{x=(x_{1},...,x_{n})\in O_{i}\colon\; x_{n}<f_{i}(x_{1},...,x_{n-1})\} 
\end{align}
for some Lipschitz maps $f_{i}\colon U_{i}\to\R$ and suitable open subsets $U_{i}\subset\R^{n-1}$, $i\in\{1,...,m\}$. In what follows, we denote by 
\begin{align}\label{eq:nowandthen1}
\mathrm{Lip}(\mathcal{O},\Omega):=\max_{i\in\{1,...,m\}}\mathrm{Lip}(f_{i}) 
\end{align}
the maximal value of the underlying Lipschitz constants. 
\begin{lemma}\label{cor:Jones}
Let $\Omega\subset\R^{n}$ be open and bounded with Lipschitz boundary, and let $\mathcal{W}_{1},\mathcal{W}_{2}$ be Whitney coverings in the sense of \ref{item:Whitney1}--\ref{item:Whitney3} of $\Omega$ and $\R^{n}\setminus\overline{\Omega}$, respectively. 
Then there exists a map $\Psi\colon\mathcal{W}_{2}\to\mathcal{W}_{1}$, a cube $Q_{0}\in\mathcal{W}_{1}$, a threshold number $\eta>0$ and a constant $C>0$ that depends only on $\mathrm{Lip}(\Omega,\mathcal{O})$ such that the following hold for all $Q\in\mathcal{W}_{2}$:
\begin{enumerate}[label=(J\arabic*)]
    \item\label{item:Jonesrefined1} If $\ell(Q) > \eta$, then $\Psi(Q) = Q_0$;
    \item\label{item:Jonesrefined2} $\frac{1}{4}\ell(Q) \leq \ell(\Psi(Q)) \leq 4 \ell(Q) $ for all $Q\in\mathcal{W}_{2}$ with $\ell(Q) \leq \eta$;
    \item \label{item:Jonesrefined3} $\dist(Q,\Psi(Q)) \leq C \ell(Q)$;
    \item \label{item:Jonesrefined3a} $C^{-1} \ell(Q)\leq \dist(\Psi(Q),\partial\Omega) \leq C \ell(Q)$; 
    \item\label{item:Jonesrefined4} $C^{-1}\ell(Q)\leq\dist(Q,\Psi(Q)) \leq C \ell(Q)$ whenever $\ell(Q)\leq\eta$.
\end{enumerate}
\end{lemma}
\begin{remark}\label{remark:Lipschitzepsdel}
   Lemma \ref{cor:Jones} holds true for a wider class of open sets that just those with Lipschitz boundary. Indeed, such a map $\Psi$ exists for so called $(\varepsilon,\delta)$-domains, cf. \cite{Jones}. These are open sets $\Omega\subset\R^{n}$ such that (for $\varepsilon>0$ and $0<\delta \leq \infty$) $x,y\in\Omega$ with $|x-y|<\delta$ there exists a rectifiable curve $\gamma \colon [0,1] \to \Omega$ connecting $x$ to $y$ (i.e. $\gamma(0)=x$, $\gamma(1)=y$) with the following properties:
\begin{enumerate}
\item\label{item:epsdelta1} $\ell(\gamma)<\frac{|x-y|}{\varepsilon}$, where $\ell(\gamma)$ is the length of $\gamma$;
\item\label{item:epsdelta2} $\mathrm{dist}(z,\partial\Omega)|x-y|>\varepsilon |x-z|\,|y-z|$ for all $z\in\gamma([0,1])$. 
\end{enumerate}
For instance, domains with Lipschitz boundaries are of class $(\varepsilon,\infty)$. In contrast, $\hold^{0,\alpha}$-domains for $0<\alpha<1$ are in general not $(\varepsilon,\delta)$ domains: Domains with suitable inward cusps violate \ref{item:epsdelta1} and domains with suitable outward cusps violate \ref{item:epsdelta2} (cf. \cite{PM}, Section \ref{sec:regularityofOmega}). Even though our constructions in the subsequent sections are partially inspired by the method of \textsc{Jones} for $(\varepsilon,\delta)$-domains,
the construction of suitable simplices (cf. Section 
\ref{sec:simplices}) however requires more regularity than merely domains of class $(\varepsilon,\delta)$.

Interestingly, in Section \ref{sec:regularityofOmega} we establish that one of the main objectives of this paper, namely a solenoidal extension operator for $\lebe^p$-functions, \emph{cannot} be achieved for $(\varepsilon,\delta)$ domains in general, and we refer the reader to Section \ref{sec:further} for more on this matter.
\end{remark}
\subsection{Signed distances and normals}
Let $\Omega \subset \R^n$ be a domain with $\hold^k$ boundary. For $x \in \partial \Omega$ we denote by $T_x \partial \Omega$ the tangent space of the manifold $\partial \Omega$. Moroever, by $\mathbf{n}(x)$ 
we denote the \emph{inner} unit normal vector at $x$, so that $T_x \partial \Omega = \langle \mathbf{n}(x) \rangle^{\perp}$. Finally, we denote by $N_x \partial \Omega = \langle \mathbf{n}(x) \rangle$ the normal space.
The following lemma is important for the construction of a collar neighbourhood (cf. Lemma \ref{lemma:collar} below) which, in turn, is crucial for the definition of simplices in domains.
\begin{lemma}[{\cite[Lem. 14.16]{GILTRU77}}]\label{lem:signeddistance}
Let $k\geq 2$ and let $\Omega\subset\R^{n}$ be open and bounded with boundary $\partial\Omega$ of class $\hold^{k}$. Then the signed distance function $x\mapsto d(x,\Omega)$ from \eqref{eq:signeddistance} is of class $\hold^{k}$ in a neighbourhood of $\partial\Omega$. Moreover, at each point $x\in\partial\Omega$, we have $\nabla d(x,\partial\Omega)=-\nu(x)=:\mathbf{n}(x)$ with the outer unit normal $\nu\colon\partial\Omega\to\mathbb{S}^{n-1}$. 
\end{lemma}

\section{Simplices in Lipschitz domains} \label{sec:simplices} 
The divergence-free extension operator to be introduced in Sections \ref{sec:L1} and \ref{sec:Lip} will crucially rely on the Gauss -Green theorem applied to certain simplices close to the boundary. For convex domains $\Omega\subset\R^{n}$, such simplices can be obtained as the convex hull $\mathrm{conv}(x_{1},...,x_{n})$ of points $x_{1},...,x_{n}\in\Omega$, and then $\mathrm{conv}(x_{1},...,x_{n})$ is still contained in $\Omega$ by convexity. This is the key reason for first treating the case of convex sets $\Omega\subset\R^{n}$ in Section \ref{sec:L1}, before embarking on the more general Lipschitz case in Section \ref{sec:Lip}. The reader interested in the convex case might therefore skip the following section.

Specifically, for non-convex domains with Lipschitz boundary $\partial\Omega$ an analogous construction of simplices does not work, and it is here where we borrow some elementary observations from differential geometry:
\begin{itemize}
    \item[(i)] At least locally, one may define the requisite simplices on smooth (embedded) manifolds (cf. Section \ref{sec:def:simplex}).
    \item[(ii)] A suitable neighbourhood of the boundary $\partial \Omega$ can be endowed with the structure of a smooth manifold (cf. Sections \ref{sec:approxsmooth} and \ref{sec:collar}).
\end{itemize}

\subsection{Set-up}
Let $1\leq k\leq n$. If $\Omega\subset\R^{n}$ is convex, we define for $\overline{x}=(x_0,...,x_{k}) \in \Omega^{k+1}$ a map 
    \begin{align}\label{eq:Sdef}
    S_{\bar{x}} \colon \mathbb{D}^k \longrightarrow \Omega, \quad (t_1,...,t_k) \longmapsto \sum_{i=1}^k t_i x_i+ \Big(1- \sum_{i=1}^k t_{k}\Big) x_0,     
    \end{align}
where the unit simplex $\mathbb{D}^{k}$ is given by 
\begin{equation} \label{def:unitsimplex}
    \mathbb{D}^k = \{t=(t_{1},...,t_{k}) \in [0,1]^k \colon t_1+...+t_k \leq 1\}. 
\end{equation}
In particular, $S_{\bar{x}}(\mathbb{D}^{k})$ is the  simplex with vertices $x_i$, $i\in\{0,...,k\}$, so the convex hull of the $x_{i}$'s. For general bounded domains $\Omega$ with Lipschitz boundary, this convex hull does not need to be contained in $\Omega$, and so the definition of suitable curvilinear simplices for our future applications requires a more careful construction. The latter, in turn, is motivated by the following principles which will be formalised in Theorem \ref{thm:prop:simplex} below: Given points $x_{0},...,x_{k}\in\Omega$,
\begin{enumerate} [label=(\alph*)]
    \item\label{item:principlesA} the curvilinear simplex is well-defined whenever all the $x_i$'s are close to each other and sufficiently close to the boundary. 
    \item the diameter of the curvilinear simplex is bounded in terms of the maximal distance between two points $x_{i},x_{j}$, $i,j\in\{0,...,k\}$.
    \item the curvilinear simplex is either degenerate (i.e., less than $k$-dimensional) or a $k$-dimensional manifold with boundary.
    \item\label{item:principlesD} The faces of the simplices are compatible with each other, i.e., the definition of the face opposite to some $x_i$ does not depend on the choice of $x_i$.
\end{enumerate}
Quantifying the closeness described in \ref{item:principlesA} is related to the size of the neighbourhood of $\Omega$ to which we can extend divergence-free fields; cf. Section \ref{sec:Lip}. We achieve such a definition obeying \ref{item:principlesA}--\ref{item:principlesD} following three main steps. First, a definition of such simplices can be obtained by easier means if $\Omega$ has smooth boundary. Second, we show that a small  neighbourhood of the boundary of a smooth domain is smoothly diffeomorphic to $\partial \Omega \times (-1,1)$, cf. Section \ref{sec:collar}. We finally use this observation to define simplices on smooth domains and then go back to Lipschitz domains, cf. Section \ref{sec:def:simplex}. The passage from smooth to Lipschitz domains requires an approximation result due to \textsc{Ball \& Zarnescu} \cite{BZ}, which we state in Section \ref{sec:approxsmooth}. 

\subsection{Bi-Lipschitz maps between Lipschitz and smooth domains} \label{sec:approxsmooth}
We begin by recording the following approximation result due to \textsc{Ball \& Zarnescu} \cite{BZ}:
\begin{lemma}[{\cite[Thm. 5.1, Rem. 5.3]{BZ}}] \label{lemma:BZ}
Let $\Omega \subset \R^n$ be a bounded Lipschitz domain. Then there exists a domain $V$ with smooth boundary $\partial V$ and a Bi-Lipschitz map $\Gamma_1 \colon \overline{\Omega} \to \overline{V}$ such that the following hold:
    \begin{enumerate}
        \item\label{item:BZ1} $\Gamma_1$ maps $\partial \Omega$ onto $\partial V$. 
        \item\label{item:BZ2} $\Gamma_1$ is smooth inside $\Omega$: $\Gamma_{1}|_{\Omega}\in\hold^{\infty}(\Omega;\R^{n})$.
    \end{enumerate}
\end{lemma}
The full result from \cite[Thm. 5.1 and Rem. 5.3]{BZ} is actually stronger, but Lemma \ref{lemma:BZ} suffices for our purposes below. 
For future reference,  we remark that in the situation of Lemma \ref{lemma:BZ} we have 
\begin{align}\label{eq:BZbound}
\tfrac{1}{c} \dist(x,\partial \Omega) \leq \dist(\Gamma_1(x),\partial V) \leq c \dist(x,\partial \Omega)\qquad \text{for all } x\in\overline{\Omega}.
\end{align}
with a constant $c>0$ that only depends on $\max\{\mathrm{Lip}(\Gamma_{1}),\mathrm{Lip}(\Gamma_{1}^{-1})\}$. 
\begin{remark}
For most of the subsequent definitions e.g. in Section \ref{sec:def:simplex},  a $\hold^2$- instead of $\hold^{\infty}$-boundary would be sufficient. On the contrary, as many definitions (implicitly) depend on the curvature of the boundary, it is unclear how to define simplices \emph{without} a result similar to Lemma \ref{lemma:BZ}; also see Remark
\ref{rem:collar} below.
\end{remark}

\subsection{Collar neighbourhoods} \label{sec:collar}

\begin{figure} \label{figure:collar}
\includegraphics[width=0.9\textwidth]{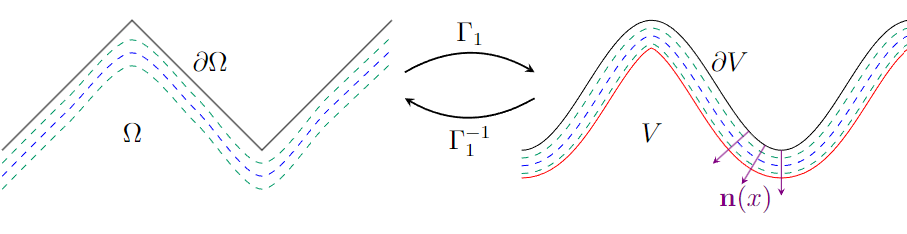}
\caption{On Corollary \ref{coro:diffeo}. \emph{Left-hand side}: The Lipschitz domain $\Omega$ is mapped to $V$ via $\Gamma_1$. We then define a collar map on $V$ and then go back with $\Gamma_1^{-1}$. The neighbourhood of the boundary of $\Omega$ then gets a smooth structure, as $\Gamma_1$ is smooth in the interior of $\Omega$. \emph{Right-hand side}: The dashed lines represent level sets of the collar map (i.e. lines $\Gamma_2^{-1}(\partial V \times \{t\})$. This collar map is constructed via taking the distance to $\partial V$ through the normal $\mathbf{n}(x)$ at $x \in \Omega$. The collar neighbourhood constructed in this way lies in between $\partial V$ and the red line (note that the red line is not smooth anymore), and in particular, the distance between $\partial V$ and this red line depends on the curvature of the domain. }

\end{figure}

In this section, we collect some background results on \emph{collar neighbourhoods} as frequently employed in algebraic topology, see e.g. \textsc{Hatcher} \cite[Proposition 3.42]{Hatcher}. We henceforth suppose that $V \subset \R^n$ is a bounded domain with smooth boundary. Then the boundary $\partial V$ is a smooth, compact submanifold in $\R^{n}$ of dimension $(n-1)$ which has a finite number of connected components. Moreover, there exists a finite collection of domains $O_i$, $i\in\{1,...,m\}$, that cover $\partial V$, such that (after translating and rotating)
\begin{align}\label{eq:covercollar1}
V \cap O_i = O_i \cap \{(x',s) \colon x' \in \R^{n-1}, s< f_i(x') \},\qquad i\in\{1,...,m\},
\end{align}
for suitable smooth functions $f_i \in \hold^{\infty}_c(\R^{n-1})$ together with
\[
\partial V \cap O_i = O_i \cap \{(x',f_i(x')) \colon x' \in \R^{n-1} \},\qquad i\in\{1,...,m\}.
\]
By the compactness of $\partial V$, Lebesgue's covering lemma (see, e.g.,  \cite[Lemma 27.5]{Munkres}) implies that there exists $\varepsilon_1>0$ such that for any $x \in \partial V$ there exists $i\in\{1,...,m\}$ with  $\ball_{\varepsilon_1}(x) \subset O_i$. For the following, let
\begin{align}\label{eq:curv}
\kappa = \sup_{i\in\{1,...,m\}} \Vert \!\D^2 f_i \Vert_{\lebe^{\infty}(\R^{n-1})}
\end{align}
be the \emph{maximal curvature} of $\partial V$. We then call $\partial V \times (-1,1) \subset \R^{n+1}$ a \emph{collar neighbourhood} of the boundary $\partial V$. We then have the following result: 
\begin{lemma} [Collar neighbourhood theorem, {\cite[Thm. 9.25]{Lee}}] \label{lemma:collar}
Let $V \subset \R^n$ be a bounded domain with boundary $\partial V$ of class $\hold^{\infty}$. Then there exists an open neighbourhood $\mathfrak{U}$ of $\partial V$ and a $\hold^{\infty}$-diffeomorphism $\Gamma_{2} \colon \mathfrak{U} \to \partial V\times (-1,1)$ such that $\Gamma_{2}(\partial V) = \partial V \times \{0\}$. 
\end{lemma}
In the setting of Lemma \ref{lemma:collar}, $\Gamma_2^{-1}(\partial V \times (0,1)) \subset V$ and $\Gamma_2^{-1}(\partial V \times (-1,0)) \subset \R^{n}\setminus \overline{V}$. We will only need the inner collar, that is, the map on $\mathfrak{U} \cap \overline{V}$. If we need both the neighbourhood is also referred to as \emph{tubular} neighbourhood. Even though Lemma \ref{lemma:collar} is well-known, our subsequent constructions hinge on the construction of the map  $\Gamma_{2}$ as appearing therein. We thus briefly pause to address the construction of $\Gamma_{2}$:
\begin{proof}[Construction of $\Gamma_{2}$ from Lemma \ref{lemma:collar}.] 
We adopt the convention from \eqref{eq:covercollar1}. By an explicit calculation for each neighbourhood $O_{i}$, we find a number $\epsilon_{i}>0$ (depending on $\kappa$), such that the distance-minimising orthogonal projection $\mathbb{P}\colon\ball_{\varepsilon_{i}}(\partial V)\cap O_{i}\to \partial V$ is well-defined. Recalling the signed distance function from \eqref{eq:signeddistance} and diminishing $\varepsilon_{i}>0$ if necessary, Lemma \ref{lem:signeddistance} implies that we may assume that 
\begin{equation*}
      \widetilde{\Gamma}_{2} \colon\ball_{\varepsilon_{i}}(\partial V)\cap O_{i}\ni x \mapsto (\mathbb{P}(x),d(x,\partial V))
\end{equation*}
is smooth and injective. For points $ x \in \partial V$, the corresponding differentials are given by $\mathrm{D} \mathbb{P}(x) = P_{T_x\partial V}$, the projection onto the tangent space $T_x \partial V = \langle \mathbf{n}(x) \rangle^{\perp}$ of $\partial V$, and $\mathrm{D}d(x, \partial V) = P_{N_x\partial V}$, the projection onto the normal space $\langle \mathbf{n}(x) \rangle$. Hence, $\D\widetilde{\Gamma}_{2} \colon \R^n \to \R^n$ is an isomorphism, and by the use of inverse function theorem we conclude that $\widetilde{\Gamma}_{2}$ is a local diffeomorphism. In particular, again by Lebesgue's covering lemma, there exists $\franz>0$, such that 
\begin{equation} \label{def:Gamma2}
    \widetilde{\Gamma}_{2} \colon \ball_{\franz}(\partial V) \to \partial V \times (-\franz,\franz)
\end{equation}
is a local diffeomorphism. We may, however, give the explicit inverse proving that this is a \emph{global} diffeomorphism. Indeed, define 
\begin{align*}
\widetilde{\Gamma}_{2}^{-1} \colon (x, s) \longmapsto x + s \mathbf{n}(x).
\end{align*}
Then $\widetilde{\Gamma}_{2}^{-1} \circ \widetilde{\Gamma}_2 = \id$ and $\widetilde{\Gamma}_2 \circ \widetilde{\Gamma}_2^{-1}=\id$, proving that $\widetilde{\Gamma}_{2}$ is globally bijective. Now rescaling the second coordinate by $\tfrac{1}{\franz}$ gives the desired map $\Gamma_2$.
\end{proof}
A combination of Lemma \ref{lemma:BZ} and \ref{lemma:collar} yields the following corollary:
\begin{corollary} \label{coro:diffeo}
    Let $\Omega \subset \R^n$ be a bounded domain with Lipschitz boundary. 
    Then there exist a neighbourhood $U\subset\overline{\Omega}$ of $\partial \Omega$ and  a bounded domain $V\subset\R^{n}$ with smooth boundary $N:=\partial V$ and a map $T=(T_{1},T_{2}) \colon U\to N\times[0,1)$ with the following properties:
    \begin{enumerate} 
        \item\label{item:collar1} $T$ maps $\partial \Omega$ onto $N \times \{0\}$; 
        \item\label{item:collar2} $T$ is Bi-Lipschitz;
        \item\label{item:collar3} $T \colon U \setminus \partial \Omega \to N \times (0,1)$ is a $\hold^{\infty}$-diffeomorphism;
        \item \label{coro:diffeo:4} there exists a constant $c>0$ such that we have 
        \[  
        \tfrac{1}{c} \dist(x,\partial \Omega) \leq T_2(x) \leq c\dist(x,\partial \Omega)\qquad\text{for all $x\in U$}.
        \]
    \end{enumerate}
\end{corollary}
\begin{proof} 
 We choose $V$ according to Lemma \ref{lemma:BZ} and $\mathfrak{U}$ according to Lemma \ref{lemma:collar}. Put $U :=\Gamma_{1}^{-1}(\mathfrak{U}\cap \overline{V})$.
We combine both Lemma \ref{lemma:BZ} and \ref{lemma:collar} and in particular, compose the maps $\Gamma_1$ and $\Gamma_2$. In particular, bound \ref{coro:diffeo:4} follows from the fact that $\Gamma_1$ is Bi-Lipschitz (where $c_1$ denotes the Bi-Lischitz constant of this map) and from the proof of Lemma \ref{lemma:collar}, i.e. set $c= c_1 \franz^{-1}$ in \ref{coro:diffeo:4}, where $c_1$ is the Bi-Lipschitz constant of $\Gamma_1$.
\end{proof}
\begin{remark}\label{rem:collar}
 The constant $c>0$ from Corollary \ref{coro:diffeo}\ref{coro:diffeo:4}  depends on several choices we make throughout. In particular, we fix the boundary of a smooth manifold $V$ that is Bi-Lipschitz homeomorphic to $\Omega$. This choice has a two-fold impact on the constant: On the one hand, if $V$ is a good approximation of the Lipschitz set $\Omega$, the constant $c$ in \ref{coro:diffeo:4} will be  close to $1$, but the curvature of the boundary of the domain might be quite high. This, in turn, diminishes the size of the collar neighbourhood. On the other hand, if the curvature of $\partial V$ is small and so the collar neighbourhood is larger, the Lipschitz constant of the map $T$ from Corollary \ref{coro:diffeo}\ref{item:collar2} deteriorates.
\end{remark}
\begin{remark}
While it is also possible to prove a collar theorem for Bi-Lipschitz homeomorphisms, in the following subsection we aim for a smoother structure than just $\hold^{0,1}$ that allows us to define smooth simplices. This is why we take the detour over the smooth approximation from \cite{BZ}.
\end{remark}
\subsection{Definition of simplices close to the boundary} \label{sec:def:simplex}
After the two previous preparatory subsections, we finally come to the definition of simplices in this section. To this end recall the definition of the $k$-dimensional unit simplex $\mathbb{D}^k$, cf. \eqref{def:unitsimplex}.
We denote by $\mathbb{D}^k_j$ the $j$-th face of $\mathbb{D}^k$:
\[
\mathbb{D}^k_j= \{ t=(t_{1},...,t_{k}) \in \mathbb{D}^k \colon t_j =0\}, \quad j=1,...,k, \quad \mathbb{D}^k_{0} =\{ z \in \mathbb{D}^k \colon z_1+...+z_k=1\}.
\]
and by $v_j$ the vertex that is opposite to $\mathbb{D}^k_j$, i.e.
\[
v_j = e_j, \quad j=1,...,k, \quad v_{0} = 0.
\]

First, we define the notion of a simplex on the one-sided collar domain $\partial V \times (0,1)$ and then use the map $T^{-1}$ from  Corollary \ref{coro:diffeo} to obtain curvilinear simplices in the original domain $\Omega$.

Let $(x_0,s_0),...,(x_k,s_k) \in \partial V \times (0,1)$. We give a coordinate-wise definition and hence introduce for $t \in \mathbb{D}^k$ the second  coordinate by
\begin{equation} \label{def:coord2}
s(t) := \sum_{i=1}^k t_i s_i + \Big(1-\sum_{i=1}^k t_{i}\Big) s_0.
\end{equation}
For the first coordinate (i.e. elements in $\partial V$), one might define simplices in multiple ways, most notably by a barycentre construction, cf. \cite[Chapter 8]{BK81} or \cite{Sander}, using that $\partial V$ is a smooth manifold. For our purposes, however, the following construction is well-suited (also see Remark \ref{rem:simplex:choice} at the end of the section as well as Figure \ref{figure:simplex} in Section \ref{sec:Lip} for a visualisation): Recall the projection map $\mathbb{P} \colon \R^n \to \partial V$, which is well-defined and smooth in a  neighbourhood of $\partial V$. Then we define the simplex in the first coordinate as the projection of the convex hull of $x_0,\ldots,x_k$ onto $\partial V$:
\begin{equation} \label{def:coord1}
    x(t) := \mathbb{P} \left( \sum_{i=1}^k t_i x_i + \Big(1-\sum_{i=1}^k t_i \Big) x_0\right).
\end{equation}
We finally define the simplex as the image of the map defined by $S_{\mathbf{x}_{I}} \colon \mathbb{D}^k \to \partial V \times (0,1)$ by
\[
S_{\mathbf{x}_{I}} \colon t \longmapsto (x(t),s(t)),
\]
where we employed the index $\mathbf{x}_{I}:=(x_{0},...,x_{k})\in (\R^{n})^{k+1 }$. To state the properties of the map $S$, we recall the threshold number $\franz>0$ from the proof of Lemma \ref{lemma:collar}:
\begin{enumerate}
    \item The map $S_{\mathbf{x}_{I}}$ is well-defined whenever $\sup_{0\leq i,j \leq k} \vert x_i -x_j \vert< \tfrac{1}{2} \franz$ and  $\sup_{0 \leq i \leq k} \dist(x,\partial V) \leq \tfrac{1}{2}\franz$ with $\franz>0$ as in \eqref{def:Gamma2}. Indeed, in this case we have for any $t \in \mathbb{D}^k$ that
    \[
    \dist\left( \sum_{i=1}^k t_i x_i + \Big(1-\sum_{i=1}^k t_{i}\Big) x_0, \partial V\right) < \franz,
    \]
   and the well-definedness follows. In the following, we will therefore always assume that $\sup_{0\leq i,j\leq k} \vert x_i -x_j \vert \leq \tfrac{\franz}{2}$ and  $\sup_{0 \leq i \leq k} \dist(x,\partial V) \leq \tfrac{\franz}{2}$.
    \item We have
        \begin{equation} \label{eq:distboundary}
            \min_{0\leq i \leq k} s_i \leq s(t) \leq \max_{0\leq i \leq k} s_i.
        \end{equation}
    \item \label{obs:3} Let $M = S_{\mathbf{x}_{I}}(\mathbb{D}^k)$. Then precisely one of the following holds: 
        \begin{itemize}
            \item Either $M$ is a $k$-dimensional manifold with boundary and $S_{\mathbf{x}_{I}}$ is a smooth diffeomorphism,
            \item or $\mathscr{H}^{k}(M) =0$. This is the case in the following two situations: Either the dimension of $\conv(x_0,...,x_k)$ (the convex hull of these $(k+1)$ points) is smaller or equal to $(k-2)$ or the dimension is $(k-1)$ and $s_0=s_1=\ldots=s_k$.
        \end{itemize}
        This can be shown by a local calculation (note that we may assume that $\partial V = \{(x,f(x)\}$ locally). An important point here is that the manifold $\conv(x_0,\ldots,x_k)$ is 'almost tangential', i.e. the normal vector (that is required for the definition of the map $\mathbb{P}$) is almost perpendicular to $\conv(x_0,...,x_k$). This entails that $\mathbb{P}$ will \emph{not} map the $k$-dimensional manifold $\conv(x_0,\ldots,x_k)$ to a lower-dimensional object.
    \item \label{obs:4} There is a constant $x>0$ not depending on $x_0,...,x_k$ such that
        \[
        \vert \D S \vert \leq C \max\{ \sup_{i,j} \vert s_i -s_j \vert , \sup_{i,j}\vert x_i-x_j \vert \} \cdot \sup_{i,j}\vert x_i-x_j \vert^{k-1}, \
        \]
        This indeed follows from the definition and the fact that the projection $\mathbb{P}$ has bounded derivative.
     \item \textbf{Compatibility:} The restriction of $S$ to the face $\mathbb{D}^k_j$ is independent on the choice of $(x_j,s_j)$. Indeed, considering the definition, we have $t_j =0$ on $\mathbb{D}^k_j$ and we directly observe that $s(t)$ and $x(t)$ are independent of $(x_j,s_j)$.  
\end{enumerate}

For $x_0,...,x_k \in \Omega$ we define curvilinear simplices as follows. Consider $T(x_0),\ldots,T(x_k) \in V \times (0,1)$, define the simplex through the map $S$ there and then apply $T^{-1}$. Denoting by $\bar{x}= (x_0,...,x_k) \in \Omega^{k+1}$ we then get a map $S_{\bar{x}} \colon \mathbb{D}^k \to \Omega$. 

\begin{theorem} [Properties of the simplex] \label{thm:prop:simplex}
    Suppose that $k \in \{2,...,n+1\}$. Let $\Omega \subset \R^n$ be an open Lipschitz domain. Let $C_1>0$. There is \begin{itemize}
    \item 
    $\varepsilon=\varepsilon(\Omega)$ (depending on $\franz$, i.e. on the curvature $\kappa$ and on the covering $\varepsilon_1$);
    \item 
    $C_2>0$ (depending on the Lipschitz constant of the map $U \to \partial V \times (0,1)$);
    \item $\alpha>0$ (depending on the Lipschitz constant of $\Gamma_1$);
    \end{itemize}
    such that the following holds: If  $\bar{x}=(x_0,...,x_k) \in \Omega^{k+1}$ satisfies
    \begin{enumerate} [label=(A\arabic*)]
        \item \label{ass1} $\eta := \inf_{0\leq i\leq k} \dist(x_i,\partial \Omega) \leq \varepsilon$, 
        \item \label{ass2} $ \sup_{0\leq i\leq k} \dist(x_i, \partial \Omega) \leq C_1 \eta$, and 
         \item \label{ass3} $\sup_{0\leq i,j \leq k} \vert x_i - x_j \vert \leq C_1 \eta$, 
    \end{enumerate}
   then the map $S_{\bar x} \colon \mathbb{D}^k \to \Omega$ and the curvilinear simplex $M_{\bar{x}} = S_{\bar{x}}(\mathbb{D}^k)$ satisfy the following: 
    \begin{enumerate} [label=(S\arabic*)]
    \item \label{def:prop:1} $S_{\bar x}\in \hold^{\infty}(\mathbb{D}^k;\Omega)$; 
    \item  \label{def:prop:2} The vertices of $\mathbb{D}^k$ are mapped to $x_0,...,x_{k}$, i.e. $S_{\bar x}(v_s) = S_{\bar x}(v_s)$;
    \item  \label{def:prop:3} One of the following holds:
        \begin{itemize}
            \item $M$ is a $k$-dimensional manifold with boundary and $S_{\bar x}$ is a $\hold^{\infty}$-diffeormorphism;
            \item The rank of $S_{\bar x}$ is smaller than $k$ for all $x \in \mathbb{D}^k$;
        \end{itemize}
    \item \label{def:prop:35} The derivative of $S_{\bar x}$ is bounded by $C_2 \eta$, i.e. the $\mathscr{H}^k$-measure of $M_{\bar x}$ is bounded by $C_2^k \eta^{k}$;
    \item  \label{def:prop:4} For all $x \in M$ we have
        \[
        \alpha^{-1} \eta \leq \dist(x,\partial \Omega) \leq C \alpha  \eta;
        \]
       
    \item  \label{def:prop:5} \emph{Compatibility:} The map $\Gamma$ restricted to the face $\mathbb{D}^k_s$ \emph{does not} depend on the $x_{k+1}$.

    \end{enumerate}
\end{theorem}

\begin{proof}[Proof (Sketch)]
    This statement can be recovered through Corollary \ref{coro:diffeo} and the previously outlined properties for the simplex defined on $\partial V \times (0,1)$. Indeed, \ref{def:prop:1} and \ref{def:prop:2} follow directly from the definition. Property \ref{def:prop:3} follows by the third observation \ref{obs:3} and, likewise, \ref{def:prop:4} may be derived from \ref{obs:4} and the fact that $T$ is Bi-Lipschitz.  The bound on the distance to the boundary, \ref{def:prop:4} is a direct consequence of \eqref{eq:distboundary} and Corollary \ref{coro:diffeo} \ref{coro:diffeo:4}. Finally the compatibility \ref{def:prop:5} follows from the compatibility for the simplices in $\partial V \times (0,1)$.
\end{proof}

\begin{remark}[On the choice of the simplex] \label{rem:simplex:choice}
Having discussed the desired properties of the simplex, we shortly discuss the choice we have made at the beginning (i.e. which object to be the simplex). To this end, note the following:
\begin{enumerate} [label=(\roman*)]
    \item On the one hand, as the projection onto the boundary is well-defined in a neighbourhood of the boundary (i.e. \emph{globally}) the definition does not depend on local choices of a coordinate system.
    \item On the other hand, as the projection onto the boundary is a \emph{local} operation, for purposes of calculation we might reduce ourselves to situations, where the boundary is given by a graph of a function.
    \item A key point in the proofs (in particular Lemma \ref{lem:LipschitzL1boonhund} that proves the $\lebe^1$-bound) is to show that the simplices behave in an appropriate way   whenever we fix $(k-1)$ points and vary the final vertex of the simplex. In particular, we want no 'concentration of simplices' (cf. Figure \ref{fig:concentration}), i.e. that different points lead to  a similar simplex.

    Such a property is much easier to verify for a construction that builds on barycentres than on a construction using geodesics.

\end{enumerate}

\end{remark}
  
\begin{figure} \label{fig:concentration}
\begin{tikzpicture}

 \coordinate[label=above:\Large{\textcolor{black}{$x_0$}}] (D) at (0,0);
 \coordinate[label=right:\Large{\textcolor{gray}{$x_1$}}] (E) at (3,0);

 \filldraw (0,0) circle (2pt);
 \filldraw[color=gray] (3,0) circle (2pt);
 \filldraw[color=gray] (3,0.5) circle (2pt);
 \filldraw[color=gray] (3,1) circle (2pt);
 \filldraw[color=gray] (3,-0.5) circle (2pt);
 \filldraw[color=gray] (3,-1) circle (2pt);
    \draw[color=black] (0,0) to (3,0);
    \draw[color=black] (0,0) to (3,0.5);
    \draw[color=black] (0,0) to (3,-0.5);
    \draw[color=black] (0,0) to (3,1);
    \draw[color=black] (0,0) to (3,-1);
 \coordinate[label=above:\Large{\textcolor{black}{$x_0$}}] (F) at (5,0);
 \coordinate[label=right:\Large{\textcolor{gray}{$x_1$}}] (G) at (8,0);   
\filldraw (5,0) circle (2pt);
 \filldraw[color=gray] (8,0) circle (2pt);
 \filldraw[color=gray] (8,0.5) circle (2pt);
 \filldraw[color=gray] (8,1) circle (2pt);
 \filldraw[color=gray] (8,-0.5) circle (2pt);
 \filldraw[color=gray] (8,-1) circle (2pt);
\draw[color=black] (5,0) to (8,0);
\draw[color=black] plot [smooth] coordinates {(5,0)(6,0.01)(7,0.05)(8,0.5)};
\draw[color=black] plot [smooth] coordinates {(5,0)(6,-0.01)(7,-0.05)(8,-0.5)};
\draw[color=black] plot [smooth] coordinates {(5,0)(6,0.02)(7,0.1)(8,1)};
\draw[color=black] plot [smooth] coordinates {(5,0)(6,-0.02)(7,-0.1)(8,-1)};
\end{tikzpicture} 
\caption{On Remark \ref{rem:simplex:choice}: Consider 1D-simplices, i.e. paths. On the left-hand-side the corresponding paths are the straight paths. Consequently, varying $x_1$ has an effect on the whole simplex, i.e. even if we change $x_1$, this has an effect on the shape of the whole simplex. On the right-hand-side the paths are chosen differently (maybe according to some underlying geometry). Now changing $x_1$ has (almost) no effect on the simplex in some distance to $x_1$: This is a situation that we would like to avoid.
}
\end{figure}
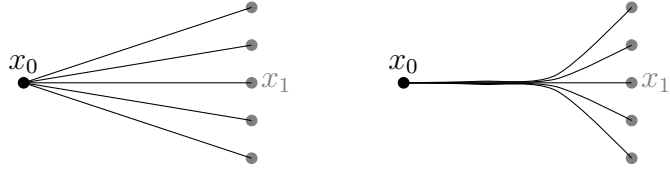

\section{Divergence-free extensions in $\lebe^{1}$ for convex domains}\label{sec:L1}
In this section we address Theorem \ref{thm:W1} in the case where $\Omega\subset\R^{n}$ is an open, bounded \emph{and convex} domain. As explained in the introduction after Theorem \ref{thm:W1}, the key construction is most transparent in this situation and will serve as a conceptual starting point in Section \ref{sec:Lip} when dealing with the more general Lipschitz domains. Specifically, our main objective of the present section is to establish the following result: 
\begin{theorem}[Theorem \ref{thm:L1} for convex domains]\label{thm:divL1conv}
Let $\Omega \subset \R^n$ be open, bounded and convex. Then there exists a \emph{linear and bounded extension operator} $\mathscr{E}_{\Omega} \colon \lebe^{1}(\Omega;\R^{n}) \to \lebe^{1}(\R^n;\R^n)$ such that, in particular,   $\mathscr{E}_{\Omega} u \vert_{\Omega} = u$ $\mathscr{L}^{n}$-a.e. in $\Omega$ whenever $u\in\lebe^{1}(\Omega;\R^{n})$ and 
\begin{align}\label{eq:extensionconvex}
u\in\lebe^{1}(\Omega;\R^{n})\;\text{and}\;\diver u =0 \text{ in } \mathscr{D}'(\Omega) \quad \Longrightarrow \quad \diver \mathscr{E}_{\Omega} u =0\;\text{ in } \;\mathscr{D}'(\R^n).
\end{align}
\end{theorem}
Towards the proof of Theorem \ref{thm:divL1conv}, we give the geometric set-up and the definition of $\mathscr{E}_{\Omega}$ in Section \ref{sec:geomsetupconvex}. Then, examining the properties of $\mathscr{E}_{\Omega}$ on $\hold^{1}(\overline{\Omega};\wedgeq^{n-1}(\R^{n})^{*})$ in Section \ref{sec:proofL1divcon}, the proof of Theorem \ref{thm:divL1conv} is then provided in Section \ref{sec:density}.
\subsection{Geometric set-up for convex $\Omega\subset\R^{n}$ and definition of $\mathscr{E}_{\Omega}$}\label{sec:geomsetupconvex}
Throughout this entire section, we assume $\Omega\subset\R^{n}$ to be convex so that, as a consequence, $\partial\Omega$ is automatically Lipschitz.  In consequence, there exists a finite open covering $\mathcal{O}=\{O_{1},...,O_{m}\}$ with \eqref{eq:nowandthen} and \eqref{eq:nowandthen1}. For future reference, let us note that if $\varepsilon>0$ is sufficiently small, then $\mathcal{O}$ is also an open covering of $\partial\Omega_{\varepsilon}$, where we denote $
\Omega_{\varepsilon}:=\{x\in\Omega\colon\;\dista(x,\partial\Omega)>\varepsilon\}$. 
Adopting the representation \eqref{eq:nowandthen} of $\Omega\cap O_{i}$ for  $i\in\{1,...,m\}$, there exists $\varepsilon_{0}>0$ such that we have 
\begin{align*}
\Omega_{\varepsilon}\cap O_{i} = \{x=(x_{1},...,x_{n})\in O_{i}\colon\; x_{n}<f_{i}^{\varepsilon}(x_{1},...,x_{n-1})\}
\end{align*}
for all $0<\varepsilon<\varepsilon_{0}$ with suitable Lipschitz maps  $f_{i}^{\varepsilon}\colon U_{i}\to\R$ having Lipschitz constants being  bounded uniformly in $0<\varepsilon<\varepsilon_{0}$. Hence, recalling the notation \eqref{eq:nowandthen1}, we may record that
\begin{align}\label{eq:Lipindependent}
\mathrm{Lip}(\mathcal{O},\Omega_{\varepsilon})\;\text{is uniformly bounded in terms of}\;\mathrm{Lip}(\mathcal{O},\Omega)\;\text{for all}\;0<\varepsilon<\varepsilon_{0}. 
\end{align}
We now pick Whitney coverings $\mathcal{W}_{1}$ and $\mathcal{W}_{2}$ of $\Omega$ and $\R^{n}\setminus\overline{\Omega}$ in the sense of \ref{item:Whitney1}--\ref{item:Whitney3}. Based on Corollary \ref{cor:Jones}, we then infer the existence of a map $\Psi\colon\mathcal{W}_{2}\to\mathcal{W}_{1}$, a cube $Q'_{0}\in\mathcal{W}_{1}$ and a threshold number $\eta>0$ such that the following hold for all $Q\in\mathcal{W}_{2}$:
\begin{enumerate}[label=(J\arabic*)]
    \item\label{A:1} If $\ell(Q) > \eta$, then $\Psi(Q) = Q_0'$; 
    \item\label{A:2} $\frac{1}{4}\ell(Q) \leq \ell(\Psi(Q)) \leq 4 \ell(Q) $ for all $Q\in\mathcal{W}_{2}$ with $\ell(Q) \leq \eta$;
    \item \label{A:3} $\dist(Q,\Psi(Q)) \leq C \ell(Q)$ and the constant $C>0$ only depends on $\Lip(\Omega,\Ocal)$.
\end{enumerate}
In consequence, if we denote by $\mathcal{W}_{1}^{\varepsilon},\mathcal{W}_{2}^{\varepsilon}$ Whitney coverings of $\Omega_{\varepsilon}$ and $\R^{n}\setminus\overline{\Omega_{\varepsilon}}$ and by $\Psi_{\varepsilon}\colon\mathcal{W}_{2}^{\varepsilon}\to\mathcal{W}_{1}^{\varepsilon}$ the corresponding map from Lemma \ref{cor:Jones}, it follows by the analogous properties \ref{A:1}--\ref{A:3} for $Q\in\mathcal{W}_{2}^{\varepsilon}$ and \eqref{eq:Lipindependent} that we have for all $Q\in\mathcal{W}_{2}^{\varepsilon}$ 
\begin{enumerate}[start=3,label=(J\arabic*')]
    \item \label{A:3a} $\dist(Q,\Psi_{\varepsilon}(Q)) \leq C \ell(Q)$ and the constant $C>0$ only depends on $\Lip(\Omega,\Ocal)$
\end{enumerate}
provided $0<\varepsilon<\varepsilon_{0}$. Moreover, the above considerations establish that the involved constants of the construction \emph{do not} depend on the choice of $0<\varepsilon<\varepsilon_0$. Therefore, in the following, we omit the index $\varepsilon>0$ and continue with the construction for the open, bounded and convex set $\Omega$.
\smallskip

From now on, we denote cubes from the exterior cover by $Q_i$ and cubes from the interior cover by $Q_j'$ (and, later on after applying the map $\Psi$, by $Q^i$).
We now choose a smooth partition of unity $(\varphi_{i})$ subject to the slightly blown up Whitney covering $\mathcal{W}_{2}=(Q_{i})$ of $\R^{n}\setminus\overline{\Omega}$. By this we understand that $(\varphi_{i})$ satisfies the following properties:  
\begin{enumerate}[label=(PU\arabic*)]
    \item\label{item:POU1} $\phi_i \in \hold_c^{\infty}(\tfrac{7}{6}Q_i)$ for all $i\in\mathbb{N}$;
    \item\label{item:POU2} $0\leq \phi_i\leq 1$ for all $i\in\mathbb{N}$ and $\sum_{i=1}^{\infty}\varphi_{i}=1$ on $\R^{n}\setminus\overline{\Omega}$; 
    \item\label{item:POU3} there exists a constant $c=c(n)>0$ such that we have 
\begin{align*}    
    \Vert\!\D \phi_i \Vert_{\lebe^{\infty}(\R^{n})} \leq \frac{c}{\ell(Q_i)}\;\;\;\text{and}\;\;\; \Vert\!\D^{2} \phi_i \Vert_{\lebe^{\infty}(\R^{n})} \leq \frac{c}{\ell(Q_i)^{2}}\qquad \text{for all}\;i\in\mathbb{N}. 
\end{align*}
\end{enumerate}
Given $i\in\mathbb{N}$, it is then convenient to define the collection of \emph{neighbouring cubes} by
\begin{align}\label{eq:peaceintheneighbourhood}
\mathcal{N}(i):=\{\tfrac{7}{6}Q\colon\; Q\in\mathcal{W}_{2},\;\;\;\tfrac{7}{6}Q\cap\tfrac{7}{6}Q_{i}\neq\emptyset\}.
\end{align}
For cubes $Q'_j \in \Wcal_{1}$, consider the down-scaled cube $\widetilde{Q}_j :=\tfrac{1}{2} Q'_j$ and define a normalised  measure $\overline{\mu}_j$ supported in $\widetilde{Q}_{j}$ by
\begin{align}\label{eq:mudefMAIN}
\bar{\mu}_j = \frac{\mathscr{L}^n \mres\widetilde{Q}_j}{\mathscr{L}^n(\widetilde{Q}_j)}. \end{align}
Writing $\mathcal{W}_{2}=(Q_{i})$, we now put $Q^{i}:=\Psi(Q_{i})\in\mathcal{W}_{1}$. Note that for $i_0 \neq i_1$, we might have $Q^{i_1}=Q^{i_2}$. The number of indices $i$ such that $Q=Q^{i}$ for cubes $Q^{i} \in \mathcal{W}_2$ of a fixed distance to $\partial\Omega$ is, however, uniformly bounded by a dimensional constant and this fixed distance.

If $Q'_j \in\mathcal{W}_{1}$ satisfies $Q'_{j}= \Psi(Q_i)$ for some $Q_{i}\in\mathcal{W}_{2}$, let us further denote $\mu^i = \bar{\mu}_j$. For a multi-index $I=(i_1,...,i_r)$ we denote by $\mu^I$ the product measure of the $\mu^{i_l}$'s acting on $(\R^n)^r$, i.e.
\begin{align}\label{eq:prodmeasuredef}
\mu^I := \mu^{i_1} \times ... \times \mu^{i_r}.
\end{align}
Given a multi-index $I=(i_{1},...,i_{r})\in\mathbb{N}^{r}$ and points $x_{i_{1}},...,x_{i_{r}} \in\R^{n}$,  we employ the shorthand $\mathbf{x}_{I}=(x_{i_{1}},...,x_{i_{r}})$ and then denote by 
\begin{align*}
\conv(\mathbf{x}_{I}):=\conv(x_{i_{1}},...x_{i_{r}})
\end{align*}
the convex hull of these points. 
Note that the convex hull is an $(r-1)$-dimensional simplex, and if its $(r-1)$-dimensional Hausdorff measure is non-zero, then we say that this simplex is \emph{non-degenerate}. If this is the case, we may define an element of ${\wedgeq}{^{r-1}}(\R^{n})$ by
\begin{align}\label{eq:normaldefine}
\nu(x_{i_1},...,x_{i_r}) := \frac{1}{(r-1)!}(x_{i_r} - x_{i_{r-1}}) \wedge (x_{i_{r-1}} - x_{i_{r-2}}) \wedge ... \wedge (x_{i_{2}}-x_{i_{1}}).
\end{align}
Note that the length of this vector is exactly the $\mathscr{H}^{r-1}$-measure of the simplex and that $\nu(x_{i_{1}},...,x_{i_{r}})$ might be seen as a \emph{normal vector to the simplex}. In particular, given a differential form $u \in  \hold^1(\R^n;\wedgeq^r (\R^{n})^{*})$, this vector helps us to formulate Stokes' or Gauss' theorem, respectively. This is the key motivation for Definition \ref{def:extension:l1} below, for which we identify a vector field $u \colon \R^n \to \R^n$ with the (non-relabeled) differential form $u \colon \R^n \to \wedgeq^{n-1}(\R^{n})^{*}$ via the identification of $\R^n$ with $\wedgeq^{n-1}(\R^{n})^{\ast}$. In particular,
\[
e_1 \cong \dif x_2 \wedge ... \wedge \dif x_n, \quad e_2 \cong - \dif x_1 \wedge \dif x_3 \wedge... \wedge \dif x_n \text{ etc.,}
\]
so that the divergence operator becomes the operator of exterior differentiation (note that $\wedgeq^n (\R^{n})^{\ast}$ is isomorphic to $\R$), cf. \eqref{eq:differentialconvention1} and \eqref{eq:differentialconvention2}. In particular, $\diver$-free functions are closed differential forms and vice-versa. We now come to the definition of $\mathscr{E}_{\Omega}$:
\begin{definition}[Divergence-free extension operator] \label{def:extension:l1}
    Let $u \in \hold^1(\overline{\Omega};\wedgeq^{n-1}(\R^{n})^{*})$. We define 
    \begin{align}\label{def:Eomegscr}
    \mathscr{E}_{\Omega}[u]:=\begin{cases} 
    u &\;\text{in}\;\;\Omega,\\ 
    \E_{\Omega}[u]&\;\text{in}\;\;\R^{n}\setminus\overline{\Omega}, 
    \end{cases}
    \end{align}
where $\E_{\Omega}u$ is given by 
    \begin{equation} \label{def:Eomega}
    \begin{split}
         \E_{\Omega}[u]  & :=  (-1)^{n-1} \sum_{I=(i_1,...,i_n)\in\mathbb{N}^{n}} \Big(\phi_{i_n} \wedge \dif \phi_{i_{n-1}} \wedge \dots \wedge \dif \phi_{i_1} \Big. \\ & \Big. \;\;\;\;\;\;\;\;\;\;\;\;\;\;\;\;\;\;\;\;\;\;\;\;\cdot \int_{(\R^{n})^{n}} \fint_{\conv(\mathbf{x}_{I})} u(z) \nu(\mathbf{x}_{I}) \dHaus^{n-1}(z) ~\textup{d}\mu^{I}(\mathbf{x}_{I})\Big).
        \end{split}
    \end{equation}
If the simplex spanned by $x_{i_1},...,x_{i_n}$ is degenerate, we define the inner integral over the convex hull/simplex $\conv(x_{i_{1}},...,x_{i_{n}})$ to be zero.    
\end{definition}
Note that due to convexity of the domain, in \eqref{def:Eomega} the integral over the convex hull  of $x_{i_1},...,x_{i_n}$ is always well-defined. Moreover, by the local finiteness of the Whitney covering (cf. \ref{item:Whitney2}), the sum in \eqref{def:Eomega} is well-defined. By virtue of Stokes' theorem, it will become clear in the proof of Theorem \ref{thm:divL1conv} in Section \ref{sec:proofL1divcon} below why the definition \eqref{def:Eomegscr} via \eqref{def:Eomega} in fact is a natural choice and maps divergence-free fields to divergence-free fields. 
\subsection{Properties of $\mathscr{E}_{\Omega}[u]$ for $u\in\hold^{1}(\overline{\Omega};\wedgeq^{n-1}(\R^{n})^{*})$}\label{sec:proofL1divcon}
We now proceed to the proof of Theorem \ref{thm:divL1conv} and devote this section to the proof of the following intermediate result:
\begin{proposition}\label{prop:intermediateL1convex}
Let $\Omega\subset\R^{n}$ be open, bounded and convex, and let the linear operator $\mathscr{E}_{\Omega}$ be as in Definition \ref{def:extension:l1}. Then there exists a constant $c>0$ only depending on $n$ and $\Lip(\Omega,\mathcal{O})$ such that we have 
\begin{align*}
\|\mathscr{E}_{\Omega}[u]\|_{\lebe^{1}(\R^{n})}\leq c\|u\|_{\lebe^{1}(\Omega)}\qquad\text{for all}\;u\in\hold^{1}(\overline{\Omega};{\wedgeq}{^{n-1}}(\R^{n})^{*}). 
\end{align*}
Moreover, if $u\in\hold^{1}(\overline{\Omega};\wedgeq^{n-1}(\R^{n})^{*})$ satisfies $\dif u=0$ in $\Omega$, then we have $\dif\mathscr{E}_{\Omega}[u]=0$ in $\mathscr{D}'(\R^{n};{\wedgeq}{^{n}}(\R^{n})^{*})$. 
\end{proposition}
The proof of Proposition \ref{prop:intermediateL1convex} is divided into several steps, and to this end, we start by giving an outline of the proof  first. Specifically, we will establish 
\begin{enumerate} 
\item\label{item:stepP1} in Lemma \ref{lemma:pointwise:solenoidal} that, if $u \in \hold^1(\bar{\Omega};\wedgeq^{n-1}(\R^{n})^{*})$, then  $\E_{\Omega}[u] \in \hold^{\infty}(\R^{n}\setminus\overline{\Omega};\wedgeq^{n-1}(\R^{n})^{*})$ and that, if $\dif u =0$ holds pointwisely in $\Omega$, the equation $\dif\mathscr{E}_{\Omega}[u]=0$ holds pointwisely in $\R^{n}\setminus\overline{\Omega}$;
\item\label{item:stepP2} in Lemma \ref{lemma:globalsolenoidality} that, if $u\in\hold^{1}(\overline{\Omega};\wedgeq^{n-1}(\R^{n})^{*})$, then the distributional differential expression $\dif\mathscr{E}_{\Omega}[u]\in\mathscr{D}'(\R^{n};\wedgeq^{n}(\R^{n})^{*})$ can be represented by an $\lebe^{1}(\R^{n};\wedgeq^{n}(\R^{n})^{*})$-function;
\item\label{item:stepP3} in Corollary \ref{coro:divfree} that, if $u \in \hold^1(\bar{\Omega};\wedgeq^{n-1}(\R^{n})^{*})$ satisfies $\dif u=0$ pointwisely in $\Omega$, then $\dif\mathscr{E}_{\Omega}[u]=0$ holds in $\mathscr{D}'(\R^{n},\wedgeq^{n}(\R^{n})^{*})$. This is a consequence of Lemmas \ref{lemma:pointwise:solenoidal} and \ref{lemma:globalsolenoidality};
\item\label{item:stepP4} in Lemmas \ref{lem:convhullL1bound} and \ref{lemma:L1bound:c1} that $\mathscr{E}_{\Omega}$ is an $\lebe^{1}$-bounded linear operator on $\hold^{1}(\overline{\Omega},\wedgeq^{n-1}(\R^{n})^{*})$, with the underlying constants only depending on certain geometric features of $\Omega$ and $\partial\Omega$.
\end{enumerate} 
Step \ref{item:stepP2} explicitly uses that $u\in\hold^{1}(\overline{\Omega};\wedgeq^{n-1}(\R^{n})^{*})$. In order to prove Theorem \ref{thm:divL1conv}, we will combine steps \ref{item:stepP1}--\ref{item:stepP4} together with a density result. This density result, in turn, will rely on the properties of $\mathscr{E}_{\Omega}$ on $\hold^{1}(\overline{\Omega};\wedgeq^{n-1}(\R^{n})^{*})$ as gathered in steps \ref{item:stepP1}--\ref{item:stepP4}, and will be given together with the proof of Theorem \ref{thm:divL1conv} in Section \ref{sec:density} below.
We begin with an elementary observation: 
\begin{lemma}\label{lem:compactsupport}
For any $u\in\hold^{1}(\overline{\Omega};\wedgeq^{n-1}(\R^{n})^{*})$, $\spt(\mathscr{E}_{\Omega}[u])$ is compact. Specifically, letting $\eta>0$ be the constant from \ref{A:1}, $c>0$ be the constant from \ref{item:Whitney3} and setting 
\begin{align}\label{eq:michaeljackson}
\vartheta =2\eta\Big(\frac{13\sqrt{n}}{12}+c\Big), 
\end{align}
we have $\spt(\mathscr{E}_{\Omega}[u])\subset\overline{\ball}_{\vartheta}(\Omega)$. 
\end{lemma}
\begin{proof}
As usual, we denote by $\Psi\colon\mathcal{W}_{2}\to\mathcal{W}_{1}$ the map from Lemma \ref{cor:Jones}. In particular, by \ref{item:Jonesrefined1}, there exists a number $\eta>0$ such that $\ell(Q)>\eta$ for $Q\in\mathcal{W}_{2}$ implies that $\Psi(Q)=Q'_{0}$. With $\vartheta$ as in \eqref{eq:michaeljackson}, we claim that 
\begin{align}\label{eq:Titanic}
y\in\R^{n}\setminus\ball_{\vartheta}(\Omega)\;\;\text{and}\;\;\varphi_{i}(y)\neq 0 \Longrightarrow \Psi(Q_{i})=Q'_{0}.
\end{align}
To see \eqref{eq:Titanic}, let $y\in\R^{n}\setminus\ball_{\vartheta}(\Omega)$ be such that $\varphi_{i}(y)\neq 0$. We note that, for all $z\in Q_{i}$ and all $y\in \frac{7}{6}Q_{i}$, there holds $|z-y|\leq \frac{13}{12}\sqrt{n}\ell(Q_{i})$, and hence we get for all $x\in\partial\Omega$ and $z\in Q_{i}$
\begin{align*}
\dist(y,\partial\Omega) & \leq |y-x| \leq |y-z|+|z-x| \leq \frac{13}{12}\sqrt{n}\ell(Q_{i}) + |z-x|. 
\end{align*}
Infimising the previous inequality over all $x\in\partial\Omega$ and $z\in Q_{i}$ then yields 
\begin{align*}
\vartheta \stackrel{y\in\R^{n}\setminus\ball_{\vartheta}(\Omega)}{\leq} \dist(y,\partial\Omega) & \leq  \frac{13}{12}\sqrt{n}\ell(Q_{i}) + \dist(Q_{i},\partial\Omega) \stackrel{\ref{item:Whitney3}}{\leq} \Big(\frac{13}{12}\sqrt{n}+c\Big)\ell(Q_{i}).
\end{align*}
By our choice of $\vartheta$, cf. \eqref{eq:Titanic}, we then infer that $\ell(Q_{i})>\eta$ and thus $\Psi(Q_{i})=Q'_{0}$. This settles the implication \eqref{eq:Titanic}. 

Hence, if $y\in\R^{n}\setminus\ball_{\vartheta}(\Omega)$ and $(i_{1},...,i_{n})\in\mathbb{N}^{n}$ are such that  $y\in\spt(\varphi_{i_{1}})\cap...\cap\spt(\varphi_{i_{n}})$, then each of the $Q_{i_{j}}$'s, $j\in\{1,...,n\}$,  is mapped to the same cube $Q'_{0}\subset\Omega$. However, since we have $\sum_{i_{s}\in\mathbb{N}}\dif\varphi_{i_{s}}=0$ in $\R^{n}\setminus\overline{\Omega}$, this implies $\mathscr{E}_{\Omega}[u]=0$ in $\R^{n}\setminus\overline{\ball}_{\vartheta}(\Omega)$ by the very definition of $\mathscr{E}_{\Omega}[u]$, cf. \eqref{def:Eomega}. This completes the proof. 
\end{proof}

\begin{lemma} \label{lemma:pointwise:solenoidal}
    Let $u \in \hold^{1}(\overline{\Omega};\wedgeq^{n-1}(\R^{n})^{*})$. Then there holds $\mathscr{E}_{\Omega}[u] \in \hold^{\infty}(\R^{n}\setminus\overline{\Omega};\wedgeq^{n-1}(\R^{n})^{*})$. Moreover, if we have $\dif  u = 0$ in $\Omega$, then also $ \dif\mathscr{E}_{\Omega}[u]=0$ in $\R\setminus\overline{\Omega}$. 
\end{lemma}

\begin{proof}
Let $y\in\R^{n}\setminus\overline{\Omega}$. By \ref{item:Whitney1} and \ref{item:Whitney2}, there exists an open neighbourhood $U\Subset\R^{n}\setminus\overline{\Omega}$ such that $U\cap \tfrac{7}{6}Q\neq\emptyset$ for at most $\mathtt{N}$ cubes $Q\in\mathcal{W}_{2}$. Hence, \ref{item:POU1} implies that the sum on the right-hand side of \eqref{def:Eomega} is locally finite in $\R^{n}\setminus\overline{\Omega}$. Hence, since each individual summand is smooth, so is the sum, and we may compute $\dif\mathscr{E}_{\Omega}[u]$ in $\R^{n}\setminus\overline{\Omega}$ by interchanging the exterior derivative and the sum defining $\mathscr{E}_{\Omega}[u]$. Given $\widetilde{I}=(i_1,..,i_{n+1}) \in \N^{n+1}$ we denote for $s\in\{1,...,n\}$ 
\begin{align*}
I_s=(i_1,...,i_{s-1},i_{n+1},i_{s+1},...,i_n)\;\;\;\text{and}\;\;\; I_{n+1} =(i_1,...,i_n).
\end{align*} 
We then have for $y\in\R^{n}\setminus\overline{\Omega}$
\begin{align*}
\dif\,(\mathscr{E}_{\Omega}[u](y)) & = (-1)^{n-1} \sum_{I=(i_1,...,i_n)\in\mathbb{N}^{n}} \Big(\dif \phi_{i_n} \wedge \dif \phi_{i_{n-1}} \wedge \dots \wedge \dif \phi_{i_1}(y)\Big. \\ & \Big. \;\;\;\;\;\;\;\;\;\;\;\;\;\;\;\;\;\;\;\;\cdot \int_{(\R^{n})^{n}} \fint_{\conv(\mathbf{x}_{I})} u(z) \nu(\mathbf{x}_{I}) \dHaus^{n-1}(z) ~\textup{d}\mu^I(\mathbf{x}_{I})\Big)\\ 
& = {(-1)^{n-1}} \sum_{\widetilde{I}=(i_1,...,i_n,i_{n+1})\in\mathbb{N}^{n+1}} \Big(\varphi_{i_{n+1}}\dif \phi_{i_n} \wedge \dif \phi_{i_{n-1}} \wedge \dots \wedge \dif \phi_{i_1}(y)\Big. \\ & \Big. \;\;\;\;\;\;\;\;\;\;\;\;\;\;\;\;\;\;\;\;\cdot \int_{(\R^{n})^{n}} \fint_{\conv(\mathbf{x}_{{I_{n+1}}})} u(z) \nu(\mathbf{x}_{{I_{n+1}}}) \dHaus^{n-1}(z) ~\textup{d}\mu^{{I_{n+1}}}(\mathbf{x}_{{I_{n+1}}})\Big) =:  \mathrm{I}, 
\end{align*}
where we have used in the last step that $(\varphi_{i_{n+1}})_{i_{n+1}\in\mathbb{N}}$ is a partition of unity in $\R^{n}\setminus\overline{\Omega}$, cf. \ref{item:POU2}. Next, we observe that 
\begin{align}\label{eq:vanish}
\begin{split}
\sum_{s=1}^{n}&\sum_{\widetilde{I}=(i_{1},...,i_{n+1})\in\mathbb{N}^{n+1}}\Big(\varphi_{i_{n+1}}\dif \phi_{i_n} \wedge \dif \phi_{i_{n-1}} \wedge \dots \wedge \dif \phi_{i_1}(y)\Big. \\ & \;\;\;\;\;\;\;\;\;\;\;\;\;\;\Big.\cdot \int_{(\R^{n})^{n}}\dashint_{\mathrm{conv}(\mathrm{x}_{I_{s}})}u(z) \nu(\mathbf{x}_{I_{s}}) \dHaus^{n-1}(z) ~\textup{d}\mu^{I_{s}}(\mathbf{x}_{I_{s}})\Big) = 0.
\end{split}
\end{align}
This is a consequence of the ultimate integral being independent of $s$,  $(\varphi_{i_{s}})_{i_{s}\in\mathbb{N}}$ being a partition of unity on $\R^{n}\setminus\overline{\Omega}$ and therefore 
\begin{align*}
\sum_{i_{s}\in\mathbb{N}}\dif\varphi_{i_{s}}=\dif\sum_{i_{s}\in\mathbb{N}}\varphi_{i_{s}} = 0\qquad\text{in}\;\R^{n}\setminus\overline{\Omega}, 
\end{align*}
from where \eqref{eq:vanish} follows. 
Hence, by \eqref{eq:vanish}, we arrive at 
\begin{align*}
\mathrm{I} & = (-1)^{n-1} \sum_{\widetilde{I}=(i_1,...,i_n,i_{n+1})\in\mathbb{N}^{n+1}} \Big(\varphi_{i_{n+1}}\dif \phi_{i_n} \wedge \dif \phi_{i_{n-1}} \wedge \dots \wedge \dif \phi_{i_1}(y)\Big.\\ & \;\;\;\;\;\;\;\;\;\;\;\;\;\;\;\;\;\;\;\;\;\;\;\; \cdot \Big(\int_{(\R^{n})^{n}} \fint_{\conv(\mathbf{x}_{I_{n+1}})} u(z) \nu(\mathbf{x}_{I_{n+1}}) \dHaus^{n-1}(z) ~\textup{d}\mu^{{I_{n+1}}}(\mathbf{x}_{I_{n+1}}) \Big. \\ & \Big. \;\;\;\;\;\;\;\;\;\;\;\;\;\;\;\;\;\;\;\;\;\;\;\;\;\;\;\;\;\;\;\;\;\;\;\;- \sum_{s=1}^{n}\int_{(\R^{n})^{n}}\dashint_{\mathrm{conv}(\mathrm{x}_{I_{s}})}u(z) \nu(\mathbf{x}_{I_{s}}) \dHaus^{n-1}(z) ~\textup{d}\mu^{I_{s}}(\mathbf{x}_{I_{s}})\Big) \Big) \\ 
& =: (-1)^{n-1} \sum_{\widetilde{I}=(i_1,...,i_n,i_{n+1})\in\mathbb{N}^{n+1}} \Big(\varphi_{i_{n+1}}\dif \phi_{i_n} \wedge \dif \phi_{i_{n-1}} \wedge \dots \wedge \dif \phi_{i_1} \Big)\mathrm{II}(\widetilde{I})
\end{align*}
with an obvious definition of $\mathrm{II}(\widetilde{I})$. We may now apply Gauss' theorem to the sum defining $\mathrm{II}(\widetilde{I})$, as the $(n+1)$-different $(n-1)$-dimensional simplices form the boundary of an $n$-dimensional simplex and $\nu_{I_s}$ is the corresponding normal vector. Since we have $\dif u=0$ in $\Omega$, we conclude that
\begin{align*}  
\mathrm{II}(\widetilde{I}) = \int_{(\R^{n})^{n+1}} \int_{\conv_{\tilde{I}}} \dif u(z) \nu(\mathbf{x}) ~\textup{d}z \dmu^{\tilde{I}}(\mathbf{x}) =0
\end{align*}
and therefore also $\mathrm{I}=0$. This completes the proof. 
\end{proof}
The next lemma states that the distributional exterior derivative can be represented by an $\lebe^1$-function and thus is a regular distribution. We explicitly mention that, in this lemma, we do not assume the differential form $u$ to be closed. Moreover, it is here where we require $u$ and $\D u$ to be continuous up to the boundary.
\begin{lemma} \label{lemma:globalsolenoidality}
Let $u \in \hold^{1}(\overline{\Omega};\wedgeq^{n-1}(\R^{n})^{*})$. Then there holds $\dif\,(\mathscr{E}_{\Omega}[u]) \in \lebe^1(\R^n;\wedgeq^{n}(\R^{n})^{*})$, meaning that there exists $h \in \lebe^1(\R^n;\wedgeq^{n}(\R^{n})^{*})$ such that we have 
\begin{align}
\int_{\R^{n}} \mathscr{E}_{\Omega}[u] \wedge \dif\psi \dx = \int_{\R^{n}} h \psi \dx\qquad \text{for all}\;\psi\in\hold_{c}^{\infty}(\R^{n}).
\end{align}
\end{lemma}
\begin{proof}
  The proof of this lemma is heavily inspired by  \cite[Lemma 4.7]{Schiffer} by the second named author. As $u\in\hold^{1}(\overline{\Omega};\wedgeq^{n-1}(\R^{n})^{*})$ is Lipschitz continuous up to the boundary, the classical \textsc{McShane} extension (cf. \textsc{Evans \& Gariepy} \cite[\S 3.1.1, Thm. 1]{EvansGariepy}) provides us with some  $v\in\sobo^{1,\infty}(\R^{n};\wedgeq^{n-1}(\R^{n})^{*})$ such that $v|_{\Omega}=u$ in $\Omega$. Cutting off $v$ far away from $\partial\Omega$ with a smooth cut-off function, it is then no loss of generality to assume that $v\in(\sobo^{1,1}\cap\sobo^{1,\infty})(\R^{n};\wedgeq^{n-1}(\R^{n})^{*})$. We note that $v$ does not necessarily satisfy $\dif v =0$.  
  
  We aim to show that there exists  $h\in\lebe^{1}(\R^{n})$ such that 
\begin{equation} \label{eq:steffenbaumgart}
\int_{\R^{n}\setminus\overline{\Omega}} \E_{\Omega} [u] \wedge \dif\psi \,\dif x = \int_{\R^{n}\setminus\overline{\Omega}} h \psi \dif x+ \int_{\R^{n}\setminus\overline{\Omega}} v \wedge \dif\psi  \,\dif x\qquad \text{for all}\;\psi\in\hold_{c}^{\infty}(\R^{n}).
\end{equation}
Using that $\dif v \in \lebe^{1}(\R^{n};\wedgeq^{n}(\R^{n})^{*})$, one then infers that $\dif \mathscr{E}_{\Omega}[u] \in \lebe^1(\R^{n};\wedgeq^{n}(\R^{n})^{*})$ as follows: First using that $\mathscr{E}_{\Omega}[u]|_{\Omega}=v|_{\Omega}=u$ in $\Omega$ and then employing \eqref{eq:steffenbaumgart}, we find 
\begin{align*}
\int_{\R^{n}}\mathscr{E}_{\Omega}[u]\wedge \dif\psi \,\dif x& = \int_{\R^{n}\setminus\overline{\Omega}}\mathscr{E}_{\Omega}[u]\wedge\dif\psi \,\dif x+ \int_{\Omega}v\wedge\dif\psi \,\dif x\stackrel{\eqref{eq:steffenbaumgart}}{=}  \int_{\R^{n}\setminus\overline{\Omega}}h\psi\dif x + \int_{\R^{n}}v\wedge\dif\psi\dif x.
\end{align*}
Using that $v\in(\sobo^{1,1}\cap\sobo^{1,\infty})(\R^{n};\wedgeq^{n-1}(\R^{n})^{*})$, an integration by parts in the last term and recalling that $h\in\lebe^{1}(\R^{n})$ then yields that $\dif \mathscr{E}_{\Omega}[u] \in \lebe^1(\R^{n};\wedgeq^{n}(\R^{n})^{*})$ as claimed. In order to establish \eqref{eq:steffenbaumgart}, we split the proof into two  steps.

\emph{Step 1: $\sobo_{0}^{1,1}(\R^{n}\setminus\overline{\Omega};\wedgeq^{n-1}(\R^{n})^{*})$-bounds for an auxiliary function.}
With the Lipschitz function $v$ introduced above, we  claim that the auxiliary function defined by 
\begin{align}\label{eq:thomastuchelciao}
\begin{split}
\mathcal{T} \colon y & \mapsto \sum_{I\in\mathbb{N}^{n}}\mathcal{T}_{I}(y) := \sum_{I\in\mathbb{N}^{n}}\phi_{i_n} \wedge \dif \phi_{i_{n-1}} \wedge \dots \wedge \dif \phi_{i_1}(y) T_{I}(y)\\ & := \sum_{I\in\mathbb{N}^{n}} \Big(\phi_{i_n} \wedge \dif \phi_{i_{n-1}} \wedge \dots \wedge \dif \phi_{i_1}(y)\Big. \\ & \;\;\;\;\;\;\;\;\;\;\;\;\;\;\;\;\;\;\;\;\;\;\;\;\;\;\;\;\;\;\;\;\;\Big. \cdot\int_{(\R^{n})^{n}} \fint_{\conv(\mathbf{x}_{I})} (v(z) - v(y)) \nu(\mathbf{x}_{I}) \dHaus^{n-1}(z) ~\textup{d}\mu^{I}(\mathbf{x}_{I})\Big) 
\end{split}
\end{align}
belongs to $\sobo_{0}^{1,1}(\R^{n}\setminus\overline{\Omega};\wedgeq^{n-1}(\R^{n})^{*})$, and recall that $v|_{\Omega}=u$ on $\Omega$. To this end, we establish that the series in \eqref{eq:thomastuchelciao}  is absolutely convergent in $(\sobo_{0}^{1,1}(\R^{n}\setminus\overline{\Omega};\wedgeq^{n-1}(\R^{n})^{*}),\|\cdot\|_{\sobo^{1,1}(\R^{n}\setminus\overline{\Omega})})$. In view of this aim, we note the following: 
\begin{enumerate} 
\item\label{item:dangerzone1} If $y\in\R^{n}\setminus\overline{\Omega}$ satisfies $\phi_{i_n} \wedge \dif \phi_{i_{n-1}} \wedge \dots \wedge \dif \phi_{i_1}(y)\neq 0$, then we have for all $\mathbf{x}_{I}=(x_{i_{1}},...,x_{i_{n}})\in \spt(\mu^{i_{1}})\times...\times\spt(\mu^{i_{n}})$ the inequality
\begin{align*}
\tfrac{1}{c}\ell(Q_{i_{n}})\leq |z-y|\leq c\,\ell(Q_{i_{n}})\qquad\text{for all}\;z\in \conv(\mathbf{x}_{I})
\end{align*}
with a universal constant $c>0$, cf. \ref{item:Jonesrefined4} from Lemma \ref{cor:Jones}. For such $y$ and $z$, we consequently have 
\begin{align}\label{eq:Lippo}
|v(y)-v(z)|\leq c\,\ell(Q_{i_{n}})\|\!\D v\|_{\lebe^{\infty}(\R^{n})}.
\end{align}
\item\label{item:dangerzone3} If $y\in\R^{n}\setminus\overline{\Omega}$ and $\mathbf{x}_{I}$ are as in \ref{item:dangerzone1}, we have with a universal constant $c>0$ that 
\begin{align*}
|\nu(\mathbf{x}_{I})|\leq c\ell(Q_{i_{n}})^{n-1}.
\end{align*}
\item\label{item:dangerzone2} As in Lemma \ref{lem:compactsupport}, we see that $\mathcal{T}$ is compactly supported and that $\spt(\mathcal{T})$ is contained in a ball $\ball_{R}(0)$ only depending on $\Omega$, specifically $\diameter(\Omega)$ and $\Lip(\Omega,\mathcal{O})$.

\end{enumerate}
Recalling the neighbouring cube notation from \eqref{eq:peaceintheneighbourhood} and that $\mu^{I}$ is a probability measure on $(\R^{n})^{n}$, \ref{item:dangerzone1}--\ref{item:dangerzone3} and \ref{item:POU3} imply that 
\begin{align}\label{eq:thirtydays}
\begin{split}
|\mathcal{T}_{I}(y)| & \leq c\,\frac{1}{\ell(Q_{i_{n}})^{n-1}}\mathbbm{1}_{\bigcup\{Q\colon\;Q\in\mathcal{N}(i_{n})\}}(y)\ell(Q_{i_{n}})^{{n}}\|\!\D v\|_{\lebe^{\infty}(\R^{n})}\\ & \leq c\,{\ell(Q_{i_{n}})}\mathbbm{1}_{\bigcup\{Q\colon\;Q\in\mathcal{N}(i_{n})\}}(y)\|\!\D v\|_{\lebe^{\infty}(\R^{n})}.
\end{split}
\end{align}
Similarly, for the estimation of the $\lebe^{1}$-norm of the first order derivatives, we employ the Leibniz rule to find by use of \ref{item:dangerzone1}--\ref{item:dangerzone3} and \ref{item:POU3}:
\begin{align}\label{eq:manonthemoon}
\begin{split}
|\D \mathcal{T}_{I}(y)| & = \big\vert\!\D\,(\phi_{i_n} \wedge \dif \phi_{i_{n-1}} \wedge \dots \wedge \dif \phi_{i_1}(y))\big\vert\,|T_{I}(y)| \\ 
& + \big\vert (\phi_{i_n} \wedge \dif \phi_{i_{n-1}} \wedge \dots \wedge \dif \phi_{i_1}(y))\big\vert\,|\!\D T_{I}(y)|\\ 
& \leq c\,\frac{1}{\ell(Q_{i_{n}})^{n}}\mathbbm{1}_{\bigcup\{Q\colon\;Q\in\mathcal{N}(i_{n})\}}(y)\ell(Q_{i_{n}})^{n}\|\!\D v\|_{\lebe^{\infty}(\R^{n})} \\ 
& + c\,\frac{1}{\ell(Q_{i_{n}})^{n-1}}\mathbbm{1}_{\bigcup\{Q\colon\;Q\in\mathcal{N}(i_{n})\}}(y)\ell(Q_{i_{n}})^{n-1}\|\!\D v\|_{\lebe^{\infty}(\R^{n})} \\ 
& \leq c \mathbbm{1}_{\bigcup\{Q\colon\;Q\in\mathcal{N}(i_{n})\}}(y) \|\!\D v\|_{\lebe^{\infty}(\R^{n})}.
\end{split}
\end{align}
By \ref{item:Whitney2}, for each $i_{n}\in\mathbb{N}$, $\mathcal{N}(i_{n})$ has cardinality uniformly bounded by some $\mathtt{N}=\mathtt{N}(n)\in\mathbb{N}$. 
Combining the pointwise estimates \eqref{eq:thirtydays},  \eqref{eq:manonthemoon} and recalling the compact support of $\mathcal{T}_{I}$ as asserted in \ref{item:dangerzone2}, we then arrive at 
\begin{align*}
\sum_{I=(i_{1},...,i_{n})\in\mathbb{N}^{n}}\|\mathcal{T}_{I}\|_{\sobo^{1,1}(\R^{n}\setminus\overline{\Omega})} & = \sum_{I=(i_{1},...,i_{n})\in\mathbb{N}^{n}}\|\mathcal{T}_{I}\|_{\sobo^{1,1}(\ball_{R}(0)\setminus\overline{\Omega})}\\ & \!\!\!\!\!\!\!\! = \sum_{i_{n}\in\mathbb{N}}\sum_{\substack{(i_{1},...,i_{n-1})\in\mathbb{N}^{n-1}\\ \frac{7}{6}Q_{i_{1}},...,\frac{7}{6}Q_{i_{n-1}}\in\mathcal{N}(i_{n})}}\|\mathcal{T}_{I}\|_{\sobo^{1,1}(\ball_{R}(0)\setminus\overline{\Omega})}\\ 
& \!\!\!\!\!\!\!\!\!\!\!\!\!\!\!\!\!\! \stackrel{\eqref{eq:thirtydays},\,\eqref{eq:manonthemoon}}{\leq} c(1+R)\Big(\sum_{i_{n}\in\mathbb{N}} \Big\vert\!\Big\vert\mathbbm{1}_{\bigcup\{Q\colon\;Q\in\mathcal{N}(i_{n})\}}\Big\vert\!\Big\vert_{\lebe^{1}(\ball_{R}(0)\setminus\overline{\Omega})} \Big)\|\!\D v\|_{\lebe^{\infty}(\R^{n})}\\ 
& \!\!\!\!\!\!\!\!\!\leq c(1+R)R^{n} \|\!\D v\|_{\lebe^{\infty}(\R^{n})}. 
\end{align*}
Since we have $\mathcal{T}_{I}\in\sobo_{0}^{1,1}(\R^{n}\setminus\overline{\Omega};\wedgeq^{n-1}(\R^{n})^{*})$ for each $I=(i_{1},...,i_{n})\in\mathbb{N}^{n}$, we then conclude that the series defining $\mathcal{T}$ converges absolutely in $(\sobo_{0}^{1,1}(\R^{n}\setminus\overline{\Omega};\wedgeq^{n-1}(\R^{n})^{*}),\|\cdot\|_{\sobo^{1,1}(\R^{n}\setminus\overline{\Omega})})$ as claimed. 

Denoting by $\overline{x}_i$ the center of the inner cube $Q^i$, we may use the linearity of the normal to write
 \[
 \int_{(\R^n)^{n}}
 \nu(\mathbf{x}_{I}) \dif\mu^I(\mathbf{x}_I) = \nu(\overline{\mathbf{x}}_I),
 \]
where $\overline{\mathbf{x}}_I=(\overline{x}_{i_1},...,\overline{x}_{i_n})$. Now, because $\mathcal{T}\in\sobo_{0}^{1,1}(\R^{n}\setminus\overline{\Omega};\wedgeq^{n-1}(\R^{n})^{*})$ and $v|_{\Omega}=u$, we have that 
\begin{align*}
\int_{\R^{n}\setminus\overline{\Omega}} \mathscr{E}_{\Omega}[u]\wedge \dif\psi {\,\dif y}& = {(-1)^{n-1}}\int_{\R^{n}\setminus\overline{\Omega}}\mathcal{T}\wedge\dif\psi(y) {\,\dif y}\\ 
& + {(-1)^{n-1}}\int_{\R^{n}\setminus\overline{\Omega}}\sum_{I\in\mathbb{N}^{n}} (\varphi_{i_{n}}\wedge...\wedge\dif\varphi_{i_{1}}(y))(v(y)\nu(\overline{\mathbf{x}}_{I}))\wedge\dif\psi(y) {\,\dif y}\\ 
& ={(-1)^{n-1}}\int_{\R^{n}\setminus\overline{\Omega}} h_{0}(y)\psi(y)\dif y \\ 
& + {(-1)^{n-1}}\int_{\R^{n}\setminus\overline{\Omega}}\sum_{I\in\mathbb{N}^{n}} (\varphi_{i_{n}}\wedge...\wedge\dif\varphi_{i_{1}}(y))(v(y)\nu(\overline{\mathbf{x}}_{I}))\wedge\dif\psi(y){\,\dif y}. 
\end{align*}
for the function $h_0 = -\dif\mathcal{T} \in \lebe^1(\R^n;\wedgeq^{n}(\R^{n})^{*})$.

\emph{Step 2: Inductive integration by parts.} To conclude the proof, we claim that there exists a function $h_1 \in \lebe^1(\R^n;\wedgeq^{n}((\R^n)^{\ast})$ such that 
\begin{align} \label{claim:sb}
\begin{split}
(-1)^{n-1}\int_{\R^{n}\setminus\overline{\Omega}} &\sum_{I\in\mathbb{N}^{n}} \phi_{i_n} \wedge \dif \phi_{i_{n-1}} \wedge \dots \wedge \dif \phi_{i_1} (v(y) \nu(\overline{\mathbf{x}}_I)) \wedge \dif\psi {\,\dif y}\\ & = \int_{\R^{n}\setminus\overline{\Omega}} h_1 \psi  + v \wedge \dif\psi {\,\dif y}.
\end{split}
\end{align}
We prove \eqref{claim:sb} by an inductive argument. Namely, we establish that for any $r\in\{2,...,n\}$ there exists an $\lebe^1$-function $h_r$ such that
\begin{equation}\label{claim:sb2}\tag{$\mathrm{D}_{r}$}
\begin{split}
&\int_{\R^{n}\setminus\overline{\Omega}} \sum_{I=(i_1,...,i_r)\in\mathbb{N}^{r}} \Big(\phi_{i_r} \wedge \dif \phi_{i_{r-1}} \wedge \dots \wedge \dif \phi_{i_1} (v(y) \nu(\overline{\mathbf{x}}_I)) \wedge \dif \psi {\,\dif y}= \int_{\R^{n}\setminus\overline{\Omega}} h_r \psi {\,\dif y}
\\
& - \frac{(n+1)-r}{r-1}\int_{\R^{n}\setminus\overline{\Omega}} \sum_{\widetilde{I}=(i_1,...,i_{r-1})\in\mathbb{N}^{r-1}} \Big(\phi_{i_{r-1}} \wedge \dif \phi_{i_{r-2}} \wedge \dots \wedge \dif \phi_{i_1} (v(y) \nu(\overline{\mathbf{x}}_{\widetilde{I}}))\Big) \wedge \dif\psi {\,\dif y}.
\end{split}
\end{equation}
where $\nu$ is given as in \eqref{eq:normaldefine} and we tacitly define the normal $\nu(x) :=1$ in dimension zero. Once \eqref{claim:sb2} is established for all $r\in\{2,...,n\}$, it suffices to realise that, for $r=n$ the left-hand side of $(\mathrm{D}_{n})$ equals the left-hand side of \eqref{claim:sb} up to a prefactor of $(-1)^{n-1}$. Turning to right-hand side of $(\mathrm{D}_{n})$, we then keep the integral containing $h_{n}$ and apply $(\mathrm{D}_{n-1})$ to the second integral appearing on the right-hand side of $(\mathrm{D}_{n})$. Inductively proceeding in this way, we then conclude \eqref{claim:sb}: Indeed, the inductive appearance of the integrals containing $h_{r}$ yields the integral containing $h$ in \eqref{claim:sb}, and the corresponding prefactors in front of the second integrals on the right-hand sides of \eqref{claim:sb2} lead to the term 
\begin{align*}
    \prod_{r=2}^{n} \Big(-\frac{(n+1)-r}{r-1}\Big) = (-1)^{n-1} \frac{(n-1)!}{(n-1)!} = (-1)^{(n-1)}, 
\end{align*}
which is precisely the prefactor appearing in \eqref{claim:sb}.

It remains to show \eqref{claim:sb2} for $r\in\{2,...,n\}$. Given $I=(i_{1},...,i_{r})\in\mathbb{N}^{r}$ and $x_{i_{1}},...,x_{i_{r}}\in\R^{n}$, we have for any $y\in\R^{n}$
\[
\nu(\mathbf{\overline{x}}_I) := \int_{(\R^{n})^{r}}\nu(\mathbf{x}_{I})\dif\mu^{I}(\mathbf{x}_{I}) = \nu_{1}^{{y}}(\mathbf{\bar{x}}_I) + ... + \nu_{r}^{{y}}(\mathbf{\bar{x}}_I),
\]
where, for $1 \leq s \leq r$, 
\[
\nu_{s}^{{y}}(\mathbf{\bar{x}}_I) := \nu(\bar{x}_{i_1},...,\bar{x}_{i_{s-1}},y,\bar{x}_{i_{s+1}},...,\bar{x}_{i_r}).
\]
We also write for $\widetilde{I}=(i_1,...,i_{r-1})\in\mathbb{N}^{r-1}$ 
\[
\nu_{r}^{y}(\mathbf{\overline{x}}_{\widetilde{I}}) := \nu(\overline{x}_{i_1},...,\overline{x}_{i_{r-1}},y).
\]
By definition,  $\nu_{s}^{y}$ \emph{does not} depend on the specific choice of the value $\overline{x}_{i_s}$. Using that $(\phi_{i_s})_{i_{s}\in\mathbb{N}}$ is a partition of unity on $\R^{n}\setminus\overline{\Omega}$ and therefore $\sum_{i_s \in \N} \dif\phi_{i_s} =0$ on $\R^{n}\setminus\overline{\Omega}$, we obtain
\begin{align*}
\int_{\R^{n}\setminus\overline{\Omega}} &\sum_{I=(i_1,...,i_r)\in\mathbb{N}^{r}} \phi_{i_r} \wedge \dif \phi_{i_{r-1}} \wedge \dots \wedge \dif \phi_{i_1} (v(y) \nu(\overline{\mathbf{x}}_I)) \wedge \dif\psi {\,\dif y}
\\
&= \int_{\R^{n}\setminus\overline{\Omega}} \sum_{I=(i_1,...,i_r)\in\mathbb{N}^{r}} \phi_{i_r} \wedge \dif \phi_{i_{r-1}} \wedge \dots \wedge \dif \phi_{i_1} (v(y) \nu_{r}^{{y}}(\overline{\mathbf{x}}_{\widetilde{I}})) \wedge \dif\psi {\,\dif y}\\
&=  \int_{\R^{n}\setminus\overline{\Omega}} \sum_{\tilde{I}=(i_1,...,i_{r-1})\in\mathbb{N}^{r-1}} \dif \phi_{i_{r-1}} \wedge \dots \wedge \dif \phi_{i_1} (v(y)\nu_{r}^{{y}}(\bar{\mathbf{x}}_{\widetilde{I}})) \wedge \dif\psi {\,\dif y}.
\end{align*}
An integration by parts then yields
\begin{equation}
 \label{timoschultz}
 \begin{split}
\int_{\R^{n}\setminus\overline{\Omega}} &\sum_{\widetilde{I}=(i_1,...,i_{r-1})\in\mathbb{N}^{r-1}} \dif \phi_{i_{r-1}} \wedge \dots \wedge \dif \phi_{i_1} (v(y) \nu_{r}^{y} (\overline{\mathbf{x}}_{\widetilde{I}})) \wedge \dif \psi {\,\dif y}\\
 &= - \int_{\R^{n}\setminus\overline{\Omega}} \sum_{\widetilde{I}=(i_1,...,i_{r-1})\in\mathbb{N}^{r-1}} \phi_{i_{r-1}} \wedge \dots \wedge \dif \phi_{i_1} (\dif v(y) \nu_{r}^{y}(\overline{\mathbf{x}}_{\widetilde{I}})) \wedge \dif \psi {\,\dif y}\\
 & - \int_{\R^{n}\setminus\overline{\Omega}} \sum_{\widetilde{I}=(i_1,...,i_{r-1})\in\mathbb{N}^{r-1}} \phi_{i_{r-1}} \wedge \dots \wedge \dif \phi_{i_1} (v(y) \dif^{\ast}\nu_{r}^{y}(\overline{\mathbf{x}}_{\widetilde{I}})) \wedge \dif\psi {\,\dif y}.
\end{split}
\end{equation}
Again, one can argue that the first summand converges absolutely in $\sobo^{1,1}$, as 
\begin{equation} \label{w11}
\vert  \phi_{i_{r-1}} \wedge \dots \wedge \dif \phi_{i_1} \vert \leq C \delta^{r-2}, \quad \vert \nu_r(\mathbf{\bar{x}}_{\tilde{I}}) \vert \leq C \delta^{r-1},
\end{equation}
and we may prove that the derivative of each summand
\[
\D\,( \phi_{i_{r-1}} \wedge \dots \wedge \dif \phi_{i_1} (\dif v(y) \nu_r(\bar{\mathbf{x}}_{\tilde{I}}))
\]
is uniformly bounded in $\lebe^{\infty}$, i.e. the sum converges in $\lebe^1$ (the $\lebe^1$ norm of each summand is bounded by the the $\lebe^{\infty}$-norm times the volume of $Q_{i_1}$).
One the other hand, (for example by a coordinatewise calculation\footnote{
For instance, assuming $v= \dif x_1 \wedge ... \wedge \dif x_{n-1}$ and $\nu = (y-x_{i_r}) \wedge e_{s} \wedge .... \wedge e_{n-1}$ one may calculate that $u d^{\ast} \nu = (s-1) \dif x_1 \wedge ... \wedge \dif x_{s-1}$.
}) we have
\begin{align*}
    v(y) \dif^{\ast} \nu_r(\mathbf{\overline{x}}_{\tilde{I}}) 
    &= \frac{(n+1-r)}{(r-1)!} v(y) \left((x_{i_{r-1}} - x_{i_{r-2}}) \wedge \ldots \wedge (x_{i_{2}}-x_{i_{1}}))\right)  \\
    &=\frac{(n+1-r)}{(r-1)} v(y) \nu_{r-1}(\mathbf{\overline{x}}_{\tilde{I}})
\end{align*}
Plugging this into the previous formula yields \eqref{claim:sb2}.
\end{proof}
Combining the previous two lemmas yields the desired solenoidality:
\begin{corollary} \label{coro:divfree}
    Let $u \in \hold^{1}(\overline{\Omega};\wedgeq^{n-1}(\R^{n})^{*})$ satisfy $\dif u=0$ in $\Omega$. Then we have $\dif \mathscr{E}_{\Omega}[u]=0$ in $\mathscr{D}'(\R^n;\wedgeq^{n}(\R^{n})^{*})$.
\end{corollary}
\begin{proof}
By Lemma \ref{lemma:pointwise:solenoidal}, we have $\dif\mathscr{E}_{\Omega}[u]=0$ pointwisely in $\R^{n}\setminus\overline{\Omega}$.  Moreover, by Lemma \ref{lemma:globalsolenoidality}, the distributional expression  $\dif\mathscr{E}_{\Omega}[u]$ can be represented by an $\lebe^{1}(\R^{n};\wedgeq^{n}(\R^{n})^{*})$-map. Since  the $\mathscr{L}^{n}$-measure of the boundary $\partial \Omega$ is zero (for $\Omega$ is convex and thus has Lipschitz boundary), we obtain $\dif\mathscr{E}_{\Omega}[u]=0$ in $\mathscr{D}'(\R^{n};\wedgeq^{n}(\R^{n})^{*})$ as claimed.
\end{proof}
The last step is to show a uniform $\lebe^1$-bound on the extension $\mathscr{E}_{\Omega}[u]$. For this, we require a preparatory lemma as follows: 
\begin{lemma} \label{lem:convhullL1bound}
There is a constant $c>0$ depending on the dimension $n$ and $\Lip(\Omega,\mathcal{O})$, such that the following holds: If $I=(i_1,\ldots,i_n) \in \mathbb{N}^n$ is such that ${\frac{7}{6}}Q_{i_1},\ldots,{\frac{7}{6}}Q_{i_n}$ intersect, then for any $u\in\hold^{1}(\overline{\Omega};\wedgeq^{n-1}(\R^{n})^{*})$ there holds 
\begin{align}\label{eq:convhullbound1}
\left\vert \int_{(\R^{n})^{n}} \fint_{\conv(\mathbf{x}_{I})} u(z) \nu(\mathbf{x}_{I}) \dHaus^{n-1}(z) ~\textup{d}\mu^I(\mathbf{x}_{I}) \right\vert \leq \frac{c}{\delta_{I}}\int_{\conv_{I}}|u(x)|\dif x, 
\end{align}
where $\conv_{I}:=\conv( \tfrac{1}{2}Q^{i_{1}},..., \tfrac{1}{2}Q^{i_{n}})$ and
\begin{align}\label{eq:deltadef}
\delta_{I}:=\min\{\ell(Q_{i_{j}})\colon\;j\in\{1,...,n\}\}.
\end{align}
\end{lemma} 
\begin{proof}
Since $\frac{7}{6}Q_{i_{1}}\cap...\cap\frac{7}{6}Q_{i_{n}}\neq\emptyset$, \ref{item:Whitney3} implies that the sidelenghts of all cubes $Q_{i_{j}}$, $j\in\{1,...,n\}$, are uniformly comparable. If $\mathbf{x}_{I}$ belongs to the support of $\mu^{I}$, their distance is bouned by $c \delta_I$ and thus 
\begin{align}\label{eq:leroysane}
\mathscr{H}^{n-1}(\conv(\mathbf{x}_{I}))\leq c\delta_{I}^{n-1}. 
\end{align}
Since $\mathscr{L}^{n}(\conv_{I})$ moreover is uniformly proportional to $\delta_{I}^{n}$, inequality \eqref{eq:convhullbound1} is seen to scale in the right way. We now rewrite the inner integral  on the left-hand side of \eqref{eq:convhullbound1} as follows: We put \begin{align*}
\widetilde{\mathbb{D}}^{n-1} := \{ t=(t_{1},...,t_{n})\in [0,1]^n \colon\;t_{1}+...+t_{n}=1 \}
\end{align*}
and then employ the change of variables $\Phi_{\mathbf{x}_{I}}\colon \widetilde{\mathbb{D}}^{n-1}\ni (t_{1},...,t_{n})\mapsto t_{1}x_{i_{1}}+...+t_{n}x_{i_{n}}\in\mathrm{conv}(\mathbf{x}_{I})$ to obtain 
 \begin{equation} \label{eq:trafo}
 \begin{split}
 \dashint_{\conv(\mathbf{x}_{I})} \vert u(z) \nu(\mathbf{x}_{I}) \vert \dHaus^{n-1}(z) & = \mathscr{H}^{n-1}(\mathrm{conv}(\mathbf{x}_{I}))\int_{\widetilde{\mathbb{D}}^{n-1}} |u(\Phi_{\mathbf{x}_{I}}(t))| \dHaus^{n-1}(t) \\ 
 & \!\!\!\!\stackrel{\eqref{eq:leroysane}}{\leq} c\delta_{I}^{n-1}\int_{\widetilde{\mathbb{D}}^{n-1}} |u(\Phi_{\mathbf{x}_{I}}(t))| \dHaus^{n-1}(t).
 \end{split}
\end{equation}
Now consider for $k\in\{1,...,n\}$ \begin{align}\label{eq:thomasmueller}
\begin{split}
\widetilde{\mathbb{D}}^{n-1}_{k}= \{t=(t_{1},...,t_{n}) \in \widetilde{\mathbb{D}}^{n-1} \colon t_k \geq \tfrac{1}{n} \},\;\;\;\text{so that}\;\;\;\widetilde{\mathbb{D}}^{n-1} = \bigcup_{k=1}^n \widetilde{\mathbb{D}}^{n-1}_{k}.
\end{split}
\end{align}
By the definition of $\mu^{i_{k}}$, see \eqref{eq:mudefMAIN} ff., $\mu^{i_{k}}$ satisfies with a purely dimensional constant $c>0$
\begin{align*}
0\leq \mu^{i_{k}} \leq \frac{c}{\delta_{I}^{n}}\mathscr{L}^{n}\mres \mathrm{conv}_{I}\qquad\text{for all}\;k\in\{1,...,n\},
\end{align*}
and so we obtain for all $k\in\{1,...,n\}$ and any fixed $\mathbf{x}_{I^{(k)}}:=(x_{i_{1}},...,x_{i_{k-1}},x_{i_{k+1}},...,x_{i_{n}})\in(\R^{n})^{n-1}$
\begin{align*}
\int_{\R^{n}}\int_{\widetilde{\mathbb{D}}_{k}^{n-1}}|u(\Phi_{\mathbf{x}_{I}}(t))|\dif\mathscr{H}^{n-1}(t)\dif\mu^{i_{k}}(x_{i_{k}}) & = \int_{\widetilde{\mathbb{D}}_{k}^{n-1}}\int_{\R^{n}}|u(\Phi_{\mathbf{x}_{I}}(t))|\dif\mu^{i_{k}}(x_{i_{1}})\dif\mathscr{H}^{n-1}(t) \\ 
& \leq \frac{1}{\delta_{I}^{n}}\int_{\widetilde{\mathbb{D}}_{k}^{n-1}}\frac{1}{t_{k}^{n}}\int_{\mathrm{conv}_{I}}|u(z)|\dif z \dif\mathscr{H}^{n-1}(t)\\ 
& \leq \frac{c(n)}{\delta_{I}^{n}} \int_{\mathrm{conv}_{I}}|u(z)|\dif z. 
\end{align*}
Since $\mu^{i_{1}},...,\mu^{i_{k-1}},\mu^{i_{k+1}},...,\mu^{i_{n}}$ are probability measures, we may integrate the preceding inequality with respect to the product measure $\mu^{I^{(k)}}$ and arrive at 
\begin{align*}
\int_{(\R^{n})^{n-1}}\int_{\R^{n}}\int_{\widetilde{\mathbb{D}}_{k}^{n-1}}|u(\Phi_{\mathbf{x}_{I}}(t))|\dif\mathscr{H}^{n-1}(t)\dif\mu^{i_{k}}(x_{i_{k}})\dif\mu^{I^{(k)}}(\mathbf{x}_{I^{(k)}})\leq \frac{c(n)}{\delta_{I}^{n}} \int_{\mathrm{conv}_{I}}|u(z)|\dif z.
\end{align*}
Summing this inequality over all $k\in\{1,...,n\}$ and recalling the ultimate identity recorded in \eqref{eq:thomasmueller}, we then conclude \eqref{eq:convhullbound1} in view of \eqref{eq:trafo}. This completes the proof. 
\end{proof}

\begin{figure}
\begin{center}
\includegraphics[scale=0.35]{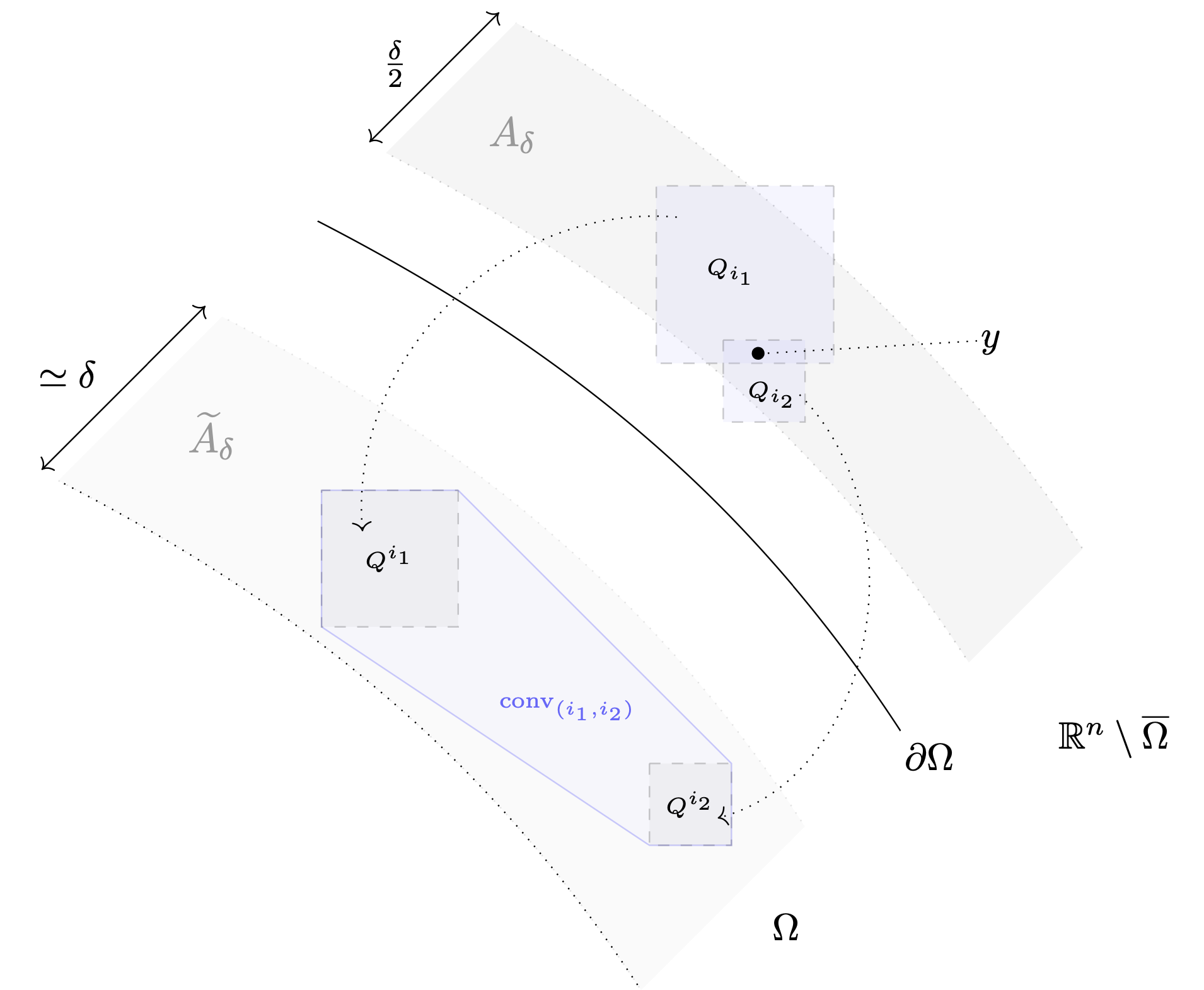} 
\end{center}
\caption{On the strategy of proof of Lemma \ref{lemma:L1bound:c1}, schematically displayed for $n=2$ dimensions.}
\end{figure}
We next turn to the uniform $\lebe^{1}$-bounds required as a main ingredient of Proposition \ref{prop:intermediateL1convex}:
\begin{lemma}[Uniform $\lebe^{1}$-bounds] \label{lemma:L1bound:c1}
    Let $u \in \hold^1(\overline{\Omega};\wedgeq^{n-1}(\R^{n})^{*})$. Then we have 
    \begin{align*}
    \mathscr{E}_{\Omega}[u] \in \lebe^1(\R^n;{\wedgeq}^{n-1}(\R^{n})^{*}).
    \end{align*}
    Specifically, there exists constants $c_1,c_2,c_3>0$, only depending on the underlying space dimension $n\in\mathbb{N}$ and, for a given finite cover $\Ocal$ of $\overline{\Omega}$, on $\Lip(\mathcal{O},\Omega)$ (cf. \eqref{eq:nowandthen1}) and $\mathrm{M}:=\min_{O\in\Ocal}\mathrm{diam}(O)$, such that the following hold:
\begin{enumerate}
\item \label{item:L1bound0} There is a constant $\mathtt{c}>0$ and some $\delta>0$ such that $0<\delta < \delta_0$ implies
\begin{equation} \label{eq:strip}
    \int_{\{x \in \R^n \setminus \bar{\Omega} \colon \delta/2 < \dist(x,\partial \Omega) < \delta \}} \vert \E_{\Omega}[u] \vert \dif x \leq c_1 \int_{\{x \in \Omega \colon \delta/(2\mathtt{c}) < \dist(x,\partial \Omega < \mathtt{c}\delta\}} \vert u \vert \dif x;
\end{equation}
\item\label{item:L1bound1}  $\displaystyle 
            \Vert \mathscr{E}_{\Omega}[u] \Vert_{\lebe^1(\R^n)} \leq c_2 \Vert u \Vert_{\lebe^1(\Omega)};
         $
\item\label{item:L1bound2} there exists a constant $\mathtt{c}>0$  and $\delta_{0}>0$ such that $0<\delta<\delta_{0}$ implies 
\begin{align}\label{eq:stripsum}
\int_{\{x\in\R^{n}\setminus\overline{\Omega}\colon\;\mathrm{dist}(x,\partial\Omega)<\delta\}}|\mathscr{E}_{\Omega}[u]|\dif x \leq c_3 \int_{\{x\in\Omega\colon\;\mathrm{dist}(x,\partial\Omega)<\mathtt{c}\delta\}}|u|\dif x.
\end{align}
\end{enumerate}
\end{lemma}

\begin{proof}
We pick the number $\vartheta$ from Lemma \ref{lem:compactsupport}, cf. \eqref{eq:michaeljackson}. In particular, we have $\spt(\mathscr{E}_{\Omega}[u])\subset\overline{\ball}_{\vartheta}(\Omega)$, and thus we may thus focus on estimating $\|\mathscr{E}_{\Omega}[u]\|_{\lebe^{1}(\ball_{\vartheta}(\Omega))}$ in the sequel. To this end, we define exterior strips as follows: 
\begin{align}\label{eq:strips}
\mathtt{S}_{k}:=\ball_{2^{-k+1}\vartheta}(\Omega)\setminus\overline{\ball}_{2^{-k-1}\vartheta }(\Omega),\qquad k\in\mathbb{N}.
\end{align}
Then we have $\bigcup_{k\in\mathbb{N}}\mathtt{S}_{k}=\ball_{\vartheta}(\Omega)\setminus\overline{\Omega}$, and at most two of the $\mathtt{S}_{k}$'s have non-trivial intersection. The key idea of the remaining part of the proof then is as follows: If a cube $Q\in\mathcal{W}_{2}$ is far away from $\partial\Omega$ but close to $\partial\!\ball_{\vartheta}(\Omega)$, $\Psi$ as given by Lemma \ref{cor:Jones} maps $Q$ to a cube in the center or bulk part of $\Omega$. If it is, however, close to $\partial\Omega$, then the convex hull of its image under $\Psi$ belongs to a strip inside $\Omega$ and close to $\partial\Omega$. This shows part \ref{item:L1bound0}. These strips, in turn, have a uniformly finite overlap. This observation allows us to sum up the $\lebe^{1}$-norms of $u$ over these strips and thereby obtain the requisite estimate of the $\lebe^{1}$-norm of $\mathscr{E}_{\Omega}[u]$ close to $\partial\Omega$, i.e. \ref{item:L1bound1} and \ref{item:L1bound2}. \smallskip

To this end, we claim that there exist numbers $0<\mathtt{c}_{0}<\mathtt{c}_{1}$ only depending on $n$ and  $\Lip(\Omega,\mathcal{O})$ such that the following holds for all $k\in\mathbb{N}$: If $I=(i_{1},...,i_{n})\in\mathbb{N}^{n}$ is such that $\mathtt{S}_{k}\cap\frac{7}{6}Q_{i_{1}}\cap...\cap\frac{7}{6}Q_{i_{n}}\neq\emptyset$ and $\ell(Q_{i_{j}})\leq \eta$ holds for all $j\in\{1,...,n\}$, then 
\begin{align}\label{eq:uniformity}
\conv\Big(\bigcup_{j=1}^{n}\tfrac{1}{2}\Psi(Q_{i_{j}})\Big) \subset \{x\in\Omega\colon\;\mathtt{c}_{0}2^{-k} \leq \dista(x,\partial\Omega)\leq \mathtt{c}_{1}2^{-k}\} =: \mathtt{S}'_{k}. 
\end{align}
We note that the sets $\mathtt{S}'_{k}$ are only proper \emph{strips} for sufficiently large $k$.

To see \eqref{eq:uniformity}, we first note that if $z\in\mathtt{S}_{k}\cap\frac{7}{6}Q_{i_{1}}\cap...\cap\frac{7}{6}Q_{i_{n}}\neq\emptyset$, then we have for all $j\in\{1,...,n\}$, all $\xi\in\partial\Omega$ and all $y\in \frac{7}{6}Q_{i_{j}}$
\begin{align*}
2^{-k-1}\vartheta & \stackrel{z\in\mathtt{S}_{k}}{\leq} \dista(z,\partial\Omega) \leq |z-\xi| \leq |z-y|+|y-\xi| \leq \tfrac{7}{6}\sqrt{n}\ell(Q_{i_{j}}) + |y-\xi|.
\end{align*}
First infimising the right-hand side in $\xi\in\partial\Omega$ and then in $y\in\frac{7}{6}Q_{i_{j}}$ yields that 
\begin{align}\label{eq:stereolove}
\begin{split}
2^{-k-1}\vartheta & \leq \tfrac{7}{6}\sqrt{n}\ell(Q_{i_{j}}) + \dista(\tfrac{7}{6}Q_{i_{j}},\partial\Omega) \stackrel{\text{\ref{item:Whitney3}}}{\leq} c\,\ell(Q_{i_{j}}) \\ & \!\!\!\stackrel{\text{\ref{item:Whitney3}}}{\leq} c\,\dista({\tfrac{7}{6}Q_{i_{j}},\partial\Omega) \stackrel{z\in \frac{7}{6}}Q_{i_{j}}}\leq c\,\dist(z,\partial\Omega)  \stackrel{z\in\mathtt{S}_{k}}{\leq} c\,2^{-k+1}\vartheta 
\end{split}
\end{align}
with a constant $c=c(n)>0$. Based on \ref{A:2} and $\ell(Q_{i_{j}})\leq\eta$, we infer that 
\begin{align}\label{eq:kiss}
\begin{split}
2^{-k-1}\vartheta & \stackrel{\eqref{eq:stereolove}}{\leq} c\,\ell(Q_{i_{j}})  \stackrel{\text{\ref{A:2}}}{\leq} c\,\ell(\Psi(Q_{i_{j}})) \stackrel{\ref{item:Jonesrefined3a}}{\leq} c\,\dista(\Psi(Q_{i_{j}}),\partial\Omega) \\ 
& \;\;\leq c\,\dista(\tfrac{1}{2}\Psi(Q_{i_{j}}),\partial\Omega) \leq c\big(\tfrac{1}{4}\sqrt{n}\ell(\Psi(Q_{i_{j}}))+\dista(\Psi(Q_{i_{j}}),\partial\Omega)\big) \\ 
& \!\!\!\!\!\stackrel{\ref{A:2},\,\ref{A:3}}{\leq} c\,\ell(Q_{i_{j}}) \stackrel{\eqref{eq:stereolove}}{\leq} c\,2^{-k+1}\vartheta, 
\end{split}
\end{align}
and here we have $c=c(n,\Lip(\Omega,\mathcal{O}))>0$. If instead $\ell(Q_{i_j}) > \eta$, we have $\psi(Q_{i_j})= Q_0'$. From here we conclude that there exists a constant $\mathtt{c}=\mathtt{c}(n,\Lip(\Omega,\mathcal{O}),\eta,Q_0')>1$ such that 
\begin{align*}
\tfrac{1}{2}\Psi(Q_{i_{1}}),...,\tfrac{1}{2}\Psi(Q_{i_{n}})\subset \Big\{x\in\Omega\colon\;\tfrac{1}{c}\vartheta 2^{-k} \leq \dista(x,\partial\Omega)\leq c\vartheta 2^{-k}\Big\}. 
\end{align*}
On the other hand, if $\mathtt{S}_{k}\cap\frac{7}{6}Q_{i_{1}}\cap...\cap\frac{7}{6}Q_{i_{n}}\neq\emptyset$, then all of $\tfrac{1}{2}\Psi(Q_{i_{1}}),...,\tfrac{1}{2}\Psi(Q_{i_{n}})$ have length, distance from $\partial\Omega$ and distance from each other uniformly comparable to $\delta_{I}$. From here it follows that \eqref{eq:uniformity} is satisfied for suitable choices of $\mathtt{c}_{0}$ and $\mathtt{c}_{1}$. 

We are now ready to conclude the proof:  Whenever $I=(i_{1},...,i_{n})\in\mathbb{N}^{n}$ is such that $\varphi_{i_{n}}\wedge\dif\varphi_{i_{n-1}}\wedge ... \wedge\dif\varphi_{i_{1}}(y)\neq 0$, then the  sidelengths of the cubes $Q_{i_{1}},...,Q_{i_{n}}$ are uniformly comparable, cf. \ref{item:Whitney3}. Combining this with \ref{item:Whitney2}, we then conclude that the number $\delta_{I}$ from \eqref{eq:deltadef} is uniformly comparable to $\dist(y,\partial\Omega)$. We then use \ref{item:POU3} and Lemma \ref{lem:convhullL1bound} to find for all $j\in\{1,...,n-1\}$
\begin{align}\label{eq:faraway1}
\begin{split}
    & \|\phi_{i_{n}}\|_{\lebe^{\infty}(\R^{n})}\leq 1,\;\;\Vert\!\dif\phi_{i_{j}}\!\Vert_{\lebe^{\infty}(\R^{n})} \leq \frac{c}{\ell(Q_{i_{j}})} \leq \frac{c}{\mathrm{dist}(Q_{i_{j}},\partial\Omega)}\leq \frac{c}{\delta_{I}}, \\ 
    &\left\vert \int_{(\R^{n})^{n}} \fint_{\conv(\mathbf{x}_{I})} u(z) \nu(\mathbf{x}_{I}) \dHaus^{n-1}(z) ~\textup{d}\mu^I(\mathbf{x}_{I}) \right\vert \leq \frac{c}{\delta_{I}} \int_{\conv_{I}}|u(x)|\dif x, 
    \end{split}
\end{align}
where $c=c(n,\Lip(\Omega,\mathcal{O}))>0$. Recalling the neighbouring cube notation from \eqref{eq:peaceintheneighbourhood}, we have $|\mathcal{N}(i_{n})|\leq c(n)$ for all $i_{n}\in\mathbb{N}$ by \ref{item:Whitney2}. Hence we obtain for all $k\in\mathbb{N}$
\begin{align*}
\|& \mathscr{E}_{\Omega}[u]\|_{\lebe^{1}(\mathtt{S}_{k})}  \stackrel{\eqref{eq:faraway1}_{2}}{\leq} c\int_{\mathtt{S}_{k}}\sum_{I=(i_{1},...,i_{n})}\frac{1}{\delta_{I}}\big\vert\varphi_{i_{n}}\wedge\dif\varphi_{i_{n-1}}\wedge ... \wedge \dif\varphi_{i_{1}}(y)|\Big(\int_{\conv_{I}}|u(x)|\dif x\Big)\dif y 
\\ & \;\;\;= c \sum_{\substack{I=(i_{1},...,i_{n})\\ \mathtt{S}_{k}\cap \bigcup\{Q\in\mathcal{N}(i_{n})\}\neq\emptyset}}\frac{1}{\delta_{I}}\int_{\bigcup\{Q\in\mathcal{N}({i_{n}})\}}\big\vert\varphi_{i_{n}}\wedge\dif\varphi_{i_{n-1}}\wedge ... \wedge \dif\varphi_{i_{1}}(y)|\Big(\int_{\conv_{I}}|u(x)|\dif x\Big)\dif y \\ 
& \stackrel{\eqref{eq:faraway1}_{1}}{\leq} c  \sum_{\substack{I=(i_{1},...,i_{n})\\ \mathtt{S}_{k}\cap \bigcup\{Q\in\mathcal{N}(i_{n})\}\neq\emptyset}}\frac{1}{\delta_{I}^{n}}\mathscr{L}^{n}\Big(\bigcup \{Q\in\mathcal{N}(i_{n})\} \Big)\Big(\int_{\conv_{I}}|u(x)|\dif x\Big) \\
\\ & \;\;\; \leq c \sum_{\substack{I=(i_{1},...,i_{n})\\ \mathtt{S}_{k}\cap \bigcup\{Q\in\mathcal{N}(i_{n})\}\neq\emptyset}}\int_{\conv_{I}}|u(x)|\dif x . 
\\ & \;\;\; \leq c \int_{\mathtt{S}'_k} \vert u(x) \vert \cdot  \# \bigl \{ I =(i_1,...,i_n) \colon \mathtt{S}_k \cap \bigcup \{Q \in \mathcal{N}(i_n)\} \neq \emptyset, x \in \conv_I \bigr\}\dif x =: I_k.
\end{align*}
It remains to uniformly bound the number of indices above. Note that due to the definition of $\conv_I$ and convexity of the distance function
\[
x \in \conv_I \quad \Longrightarrow \quad \dist(x,Q_{i_n}) \leq \sup_{y \in \cup_{j=1}^n \Psi(Q_{i_j})} \dist(y,Q_{i_n}).
\]
Observe that if $y \in \Psi(Q_{i_j})$ and $Q_{i_j}$ neighbours $Q_{i_n}$ we can bound the distance to $Q_{i_n}$ as follows:
\[
\dist(y,Q_{i_n}) \leq \sqrt{n} \ell(\Psi(Q_{i_j}) + \dist(\Psi(Q_{i_j}),Q_{i_j}) + \sqrt{n} \ell(Q_{i_j}).
\]
 Du to properties \ref{item:Jonesrefined2} and \ref{item:Jonesrefined2} of the map $\Psi$ and \ref{item:Whitney1} we obtain for a constant $C=C(n,\Lip(\Omega,\mathcal{O}))$
 \[
 \dist(y,Q_{i_n}) \leq C \ell(Q_{i_n}) \leq C 2^{-k-1} \vartheta.
 \]
Consequently, the distance of $x \in \conv_I$ to $Q_{i_n}$ is bounded by $C2^{-i-1} \vartheta$. Now recall that any $Q_{i_n}$ with $\cup \{Q \in \mathcal{N}(Q_{i_n})\} \cap \mathtt{S}_k$ needs to satisfy
\[
2^{-k-3 \vartheta} \leq  \ell Q_{i_n} \leq 2^{-k+3} \vartheta.
\]
Then, on the one hand, if $x \in \conv_I$
\[
Q_{i_n} \subset B_{2^{-i-1}(C+16\sqrt{n})}(x)
\]
and, on the other hand, 
\[
\mathscr{L}^n(Q_{i_n}) \geq (2^{-k-3 \vartheta})^n.
\]
Using that only finitely many cubes overlap, i.e. \ref{item:Whitney3}, we obtain that the number of $i_n$ such that for some index $I$ (with neighbouring cubes $Q_{i_1},....,Q_{i_{n-1}}$), $x \in \conv_I$, is bounded by
\[
\mathtt{N}(n)\frac{\mathscr{L}^n (B_{2^{-i-1}(C+16\sqrt{n})}(x))}{(2^{-k-3 \vartheta})^n} =:\tilde{C}(n).
\]
As only $\mathtt{N}(n)$ cubes are a neighbour of $Q_{i_n}$ we conclude
\[
\# \bigl \{ I =(i_1,...,i_n) \colon \mathtt{S}_k \cap \bigcup \{Q \in \mathcal{N}(i_n)\} \neq \emptyset, x \in \conv_I \bigr\} \leq \tilde{C}(n) \mathtt{N}(n)^n.
\]
Thus,
\[
I_k \leq c \tilde{C}(n) \mathtt{N}(n)^n \int_{\mathtt{S}'_k} \vert u(x) \vert \dif x 
\]
and \ref{item:L1bound0} is established with $c_1= c \tilde{C}(n) \mathtt{N}(n)^n$.

We continue with \ref{item:L1bound2}. The definition of the strips $S_k'$ implies that any strip intersects with at most $2 + \log_2(\mathtt{c}_1/\mathtt{c}_0)$ and thus
\begin{align}\label{eq:jarrett}
\begin{split}
\|\mathscr{E}_{\Omega}[u]\|_{\lebe^{1}(\ball_{\vartheta}(\Omega)\setminus\Omega)} & \leq \sum_{k=0}^{\infty}\|\mathscr{E}_{\Omega}[u]\|_{\lebe^{1}(\mathtt{S}_{k})} = \\
& \leq c_1\sum_{k=0}^{\infty} \|u\|_{\lebe^{1}(\mathtt{S}'_{k})} \leq   c_1 (3 + \log_2(\mathtt{c}_1/\mathtt{c}_0))\|u\|_{\lebe^{1}(\Omega)}  =: c_2 \|u\|_{\lebe^{1}(\Omega)}
 \end{split}
    \end{align}
and from here \ref{item:L1bound1} is immediate. For  \ref{item:L1bound1} we may only sum over $k \geq k_0$ for some $k_0 \in \N$ to obtain the bound.
\end{proof}
\begin{remark}
Observe that all constants in this proof are uniform in $\varepsilon$, when the domain $\Omega$ is replaced by the inner approximation $\Omega_{\varepsilon}= \{x \in \Omega \colon \dist(x,\partial \Omega)> \varepsilon\}$. In particular, in the construction of the double Whitney cover and the map $\psi$, involved constants are only dependent on $\Lip(\mathcal{O},\Omega_{\varepsilon})$, which is uniformly bounded in $\varepsilon$.
\end{remark}
\noindent We are now ready to give the proof of Proposition \ref{prop:intermediateL1convex}:
\begin{proof}[Proof of Proposition \ref{prop:intermediateL1convex}]
In Lemma \ref{lemma:L1bound:c1}, we established that 
\begin{align*}
\mathscr{E}_{\Omega}\colon \hold^{1}(\overline{\Omega};{\wedgeq}{^{n-1}}(\R^{n})^{*})\to\lebe^{1}(\R^{n};{\wedgeq}{^{n-1}}(\R^{n})^{*})
\end{align*}
is a bounded linear operator with respect to the $\lebe^{1}$-norm. On the other hand, Corollary \ref{coro:divfree} implies that, if $u\in\hold(\overline{\Omega},\wedgeq^{n-1}(\R^{n})^{*})$ satisfies $\dif u = 0$ in $\Omega$, then $\dif\mathscr{E}_{\Omega}[u]=0$ in $\mathscr{D}'(\Omega;\wedgeq^{n}(\R^{n})^{*})$, and this completes the proof of Proposition \ref{prop:intermediateL1convex}.
\end{proof} 
\subsection{Density of $\hold_{\mathrm{div}}^{1}(\overline{\Omega};\R^{n})$ in $\lebe_{\mathrm{div}}^{1}(\Omega;\R^{n})$ and the proof of Theorem \ref{thm:divL1conv}}\label{sec:density}
In order to introduce the divergence-free extension operator on $\lebe^{1}$, we now proceed to establish a density result. For its proof, we directly utilise Proposition \ref{prop:intermediateL1convex} and translate back from differential forms to $\R^{n}$-valued maps. 
\begin{proposition}\label{prop:densitymain}
Let $\Omega\subset\R^{n}$ be open, convex and bounded. Then $\hold_{\di}^{1}(\overline{\Omega};\R^{n})$ is dense in $\lebe_{\mathrm{div}}^{1}(\Omega;\R^{n})$ with respect to the $\lebe^{1}$-norm. 
\end{proposition} 
\begin{proof} 
Let $u\in\lebe_{\di}^{1}(\Omega;\R^{n})$. Moreover, let $\rho\in\hold_{c}^{\infty}(\ball_{1}(0);[0,1])$ be a standard mollifier with $\int_{\ball_{1}(0)}\rho(x)\dif x =1$. Given $\varepsilon>0$ sufficiently small, we denote by $\rho_{\varepsilon}(x):=\frac{1}{\varepsilon^{n}}\rho(\frac{x}{\varepsilon})$ its $\varepsilon$-rescaled variant. Letting $\overline{u}$ be the trivial extension of $u$ to $\R^{n}$ by zero, we then note that the mollified map $u_{\varepsilon}:=\rho_{\varepsilon}*\overline{u}$ satisfies the following: 
\begin{enumerate}[label=(M\arabic*)]
\item\label{item:molli1} $\|\overline{u}-u_{\varepsilon}\|_{\lebe^{1}(\R^{n})}\to 0$ as $\varepsilon\searrow 0$; 
\item\label{item:molli2} $u_{\varepsilon}|_{\Omega_{\varepsilon}}\in\hold^{\infty}(\overline{\Omega_{\varepsilon}};\R^{n})$; 
\item\label{item:molli3} $\mathrm{div}(u_{\varepsilon})=0$ pointwisely in $\Omega_{\varepsilon}$. 
\end{enumerate} 
For all sufficiently small $\varepsilon>0$, the set  $\Omega_{\varepsilon}$ is non-empty and convex, and by \eqref{eq:Lipindependent}, $\Lip(\mathcal{O},\Omega_{\varepsilon})$ is uniformly bounded in terms of $\Lip(\mathcal{O},\Omega)$ for all sufficiently small $\varepsilon>0$. Based on the divergence-free extension operator from Definition \ref{def:Eomega}, we then define $v_{\varepsilon}:=\mathscr{E}_{\Omega_{\varepsilon}}[u_{\varepsilon}]$. By construction and due to Lemma \ref{lem:compactsupport} function is compactly supported in a fixed ball $\ball_{R}(0)$. Then there exists $\varepsilon_{0}>0$ such that, for all $0<\varepsilon<\varepsilon_{0}$, $\Omega_{\varepsilon}$ is non-empty, convex, open and bounded. Moreover, based on \eqref{eq:Lipindependent}, Lemma \ref{lemma:pointwise:solenoidal}, Corollary \ref{coro:divfree} and Lemma \ref{lemma:L1bound:c1} now imply the following properties of $v_{\varepsilon}$:
\begin{enumerate}[label=(V\arabic*)]
\item\label{item:molli4} $v_{\varepsilon}|_{\Omega_{\varepsilon}}\in\hold^{\infty}(\Omega_{\varepsilon};\R^{n})$, $v_{\varepsilon}|_{\R^{n}\setminus\overline{\Omega_{\varepsilon}}}\in\hold^{\infty}(\R^{n}\setminus\overline{\Omega_{\varepsilon}};\R^{n})$;
\item\label{item:molli5} $\mathrm{div}(v_{\varepsilon})=0$ in $\mathscr{D}'(\R^{n})$;
\item\label{item:molli6} $v_{\varepsilon}=u_{\varepsilon}$ in $\Omega_{\varepsilon}$;
\item\label{item:molli7} there exist $\mathtt{c}=\mathtt{c}(n,\Lip(\Omega,\mathcal{O}))>0$ such that we have $
\|v_{\varepsilon}\|_{\lebe^{1}(\R^{n})} \leq \mathtt{c}\|u_{\varepsilon}\|_{\lebe^{1}(\Omega_{\varepsilon})}.$
\end{enumerate}
We explicitly note that \ref{item:molli7} is a consequence of the fact that the $\lebe^{1}$-operator norm of $\mathscr{E}_{\Omega_{\varepsilon}}$ only depends on $n$ and $\Lip(\Omega_{\varepsilon},\mathcal{O})$ for all sufficiently small $\varepsilon>0$. By \eqref{eq:Lipindependent}, this number in turn can be bounded in terms of $\mathrm{Lip}(\Omega,\mathcal{O})$ \emph{uniformly} in $0<\varepsilon<\varepsilon_{0}$, and so the constant in \ref{item:molli7} can be assumed to be independent of $0<\varepsilon<\varepsilon_{0}$.

For $\delta>0$ sufficiently small, we then put $w_{\delta,\varepsilon}:=\rho_{\delta}*v_{\varepsilon}$. Then we have $w_{\delta,\varepsilon}\in\hold_{c}^{\infty}(\R^{n};\R^{n})$ and, by \ref{item:molli5},  $\mathrm{div}(w_{\delta,\varepsilon})=0$ pointwisely in $\R^{n}$. Lastly, we estimate 
\begin{align}\label{eq:ulihoeness}
\|u-w_{\delta,\varepsilon}\|_{\lebe^{1}(\Omega)} & \leq \|u-u_{\varepsilon}\|_{\lebe^{1}(\Omega)}+\|u_{\varepsilon}-v_{\varepsilon}\|_{\lebe^{1}(\Omega)}+\|v_{\varepsilon}-w_{\delta,\varepsilon}\|_{\lebe^{1}(\Omega)}. 
\end{align}
By \ref{item:molli1}, the first term tends to zero as $\varepsilon\searrow 0$, and the third term can also be made arbitrarily small by first sending $\delta\searrow 0$ and then $\varepsilon\searrow 0$. For the second term, we note that, by \ref{item:molli6} and \ref{item:molli7},
\begin{align*}
\|u_{\varepsilon}-v_{\varepsilon}\|_{\lebe^{1}(\Omega)} & \leq \|u_{\varepsilon}\|_{\lebe^{1}(\Omega\setminus\Omega_{\varepsilon})}+\|v_{\varepsilon}\|_{\lebe^{1}(\Omega\setminus\Omega_{\varepsilon})} \\ & \!\!\!\!\stackrel{\eqref{eq:stripsum}}{\leq} \|u\|_{\lebe^{1}(\Omega\setminus\Omega_{2\varepsilon})} + c\|u_{\varepsilon}\|_{\lebe^{1}(\Omega\setminus\Omega_{\mathtt{c}\varepsilon})} \leq c\|u\|_{\lebe^{1}(\Omega\setminus\Omega_{\max\{2,\mathtt{c}\}\varepsilon})}.
\end{align*}
Since $u\in\lebe^{1}(\Omega;\R^{n})$, it is then clear that the last term vanishes as $\varepsilon\searrow 0$. Going back to \eqref{eq:ulihoeness}, $\|u-w_{\delta,\varepsilon}\|_{\lebe^{1}(\Omega)}$ can be made arbitrarily small, and since $w_{\delta,\varepsilon}\in\hold_{\diver}^{1}(\overline{\Omega};\R^{n})$, this completes the proof. 
\end{proof} 
\noindent Based on Propositions  \ref{prop:intermediateL1convex} and \ref{prop:densitymain}, we are now in a position to prove Theorem \ref{thm:divL1conv}: 
\begin{proof}[Proof of Theorem \ref{thm:divL1conv}]
Let $u\in\lebe^{1}(\Omega;\R^{n})$ and let $(u_{j})\subset\hold^{1}(\overline{\Omega};\R^{n})$ be an arbitrary sequence with $\|u-u_{j}\|_{\lebe^{1}(\Omega)}\to 0$ as $j\to\infty$. By Proposition \ref{prop:intermediateL1convex}, $(\mathscr{E}_{\Omega}[u_{j}])$ is a Cauchy sequence in $\lebe^{1}(\R^{n};\R^{n})$ and thus strongly converges to some element in $\lebe^{1}(\R^{n};\R^{n})$. Since this $\lebe^{1}(\R^{n};\R^{n})$-limit is clearly independent of the particular approximating sequence $(u_{j})\subset\hold^{1}(\overline{\Omega};\R^{n})$, we may define $\mathscr{E}_{\Omega}[u]$ as the $\lebe^{1}(\R^{n};\R^{n})$-limit of $(\mathscr{E}_{\Omega}[u_{j}])$ for any such approximating sequence. Clearly, $\mathscr{E}_{\Omega}$ is linear on $\lebe^{1}(\Omega;\R^{n})$. Moreover, by Proposition \ref{prop:intermediateL1convex}, we then have 
\begin{align*}
\|\mathscr{E}_{\Omega}[u]\|_{\lebe^{1}(\R^{n})}=\lim_{j\to\infty} \|\mathscr{E}_{\Omega}[u_{j}]\|_{\lebe^{1}(\R^{n})} \leq c\lim_{j\to\infty}\|u_{j}\|_{\lebe^{1}(\Omega)} = c\|u\|_{\lebe^{1}(\Omega)}
\end{align*}
with $c>0$ only depending on $n$ and $\Lip(\Omega,\mathcal{O})$. 

Now let $u\in\lebe_{\mathrm{div}}^{1}(\Omega;\R^{n})$. We then use Proposition \ref{prop:densitymain} to find a sequence $(u_{j})\subset\hold_{\mathrm{div}}^{1}(\overline{\Omega};\R^{n})$ such that $\|u-u_{j}\|_{\lebe^{1}(\Omega)}\to 0$ as $j\to\infty$. Since the definition of $\mathscr{E}_{\Omega}[u]$ is independent of the particular approximating sequence, $\mathscr{E}_{\Omega}[u]$ is the $\lebe^{1}(\R^{n};\R^{n})$-limit of $(\mathscr{E}_{\Omega}[u_{j}])$ for the particular sequence $(u_{j})\subset\hold_{\mathrm{div}}^{1}(\overline{\Omega};\R^{n})$. By Proposition \ref{prop:intermediateL1convex}, $\diver(\mathscr{E}_{\Omega}[u_{j}])=0$ in $\mathscr{D}'(\Omega)$, and so we conclude   for all $\varphi\in\hold_{c}^{\infty}(\R^{n})$ that
\begin{align*}
\int_{\R^{n}}\mathscr{E}_{\Omega}[u]\cdot\nabla\varphi\dif x & = \lim_{j\to\infty}\int_{\R^{n}}\mathscr{E}_{\Omega}[u_{j}]\cdot\nabla\varphi\dif x = 0. 
\end{align*}
Hence, $\mathscr{E}_{\Omega}\colon\lebe_{\mathrm{div}}^{1}(\Omega;\R^{n})\to\lebe_{\mathrm{div}}^{1}(\R^{n};\R^{n})$ indeed, and this completes the proof. 
\end{proof} 

\section{Divergence-free extensions in $\lebe^{1}$ for Lipschitz domains} \label{sec:Lip}
In this section, we generalise the strategy from Section \ref{sec:L1} to Lipschitz domains. Specifically, this yields the proof of the following result and hereby Theorem \ref{thm:L1}: 
\begin{theorem}[Divergence-free $\lebe^{1}$-extensions for Lipschitz domains]\label{thm:main:Lip}
    Let $\Omega \subset \R^n$ be an bounded domain with Lipschitz boundary $\partial\Omega$. Then there exist $\varepsilon >0$ and a \emph{linear and bounded extension operator}  $\mathscr{E}_{\Omega} \colon \lebe^1(\Omega;\R^n) \to \lebe^1(\Omega^{\varepsilon};\R^n)$ with $\Omega^{\varepsilon}:=\ball_{\varepsilon}(\Omega)$ such that, in particular, $\mathscr{E}_{\Omega}u|_{\Omega}=u$ $\mathscr{L}^{n}$-a.e. in $\Omega$ whenever $u\in\lebe^{1}(\Omega;\R^{n})$ and  
    \begin{align*}
    u\in\lebe^{1}(\Omega;\R^{n})\;\text{and}\;\diver u = 0 \text{ in } \mathscr{D}'(\Omega) \quad \Longrightarrow \quad \diver \mathscr{E}_{\Omega}[u] =0 \text{ in } \mathscr{D}'(\Omega^{\varepsilon}).
    \end{align*}
\end{theorem}
The construction of the extension operator underlying Theorem \ref{thm:main:Lip}  is inspired by the treatment of the easier convex case from Section \ref{sec:L1}. While the overarching strategy is the same and various intermediate results follow likewise, others require slight modification. As in Section \ref{sec:L1}, we identify divergence-free functions with closed $(n-1)$-forms, and organise the proof of Theorem \ref{thm:main:Lip} as follows: In
\begin{enumerate} [label=(\alph*)]
    \item\label{item:MaxEberl} Section \ref{sec:Lip:geometric}, we give the geometric set-up underlying Theorem \ref{thm:main:Lip} and introduce the number $\varepsilon>0$ based on the geometric limitations in the present setting.
    \item Section \ref{sec:Lip:def}, we define the extension operator in view of  \ref{item:MaxEberl}. Moreover, we outline how an analogue of Lemma \ref{lemma:pointwise:solenoidal} can be achieved.
    \item Section \ref{sec:Lip:Lem44}, we establish the solenoidality of $\mathscr{E}_{\Omega}[u]$ in $\mathscr{D}'(\Omega^{\varepsilon})$ provided $u\in\lebe^{1}(\Omega;\R^{n})$ is divergence-free. This is the counterpart of Lemma \ref{lemma:globalsolenoidality} and Corollary \ref{coro:divfree}.
    \item\label{item:MaxEberl1} Section \ref{sec:Lip:L1bounds}, we establish the corresponding $\lebe^{1}$-bounds for the extensions. This is the counterpart of Lemmas  \ref{lem:convhullL1bound} and \ref{lemma:L1bound:c1}.
    \item Section \ref{sec:Lip:final}, we employ a density argument similar to that from Section \ref{sec:density} to accomplish the passage from $\hold^{1}(\overline{\Omega};\R^{n})$- to $\lebe^{1}(\Omega;\R^{n})$-maps.
    \end{enumerate}
By the slightly different geometric set-up from \ref{item:MaxEberl}, it is especially the estimations underlying \ref{item:MaxEberl1} which require modification. Lastly, after having accomplished the proof of Theorem \ref{thm:main:Lip}, we discuss global extensions in the spirit of \cite{HKMT} for the present $\lebe^{1}$-setting in Section \ref{sec:globalext}. 
\subsection{Geometric set-up} \label{sec:Lip:geometric}
We now proceed to the detailled geometric set-up of Theorem \ref{thm:main:Lip}.
At this point we emphasise that we construct the extension for general \emph{Lipschitz} domains. The following diffeomorphism to a smooth domain is \emph{not} directly relevant for the extension itself, but only for the suitable definition of simplices in $\Omega$. Hence, given the definition of those simplices we will not need the identification with a smooth domain any longer.
\subsubsection{Lipschitz vs smooth domains}
We consider a Lipschitz domain, so an open, bounded and connected set $\Omega \subset \R^n$. By the Ball-Zarnescu Lemma \ref{lemma:BZ}, there is an open domain $V$ with smooth boundary, and a Bi-Lipschitz map $\Gamma_1 \colon \overline{\Omega} \to \overline{V}$
, which is smooth in the interior, i.e., smooth in $\Omega$. We may cover $\partial V$ with a finite number of sets $O_i$ such that (upon possible rotations)
\[
O_i \cap U = O_i \cap \{ (x',t) \colon x' \in \R^{n-1}, 
t < f_i(x') \}
\]
for $f_i \in\hold_{c}^{\infty}(\R^{n-1})$. By Lebesgue's covering lemma (\cite[Lemma 27.5]{Munkres}) there is some $\varepsilon_1>0$, such that for any $x \in \partial \Omega$ there is some $O_i$ such that $\ball_{\varepsilon_1}(x) \subset O_i$.

Moreover, due to Lemma \ref{lemma:collar} there exists $\varepsilon_2>0$ such that there is a collar domain around $\partial V$, that is,
\[
\{ x \in V \colon \dist(x, \partial V) < \franz\} \quad \text{and} \quad \partial V \times (0,1)
\]
are smoothly diffeomorphic by means of a diffeomorphism $\Gamma_2$. Letting $0<\varepsilon< \varepsilon_2$, we put
\[
\Omega_\varepsilon = T^{-1} ( \{x \in V \colon \dist(x, \partial V) > \franz \}).
\]
which forms the basis of a density argument similar to that from Section \ref{sec:density}, cf. Section \ref{sec:Lip:final}.
\subsubsection{Definition of simplices} \label{sec:def:simp}
We have discussed in Section \ref{sec:def:simplex} how to define simplices in Lipschitz domains by identification of a neighbourhood to the boundary to $\partial V \times (0,1)$ for a smoothly bounded domains $V$.
Let us shortly recall the setup. The map $\Gamma_1$ was a Bi-Lipschitz homeomorphism between $\bar{\Omega}$ and $\bar{V}$ that is smooth in the interior. Moreover, $\Gamma_2$ is a smooth diffeomorphism between a neighbourhood of  $\partial V$ and $\partial V \times (0,1)$. We defined simplices on $\Omega$ by defining simplices on $\partial V \times (0,1)$ first (cf. Figure \ref{figure:simplex}) and then using both homeomorphism to define the simplex on $\Omega$.

In particular, given $(x_0,s_0),\ldots,(x_k,s_k) \in \partial V \times (0,1)$, we established for the second coordinate $t \mapsto s(t)$ with \eqref{def:coord2}
\[
    \min_{i} s_i \leq s(t) \leq \max_{i} s_i.
\]
This observation gives us the estimate on the distance of the manifold to the boundary, i.e. Theorem \ref{thm:prop:simplex} \ref{def:prop:4}. The set $\Omega_{\varepsilon}$, is however characterised by $\Omega_{\varepsilon}=T^{-1}((\varepsilon/\franz,1)\times\partial V)$, $T= \Gamma_2 \circ \Gamma_1$. Therefore, we may extract the following.
\begin{lemma}
    Let $x_0,...,x_k \in \Omega_{\varepsilon}$ and $0<\varepsilon<\franz$ be such that the following hold:
    \begin{enumerate} [label=(A\arabic*')]
        \item \label{ass1:v2} $\eta := \inf_{0\leq i\leq k} \dist(x_i,\partial \Omega_\varepsilon) \leq \varepsilon$;
        \item \label{ass2:v2} $ \sup_{0\leq i\leq k} \dist(x_i, \partial \Omega_\varepsilon) \leq C_1 \eta$; 
         \item \label{ass3:v2} $\sup_{0\leq i,j \leq k} \vert x_i - x_j \vert \leq C_1 \eta$.
    \end{enumerate}
    Then the simplex $S_{\bar{x}}$ as defined in Section \ref{sec:def:simplex} (cf. Theorem \ref{thm:prop:simplex}) satisfies
    \begin{enumerate}
        \item[(S5')] \label{def:prop:4:v2} For all $x \in S_{\bar{x}}$ we have
        \[
        \alpha^{-1} \eta \leq \dist(x,\partial \Omega_\varepsilon) \leq C \alpha  \eta.
        \]
    \end{enumerate}
\end{lemma}
Hence, the simplices constructed for $\Omega$ through the identification of a neighbourhood of the boundary $\partial\Omega$ with $\partial V \times (-1,0)$ can also be used for $\Omega_{\varepsilon}$.

\begin{figure}
\includegraphics[width=0.9\textwidth]{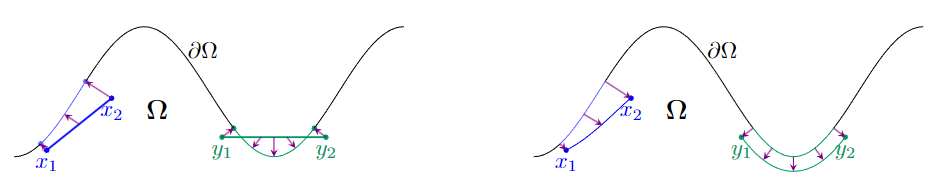}
\caption{ An exemplaric picture how the simplices are defined in dimension $n=2$ (after possibly already identifying $\Omega$ with a domain with smooth boundary, cf. Figure \ref{figure:collar}). In the first step, consider the simplex just in $\R^2$, i.e. the direct line \textcolor{blue}{$[x_1,x_2]$} and \textcolor{Gump} {$[y_1,y_2]$}. This simplex might not even be contained in $\Omega$, as one might see with $[y_1,y_2]$. Project the simplex onto the boundary through the normal such that (in 2D), we get a map $\Gamma_2^1 \colon [0,1] \to \partial \Omega$, the parametrisation of the projection. On the right hand side we see the second step of the construction: We then use the interior normal $\nu_{\Omega}$ and define $x(t)= \Gamma_2^1(x) + ((1-t) \dist(x_1,\partial \Omega)+ t \dist(x_2,\partial \Omega)) \cdot \nu_{\Omega}$ (similarly with $y$).
}
\end{figure}
\begin{proof}
   Let $W \subset \bar{\Omega}$ such that the map $T \colon W \to \partial V \times [0,1)$ is Bi-Lipschitz. Then, for any $x$
   \[
   c^{-1} \vert T_2(x) - \epsilon/\franz \vert 
    \leq \dist(x,\partial \Omega_{\varepsilon}) \leq c \vert T_2(x) - \epsilon/\franz \vert 
   \]
    As, given $(x_0',s_0'),...,(x_k',s_k') \in \partial V$, the second component of the simplex obeys
    \[
    \min_{i} s_i \leq s(t) \leq \max_{i} s(t),
    \]
   we conclude for the simplex $S_{\bar{x}}$:
    \[
 c^{-2} \min_{i} \dist(x_i,\partial \Omega_{\varepsilon}) \leq   \min_{t \in \mathbb{D}^k} \dist(S_{\bar{x}}(t), \partial \Omega_{\varepsilon}) \leq  \max_{t \in \mathbb{D}^k} \dist(S_{\bar{x}}(t), \partial \Omega_{\varepsilon}) \leq c^2 \max_{i} \dist(x_i,\partial \Omega_{\varepsilon}).
    \]
    Using Assumption \ref{ass1:v2} and \ref{ass2:v2} then yields the result.
\end{proof}

Therefore, given points $x_{i_1},...,x_{i_{k}}$, $I=(i_1, \ldots,i_k)$ we denote by \textbf{$M_I$} the $(k-1)$-dimensional simplex with vertices $x_{i_j}$, $j=1,\ldots,k$,

\subsubsection{Whitney cubes etc.}
In the following we consider domains $U=\Omega_{\varepsilon}$. As for the convex case we might pick Whitney covers $\mathcal{W}_{1}^{\varepsilon}$ and $\mathcal{W}_{2}^{\varepsilon}$ and a map 
$\psi_\varepsilon \colon\mathcal{W}_{2}^{\varepsilon} \to \mathcal{W}_{1}^{\varepsilon}$ enjoying
\begin{enumerate} [label=(J\arabic*)]
    \item\label{item:Jonesrefined1:v2} If $\ell(Q) > \eta$, then $\Psi_\epsilon(Q) = Q_0^{\epsilon}$ for some cube $Q_0^{\epsilon}$;
    \item\label{item:Jonesrefined2:v2} $\frac{1}{4}\ell(Q) \leq \ell(\Psi_\epsilon(Q)) \leq 4 \ell(Q) $ for all $Q\in\mathcal{W}_{2}^{\epsilon}$ with $\ell(Q) \leq \eta$;
    \item \label{item:Jonesrefined3:v2} $\dist(Q,\Psi(Q)) \leq C \ell(Q)$, where $C>0$ only depends on $\Lip(\Omega_{\varepsilon},\mathcal{O})$;
    \item \label{item:Jonesrefined3a:v2} $c\ell(Q)\leq \dist(\Psi(Q),\partial\Omega) \leq C \ell(Q)$, where $c,C$ only depend on $\Lip(\Omega_{\varepsilon},\mathcal{O})$;
    \item\label{item:Jonesrefined4:v2} $C^{-1}\ell(Q)\leq\dist(Q,\Psi(Q)) \leq C \ell(Q)$ whenever $\ell(Q)\leq\eta$, where $C$ only depends on $\Lip(\Omega_{\varepsilon},\mathcal{O})$.
\end{enumerate}
In other words, the size of the cube $\psi_{\epsilon}(Q)$, its distance to the original exterior cubes $Q$ and its distance the boundary,  are all comparable to the sidelength of the cubes. In particular, involved constants only depend on $\Lip(\Omega_{\varepsilon},\mathcal{O})$, which \emph{does not} depend on $\varepsilon$. Moreover, recall that we have a smooth partition of unity $\phi_i$ on the exterior cubes with the properties \ref{item:POU1}, \ref{item:POU2} \& \ref{item:POU3}. 
\subsection{Definition of the extension} \label{sec:Lip:def}
Let $U= \Omega_{\delta}$ for some $0 < \delta <\franz$. We define the extension as follows.
\begin{definition} \label{def:extension:Lip}
    Let $u \in \hold^1(\overline{U};\wedgeq^{n-1}(\R^{n})^{\ast})$. Then we define the \emph{divergence-free extension operator} by
    \begin{align*}
    \mathscr{E}_{U}[u]:=\begin{cases} 
    u &\;\text{in}\;\;U,\\ 
    \E_{U}[u]&\;\text{in}\;\;\R^{n}\setminus\overline{U}, 
    \end{cases}
    \end{align*}
where $\E_{U}u$ is given by 
    \begin{equation} \label{def:Eomega:Lip}
    \begin{split}
         \E_{U}[u]  & :=  (-1)^{n-1} \sum_{I=(i_1,...,i_n)\in\mathbb{N}^{n}} \Big(\phi_{i_n} \wedge \dif \phi_{i_{n-1}} \wedge \dots \wedge \dif \phi_{i_1} \Big. \\ & \Big. \;\;\;\;\;\;\;\;\;\;\;\;\;\;\;\;\;\;\;\;\;\;\;\;\cdot \int_{(\R^{n})^{n}} \int_{M(\mathbf{x}_{I})} u(z) \nu \dHaus^{n-1}(z) ~\textup{d}\mu^{I}(\mathbf{x}_{I})\Big).
        \end{split}
    \end{equation}
In the case that the simplex spanned up by $x_{i_1},...,x_{i_n}$ is degenerate, we define the inner integral over the manifold $M_I$ to be zero.   
\end{definition}
We warn the reader that the normal $\nu$ appearing in \eqref{def:Eomega:Lip} now denotes the \emph{unit normal} of $M(\mathbf{x}_I)$
at a point $z \in M(\mathbf{x}_I)$, which is well-defined as the simplex is smooth. This is in contrast to the previous section where the normal was linearly depending on $x_I$ (this was compensated by an \emph{average} integral). The orientation of the normal is chosen in such a way that
\[
\int_{M(\mathbf{x}_I)} \nu \dif \mathscr{H}^{r-1} = \frac{1}{(r-1)!}(x_{i_r} - x_{i_{r-1}}) \wedge (x_{i_{r-1}} - x_{i_{r-2}}) \wedge ... \wedge (x_{i_{2}}-x_{i_{1}}),
\]
which, in the case of a flat curvilinear simplex (given by the convex hull of points), is up to the aforementioned normalising positive factor consistent with the definition \eqref{eq:normaldefine}.

This also hints at one of the major difference between the proofs here and the proofs of Section \ref{sec:L1ext}: The normal is not constant, hence for some of the results we need to argue slightly differently. Nevertheless, the proofs are quite parallel and use the same ideas as in Section \ref{sec:L1ext}. In the proofs, we therefore highlight the key differences and refer to the appropriate counterpart in Section \ref{sec:L1ext} for technical details instead of repeating the whole argument.

Let $U^{\varepsilon}$ be the $\varepsilon$-neighbourhood of $U$, (i.e. if $\delta$ is small enough, then $U^{\varepsilon} \subset B_{\varepsilon/2}(\Omega)$). Observe that if $Q_{i_1},...,Q_{i_n}$ are cubes in $\mathcal{W}_1^{\varepsilon}$ that touch, then for any $x_j \in Q^{i_j},x_l \in Q^{i_l}$, we have
\[
\dist(x_i,x_j) \leq C \varepsilon \leq \varepsilon_2/3.
\]
In particular, the simplex spanned by $x_1,...,x_n$ (such that $x_j \in Q^{i_j}$) as in Section \ref{sec:def:simp}, is well-defined. Consequently, $\E_U u$ is well-defined in $U^{\varepsilon}$.

First of all, we may check that $\E_U u$ is solenoidal pointwisely in $U^{\varepsilon}$.
\begin{lemma} \label{lemma:pointwise:solenoidal:Lip}
    Let $u \in \hold^{1}(\overline{U},\wedgeq^{n-1}(\R^{n})^{*})$ and $U^{\varepsilon}$ be the previously defined neighbourhood of $U$. Then there holds $\mathscr{E}_{U}[u] \in \hold^{\infty}(U^{\varepsilon}\setminus\overline{U};{\wedgeq}^{n-1}(\R^n)^{\ast})$. Moreover, if we have $\dif  u = 0$ in $U$, then also $ \dif\mathscr{E}_{U}[u]=0$ in $U^{\varepsilon} \setminus\overline{U}$. 
\end{lemma}

This is indeed the same computation as in Lemma \ref{lemma:pointwise:solenoidal}. In particular, the way we constructed simplices $M_I$ yields that we may use Gau{\ss} theorem/ Stokes' theorem in dimension $n$.
\begin{proof}
As in the proof of Lemma \ref{lemma:pointwise:solenoidal}, the sum is locally finite and hence smooth in $U^{\varepsilon} \setminus \bar{U}$. It remains to calculate the pointwise derivative. Given $\widetilde{I}=(i_1,..,i_{n+1}) \in \N^{n+1}$, we denote for $s\in\{1,...,n\}$ 
\begin{align*}
I_s=(i_1,...,i_{s-1},i_{n+1},i_{s+1},...,i_n)\;\;\;\text{and}\;\;\; I_{n+1} =(i_1,...,i_n).
\end{align*} 
Using that the $\phi_{i_s}$'s form a partition of unity, we have for $y\in\R^{n}\setminus\overline{\Omega}$
\begin{align*}
\dif\,(\mathscr{E}_{\Omega}[u](y))  
& = \sum_{\widetilde{I}=(i_1,...,i_n,i_{n+1})\in\mathbb{N}^{n+1}} \Big(\varphi_{i_{n+1}}\dif \phi_{i_n} \wedge \dif \phi_{i_{n-1}} \wedge \dots \wedge \dif \phi_{i_1} \Big. \\ & \hspace{1cm}  \cdot \Big(\int_{(\R^{n})^{n}} \int_{M(\mathbf{x}_{I_{n+1}})} u(z) \nu(\mathbf{x}_{I_{n+1}}) \dHaus^{n-1}(z) ~\textup{d}\mu^I(\mathbf{x}_{I_{n+1}}) \Big. \\ & \hspace{1.3cm} - \sum_{s=1}^{n}\int_{(\R^{n})^{n}}\int_{M(\mathrm{x}_{I_{s}})}u(z) \nu(\mathbf{x}_{I_{s}}) \dHaus^{n-1}(z) ~\textup{d}\mu^{I_{s}}(\mathbf{x}_{I_{s}})\Big) \Big).
\end{align*}

Again, one may now apply Gauss theorem to the sum, as the the $(n+1)$ different $(n-1)$-dimensional simplices form the boundary of an $n$-dimensional simplex (and $\nu_{I_s}$ is the corresponding normal vector).
\end{proof}

\begin{remark} \label{remark:thomasmueller}
Observe that we can only show global solenoidality of $u$ in a small neighbourhood of $\Omega$, as otherwise we obtain a problem in defining simplices, such that the application of Stokes' theorem is reasonable. For example, considering an annulus (i.e. a domain with a hole), it is clear that $n$-dimensional simplices with an interior \emph{inside} of $\Omega$ can only be defined canonically \emph{locally} but not \emph{globally}, cf. Figure \ref{annulus}. \begin{figure}
\centering
\includegraphics[width=0.4\textwidth]{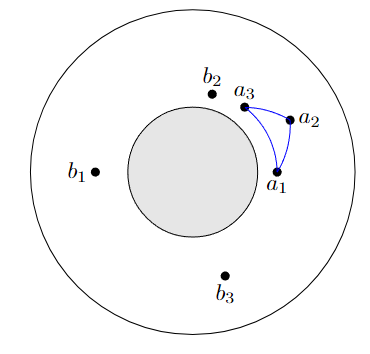}
\caption{For points $a_i$ close to each other, a meaningful simplex may be uniquely defined as in the picture above (in that specific example by identifying the simplex to $(1,2) \times \mathbb{S}^1$. For points $b_i$ far away from each other such a 'natural' definition is, however, impossible.} \label{annulus} \end{figure}
\end{remark} 
\subsection{Global Solenoidality} \label{sec:Lip:Lem44}

Compared to the convex framework, more care is needed for checking that the divergence is an $\lebe^1$-function. One reason for this is that we need to modify the proof of Lemma \ref{lemma:globalsolenoidality}, especially the recursion in Step 2, as our unit normal $\nu$ is \emph{not} explicitly given.
\begin{lemma} \label{lemma:globalsolenoidality:Lip}
Let $u \in \hold^{1}(\bar{U};{\wedgeq}^{n-1}(\R^n)^{\ast})$. Then $d (\E_{\Omega} u) \in \lebe^1(U^{\varepsilon};{\wedgeq}^{n}(\R^n)^{\ast})$, i.e. there exists a function $h \in \lebe^1(U^{\varepsilon};{\wedgeq}^{n}(\R^n)^{\ast})$, such that for any $\psi \in \hold_c^{\infty}(U^{\varepsilon})$
\[
\int u \wedge d \psi \dx = \int h \psi \dx.
\]
\end{lemma}
\noindent Given the validity of this lemma and Lemma \ref{lemma:pointwise:solenoidal:Lip}, we can directly infer global solenoidality of $\E_{\Omega}$ in $U^{\varepsilon}$, i.e.
\begin{corollary}
    Let $u \in \hold^1(\bar{U};{\wedgeq}^{n-1}(\R^n)^{\ast})$ satisfying $d u =0$. Then $\dif \E_{U} u=0$ in the sense of distributions in $U^{\varepsilon}$.
\end{corollary}

\begin{proof}[Proof of Lemma \ref{lemma:globalsolenoidality:Lip}]
   Recall that 
    \begin{align*}
    \int_{U} \E_{\Omega} u \wedge d\psi \dif x = \int &\sum_{I=(i_1,...,i_n)\in\mathbb{N}^{n}} \Big(\phi_{i_n} \wedge \dif \phi_{i_{n-1}} \wedge \dots \wedge \dif \phi_{i_1} \Big. \\ 
    &\int_{(\R^{n})^{n}} \int_{M_I} u(z) \nu^{n-1} \dHaus^{n-1}(z) ~\textup{d}\mu^I(\mathbf{x}_{I})\Big) \wedge d \psi \dif x. 
      \end{align*}
As $u \in \hold^{1}(\bar{\Omega};{\wedgeq}^{n-1}(\R^n)^{\ast})$ we may infer the existence of $h_1 \in \lebe^1(U;{\wedgeq}^{n}(\R^n)^{\ast})$, such that 
    \begin{align*}
        \int_{U} \E_{\Omega} u \wedge d\psi \dif x = \int h_1 \dif \psi \dif x + \int  &\sum_{I=(i_1,...,i_n)\in\mathbb{N}^{n}} \Big(\phi_{i_n} \wedge \dif \phi_{i_{n-1}} \wedge \dots \wedge \dif \phi_{i_1} \Big. \\ 
    &\int_{(\R^{n})^{n}} \int_{M_I} u(y)  \nu^{n-1} \dHaus^{n-1}(z) ~\textup{d}\mu^I(\mathbf{x}_{I})\Big) \wedge d \psi \dif x, 
    \end{align*}
i.e. we replaced the dependence of $u$ in $z$ by a dependence in $y$. We now may apply Stokes' theorem backwards and write 
    \begin{align*}
        \int_{M_I} u(y)  \nu^{n-1} \dHaus^{n-1}(z) &= \int_{\partial M_I} \langle (z-y), u(y) \rangle  \nu^{n-2} \dHaus^{n-2}(z) \\
        &= \sum_{i=1}^n \int_{M_I^s}   \langle (z-y), u(y) \rangle  \nu^{n-2} \dHaus^{n-2}(z),
    \end{align*}
where $M_I^s$ denotes the face of $M_I$ that is opposite to $x_{i_s}$. With the same argumentation as in Lemma \ref{lemma:globalsolenoidality}, we observe that for $s=1,...,n-1$      
\begin{align*}
    \sum_{I=(i_1,...,i_n)\in\mathbb{N}^{n}} &\Big(\phi_{i_n} \wedge \dif \phi_{i_{n-1}} \wedge \dots \wedge \dif \phi_{i_1} \Big.  \\
    &\cdot \int_{(\R^{n})^{n}} \int_{M_I^s}  \langle (z-y), u(y) \rangle  \nu^{n-2} \dHaus^{n-2}(z) ~\textup{d}\mu^I(\mathbf{x}_{I})\Big)=0
\end{align*}
    as $\phi_{i_s}$ is a partition of unity, the integral over $M_i^s$ \emph{does not} depend on $i_s$ and $\int d\mu_{i_s} =1$.

    With the same argument
    \begin{align*}
       \sum_{I=(i_1,...,i_n)\in\mathbb{N}^{n}} &\Big(\phi_{i_n} \wedge \dif \phi_{i_{n-1}} \wedge \dots \wedge \dif \phi_{i_1} \Big.  \\
    &\cdot \int_{(\R^{n})^{n-1}} \int_{M_I^n}  \langle (z-y), u(y) \rangle  \nu^{n-2} \dHaus^{n-2}(z) ~\textup{d}\mu_I(\mathbf{x}_{I})\Big)\\
        = \sum_{I=(i_1,...i_{n-1})\in\mathbb{N}^{n-1}} &\Big(\dif \phi_{i_{n-1}} \wedge \dots \wedge \dif \phi_{i_1} \Big.  \\
    &\cdot \int_{(\R^{n})^{n-1}} \int_{M_I^n}  \langle (z-y), u(y) \rangle  \nu^{n-2} \dHaus^{n-2}(z) ~\textup{d}\mu_I(\mathbf{x}_{I})\Big)     
    \end{align*}
and we conclude
    \begin{align*}
    \int_{U} \E_{\Omega} u \wedge d\psi \dif x = \int h_1& \psi \dif x +  \int \sum_{I=(i_1,...,i_{n-1})\in\mathbb{N}^{n-1}} \Big(\dif \phi_{i_{n-1}} \wedge \dots \wedge \dif \phi_{i_1} \Big. \\ 
    &\cdot \int_{(\R^{n})^{n-1}} \int_{M_I} \langle(z-y),u(y)\rangle  \nu^{n-2} \dHaus^{n-2}(z) ~\textup{d}\mu_I(\mathbf{x}_{I})\Big) \wedge d \psi \dif y.
      \end{align*}
    Now integrating by parts yields for each summand
    \begin{align*}
        \int & \dif \phi_{i_{n-1}} \wedge \dots \wedge \dif \phi_{i_1} \cdot 
    \int_{(\R^{n})^{n-1}} \int_{M_I} \langle(z-y),u(y)\rangle  \nu^{n-2} \dHaus^{n-2}(z) ~\textup{d}\mu_I(\mathbf{x}_{I})\Big) \wedge d \psi \dif y \\
    &=(-1) \int \ \phi_{i_{n-1}} \wedge \dots \wedge \dif \phi_{i_1} \cdot 
    \int_{(\R^{n})^{n-1}} \int_{M_I} u(y)  \nu^{n-2} \dHaus^{n-2}(z) ~\textup{d}\mu_I(\mathbf{x}_{I})\Big) \wedge d \psi \dif x \\
     &+ \int \ \phi_{i_{n-1}} \wedge \dots \wedge \dif \phi_{i_1} \cdot 
    \int_{(\R^{n})^{n-1}} \int_{M_I} -\langle (z-y),\dif u(y) \rangle  \nu^{n-2} \dHaus^{n-2}(z) ~\textup{d}\mu_I(\mathbf{x}_{I})\Big) \wedge d \psi \dif y \\
    \end{align*}
As in the case of convex domains (cf. \eqref{w11} ff.) we may argue that the latter summand may be written as 
\[
\int h_2 \psi \dif y 
\]
for some $h_2 \in \lebe^1(\R^n;{\wedgeq}^n(\R^n)^{\ast})$. The first term now has the same form as the term we started with. An inductive argument (as in the proof of Lemma \ref{lemma:globalsolenoidality}) then yields that there is an $\lebe^1$-function $h$, such that 
\[
\int_{U} \E_{\Omega} u \wedge d \psi \dif y = -\int_{\Omega} u \wedge d \psi \dif y + \int_{U} h \psi \dif y.
\]
\end{proof}

\subsection{(Uniform) $\lebe^1$-bounds} \label{sec:Lip:L1bounds}
This also roughly follows the same treatment as for the convex case. For a rigorous proof, it is at this point where we need to have clever choice for the definition of simplices, cf. Remark \ref{rem:simplex:choice}. We begin with the analogue of Lemma \ref{lem:convhullL1bound} that is adjusted to the setting at hand:

\begin{lemma} \label{lem:LipschitzL1boonhund}
There exists a constant $c>0$ depending on the dimension $n$ and $\Lip(\Omega,\mathcal{O})$ such that the following holds true. Let $I=(i_{1},...,i_{n})\in\mathbb{N}^{n}$ such that all the cubes touch and are contained in $\Omega_{\delta}$. Then for any $v\in\hold^{\infty}(\overline{\Omega};\R^{n})$ there holds 
\begin{align}\label{eq:L1bound}
\left\vert \int_{(\R^{n})^{n}} \int_{M(\mathbf{x}_{I})} v(z) (\nu) \dHaus^{n-1}(z) ~\textup{d}\mu^I(\mathbf{x}_{I}) \right\vert \leq \frac{c}{\delta_{I}}\int_{\tilde{M}_I}|v(x)|\dif x, 
\end{align}
where $\tilde{M}_I$ is the union of all manifolds $M(\mathbf{x}_I)$ for $x_{i_j} \in Q^{i_j}$ and
\begin{align}\label{eq:deltadef0}
\delta_{I}:=\min\{\ell(Q_{i_{j}})\colon\;j\in\{1,...,n\}\}.
\end{align}
\end{lemma} 
\begin{proof}[Proof]
    First of all, observe by transformation rule and it suffices to show this lemma whenever $\Omega$ has smooth boundary; cf. Lemma \ref{lemma:BZ}. 
    Comparable to the proof of Lemma \ref{lem:convhullL1bound} we may use the given parametrisation of the manifold $M(x_I)$ via $t_i$ and rewrite the inner integral in terms of $t_i$. In particular, recall that, given $x_1,\ldots,x_n$ in $\Omega$, the variable $z$ is recovered as follows. Write $x_i=(y_i,h_i) \in \partial \Omega \times (0,\eta)$ in variables of the collar domain (i.e. $x_i = y_i + h_i \nu_{\partial \Omega}(y_i)$). Then $\tilde{z}(t)= \sum_{i=1}^n t_i x_i$ is the convex combination and 
    \[
    z(t,x) = \mathbb{P} \tilde{z} + \left(\sum_{i=1}^n t_i h_i\right) \nu_{\partial \Omega}(\mathbb{P} \tilde{z}).
    \]
    Recalling the definition of $\tilde{\mathbb{D}}^{n-1} = \{t \in [0,1]^n \colon \sum_{i=1}^n t_i=1\}$ from the proof of Lemma \ref{lem:convhullL1bound}, we obtain with transfomation rule
    \[
    \int_{M(\mathbf{x}_{I})} v(z) (\nu) \dHaus^{n-1}(z) \leq 
    C \int_{\tilde{\mathbb{D}}^{n-1}} \vert v( z(t,x)) \vert \dHaus^{n-1}(t).
    \]
    Let us shortly remark that the constant $C>0$ only depends on properties of $\Omega$ that we fix beforehand (i.e. $\Lip(\Omega,\mathcal{O})$ and also the curvature), which may be verified by locally writing $\partial \Omega =\{(x',g(x'))\}$.
    Now we can again subdivide the integration domain into $\mathbb{D}^{n-1}_i$ and consider
    \[
    \mathrm{(I)}= \int_{\R^n} \int_{\tilde{\mathbb{D}}^{n-1}_1} \vert v( z(t,x)) \vert \dHaus^{n-1}(t) \dif \mu^1(x_1) 
    \]
    Again usage of Fubini and transformation rule (again, this can be calculated by hand when writing $\partial \Omega  =\{(x',g(x'))\}$) yields a bound
    \[
    \mathrm{(I)} \leq C(n,\Omega) \int_{\tilde{M}_I} \vert v \vert \dif z
    \]
    and similarly to Lemma \ref{lem:convhullL1bound} we conclude that \eqref{eq:L1bound} holds.
\end{proof}
Now observe the following from the construction: If all cubes $Q_{i_1}, \ldots Q_{i_n}$ touch, their distance to the boundary and their sidelength is comparable to some $\eta>0$. Consequently, for the interior cubes $Q^{i_k}=\Psi(Q_{i_k})$ we have due to \ref{item:Jonesrefined3a:v2}
\[
\tfrac{1}{16}\eta \leq \ell(Q^{i_k}) \leq 16 \eta, \quad c^{-1} \eta \leq \dist(Q^{i_k},\partial \Omega) \leq c\eta,
\]
and, due to \ref{item:Jonesrefined4:v2}, we further have an $\rho=\rho(n,\Omega)>0$ such that
\[
Q^{i_{k_1}} \subset B_{\rho \eta}(Q^{i_{k_2}}), \quad 1 \leq k_1,k_2 \leq n.
\]
Due to construction of the manifold $M_I$ we further know that for any $y \in \tilde{M}_I$ (cf. Lemma \ref{lem:LipschitzL1boonhund})
\[
c^{-1} \inf_{k} \inf_{y' \in Q^{i_k}} \dist(y',\partial \Omega) \leq \dist(y,\partial \Omega) \leq c \sup_{k} \inf_{y' \in Q^{i_k}} \dist(y',\partial \Omega),
\]
where $c$ hails from Bi-Lipschitz constant of the diffeomorphism of $\Omega$ to a smooth domain.

We infer that
\[
\tilde{M}_I \subset B_{\rho \eta}/Q^{i_1} \cap \{y \in \Omega \colon c^{-1} \eta \leq \dist(y,\partial \Omega) \leq c\eta \},
\]
i.e. this set is contained in the intersection of a ball and a strip.
This observation enables us to repeat the proof of Lemma \ref{lemma:L1bound:c1}
\begin{lemma}[Uniform $\lebe^{1}$-bounds] \label{lemma:L1bound:c1:Lip}
    Let $u \in \hold^1(\overline{\Omega};\wedgeq^{n-1}(\R^{n})^{*})$. Then we have 
    \begin{align*}
    \mathscr{E}_{\Omega}[u] \in \lebe^1(\R^n;{\wedgeq}^{n-1}(\R^{n})^{*}).
    \end{align*}
In particular, the following holds:
\begin{enumerate}
\item\label{item:L1bound1:v2}  $\displaystyle 
            \Vert \mathscr{E}_{\Omega}[u] \Vert_{\lebe^1(\R^n)} \leq c \Vert u \Vert_{\lebe^1(\Omega)};
         $
\item\label{item:L1bound2:v2} there exists a constant $\mathtt{c}>0$ and $\delta_{0}>0$  such that $0<\delta<\delta_{0}$ implies 
\begin{align}\label{eq:stripsum:v2}
\int_{\{x\in\R^{n}\setminus\overline{\Omega}\colon\;\delta/2<\mathrm{dist}(x,\partial\Omega)<\delta\}}|\mathscr{E}_{\Omega}[u]|\dif x \leq c \int_{\{x\in\Omega\colon\; \mathtt{c}^{-1} \delta<\mathrm{dist}(x,\partial\Omega)<\mathtt{c}\delta\}}|u|\dif x.
\end{align}
\end{enumerate}
\end{lemma}

\subsection{Finishing the proof of Theorem \ref{thm:main:Lip}} \label{sec:Lip:final}
In this section, we finish the proof of Theorem \ref{thm:main:Lip}. In particular, we make the transition between a $\hold^1$-extension and the final $\lebe^1$-extension.

To this end, recall that $\Omega_{\delta}$ is a domain compactly contained in $\Omega$ and that for each $x \in \partial \Omega_{\delta}$ we have (for $c$ being a Bi-Lipschitz constant for $T \colon \bar{\Omega} \supset W \to \partial V \times [0,1)$)
\[
c^{-1} \delta \leq \dist(x,\partial \Omega) \leq c \delta.
\]
Moreover, in Section \ref{sec:Lip:L1bounds} we have shown that the $\lebe^1$-bound \emph{does not} depend on the choice of $\varepsilon$, i.e.
\begin{corollary}
    Let $\delta < \varepsilon_2/2$. There is an extension operator $\E_{\delta} \colon \hold^1(\Omega_{\delta};\R^n) \to \lebe^1(\Omega^{\varepsilon_2/2};\R^n)$, such that 
    \begin{enumerate}
        \item $\Vert \E_{\delta} u\Vert_{\lebe^1(\Omega^{\varepsilon_2/2})} \leq C \Vert u \Vert_{\lebe^1(\Omega_{\delta})}$ with a constant $C>0$ not depending on $\delta>0$;
        \item $\diver u =0$ in $\mathcal{D}'(\Omega_{\delta})$ $\Longrightarrow$ $\diver \E_{\delta} u =0$ in $\mathcal{D}'(\Omega^{\varepsilon_2/2})$.
    \end{enumerate}
\end{corollary}
With this statement we may now finish the proof of Theorem \ref{thm:main:Lip}.

\begin{proof}[Proof of Theorem \ref{thm:main:Lip}]
   First of all, fix $u \in \lebe^1(\Omega;\R^n)$ that obeys $\diver u =0$ in $\mathcal{D}'(\Omega)$ and $\epsilon_2$ as before. For each $\delta < \epsilon_2$ we may define $\E_{\delta} \phi_{\gamma} \ast u$ for $\gamma \leq C^{-1} \delta$, such that $\phi_\gamma \ast u$ is solenoidal in $\Omega_{\delta}$.
    As in the convex framework we conclude that 
    \[
    \hold^1(\bar{\Omega};\R^n) \cap \{u \in \lebe^1(\Omega;\R^n) \colon \diver u =0\}
    \]
    is dense in the space divergence-free $\lebe^1$-functions of $\Omega$.

    Hence, we can define the extension for general $u \in \lebe^1(\Omega;\R^n)$ as follows: Consider $u_j \in \hold^{1}_{\diver}(\bar{\Omega};\R^n)$ with $\Vert u_j - u \Vert_{\lebe^1(\Omega)} \to 0$. Then $\E_{\Omega}u_j$ is a Cauchy sequence in $\lebe^1_{\diver}(\Omega^{\varepsilon_2/2};\R^n)$ and define $\E_{\Omega} u$ as the limit of $(\E_{\Omega}[u_j])$. This limit is unique (and independent of the choice of $u_j$), forms a linear map and maps divergence-free functions to divergence-free functions, as $\diver \E_{\Omega}[u_j]=0$ for any $j$. This completes the proof.
\end{proof}

\subsection{Global extension} \label{sec:globalext}
    In this section, we shortly discuss how, under specific circumstances, one can get a global solenoidal extension. 
    As already discussed in Remark \ref{remark:thomasmueller}, the method of defining simplices has its limitations to regions close to the boundary. In fact, for geometrically non-trivial domains (i.e. domains not homeomorphic to the ball), an extension generally fails, for example consider the annulus and $u = \tfrac{x}{\vert x \vert^n}$.

    We recall the following extension result for $\lebe^p$ functions in the case $1<p<\infty$, \cite[Proposition 2.1]{HKMT}. Suppose that $\Omega$ is smoothly bounded and that its boundary consists of $J+1$ connected components $\Gamma^0$, $\Gamma^1$,...,$\Gamma^J$. We may write $\Gamma^j = \partial \mathcal{O}^j$ for $\mathcal{O}^j \subset \Omega^{\complement}$ open, where $\mathcal{O}^0$ is unbounded and $\mathcal{O}^j$ is bounded for $j \geq 1$.

    Pick $y_j \in \mathcal{O}^j$. Then
        \begin{equation}\label{def:gj}
            g_j := \tfrac{x-y_j}{\vert x - y_j \vert^n}.
        \end{equation}
        
    \begin{theorem}[\cite{HKMT}] \label{thm:Lp}
        Let $1<p<\infty$ and let $\Omega \subset \R^n$ be a domain with smooth boundary and $U$ be bounded and connected such that $\Omega \subset \subset U$. Suppose that $u \in \lebe^p(\Omega;\R^n)$ with $\diver u=0$ . Then there is $\lambda_j \in \R$, $j=1,...,J$ and $\bar{u} \in \lebe^p(U;\R^n)$ with $\diver \bar{u}=0$ in $\mathcal{D}'(U)$ and
        \begin{equation*}
            \bar{u} - \sum_{j=1}^J \lambda_j g_j =u \quad \text{on } \Omega.
        \end{equation*}
        Moreover, we can define a linear and bounded map $\tilde{\E} \colon \lebe^p(\Omega;\R^n) \to \lebe^p(U;\R^n)$ with $\bar{u} =\tilde{\E}u$.
    \end{theorem}
    This theorem is proven by solving the equation
        \[
        \begin{cases} \diver \bar{u} =0 & \text{in } \mathcal{O}_j, \\
        \bar{u} \cdot \nu = (u-\lambda_j g_j) \cdot \nu & \text{on } \partial \mathcal{O}_j,
        \end{cases}
        \]
    Observe this only works, as the normal trace operator $\lebe^p_{\diver}(\Omega;\R^n) \to B^{-1/p}_{p,p}(\partial\Omega)$ exists for $1<p<\infty$ and above equation may be solved provided that
        \[
        \int_{\partial \Ocal_j} (u - \lambda g_j) \cdot \nu =0,
        \]
    which implicitly gives $\lambda_j$.
    In Theorem \ref{thm:main:Lip} we proved a local extension provided that $\Omega$ is Lipschitz and $p=1$, which makes solution theory for the divergence-equation unavailable. Theorem \ref{thm:Lp} may be however combined with Theorem \ref{thm:main:Lip} to give a global extension.

    Again, subdivide the complement of $\Omega$ into $J+1$ connected components $\mathcal{O}_j$ and define $g_j$ as in \eqref{def:gj}.
    \begin{theorem} \label{thm:global:extension}
        Let $\Omega \subset \R^n$ be a Lipschitz domain. Then there is a linear operator $\E' \colon \lebe^1_{\diver}(\Omega;\R^n) \to \lebe^1_{\diver}(\R^n;\R^n)$ such that 
        \[
        \E' u = u+ \sum_{j=1}^J \lambda_j g_j \quad \text{on } \Omega.
        \]
    \end{theorem}
    The key in proving this Theorem is the following observation:
    \begin{lemma} \label{lemma:higherorder}
        Let $\Omega \subset \R^n$ be a Lipschitz domain and $\mathscr{E}_{\Omega}$ be the extension operator from Definition \ref{def:extension:Lip} mapping to $\lebe^1_{\diver}(\Omega^{\varepsilon};\R^n)$, $\Omega^{\varepsilon} = \{ x \in \R^n \colon \dist(x, \Omega) < \varepsilon\}$. Consider the strip $A^{\delta} = \{ x \in \Omega^{\complement} \colon \delta < \dist(x, \Omega) < \varepsilon\}$. Then, for any $m \in \N$, $\E_{\Omega} u \in \sobo^{m,1}(A^{\delta};\R^n)$ and 
        \[
            \Vert \E_{\Omega} u \Vert_{\sobo^{m,1}(A^{\delta})} \leq C \delta^{-m} \Vert u \Vert_{\lebe^1(\R^n)}.
        \]
    \end{lemma}
    \begin{proof}
        The proof can be achieved by the same means as Lemma \ref{lemma:L1bound:c1} and Lemma \ref{lemma:L1bound:c1:Lip}. In particular, taking additional derivatives only indroduces derivatives of the terms $\phi_i$. But these derivatives scale like $l(Q^i)^{-1}$, i.e. like $C(\Omega) \delta^{-1}$.
    \end{proof}
    This now tells us that we can extend $u$ by a small region in $\lebe^1$ and then apply Theorem \ref{thm:Lp}.

    \begin{proof}[Proof of Theorem \ref{thm:global:extension}]
      Consider the extension $\E_{\Omega}$ of $u$ to $\E u \in \lebe^1_{\diver}(\Omega^{\varepsilon};\R^n)$. There is a smoothly bounded domain $A$ that is contained in the strip $A^{\varepsilon/2}$. There are $J+2$ connected components of $(A^{\varepsilon/2})^{\complement}$, namely an unbounded set $\mathcal{O}_0$, $J$ sets $\mathcal{O}_j$ with $\mathcal{O}_j \cap \Omega = \emptyset$ and $\mathcal{O}_{J+1}$ that satisfies $\Omega \setminus \R^n$. By Theorem \ref{thm:Lp} there is a linear and bounded  map $\bar{E} \colon \lebe^p_{\diver}(A;\R^n) \to \lebe^p_{\diver}(B_R(0);\R^n)$ with $\bar{E} u= \E u +\sum_{j=1}^{J+1} g_j$ on $A$. Consequently, for suitable $m\in \N$
      \[
        \Vert \bar{E} (\E u) \Vert_{\lebe^p(B_R(0))} \leq C \Vert \E u \Vert_{\lebe^p(A)} \leq C \Vert \E u \Vert_{\sobo^{m,1}(A)} \leq C \Vert u \Vert_{\lebe^1(\Omega)}
      \]     
    As the function $\E u$ is however already divergence-free on $A \cup \mathcal{O}_{J+1} \subset \Omega^{\varepsilon}$, we conclude that $\lambda_{J+1}=0$. Moreover, as the divergence is a local operator, one can see that
        \[
            \E' u = \begin{cases}
                        \E u + \sum_{j=1}^J \lambda_j g_j & \text{on } A \cup \mathcal{O}_{J+1}, \\
                        \bar{E} u & \text{on } \bigcup_{j=0}^J \mathcal{O}_J, 
                    \end{cases}
        \]
    already is divergence-free. This finishes the proof.
    \end{proof}


\section{Divergence-free extensions in the Sobolev case}\label{sec:sobolev}
In this section we deal with the extension result in the Sobolev space setting. While this task differs from the zeroth order regularity regime $\lebe^{1}:=\sobo^{0,1}$ discussed in the preceding sections, the underlying strategy is analogous for all spaces $\sobo^{m,1}$ with $m\in\mathbb{N}$. Hence, we will focus on $m=1$ and confine ourselves to commenting on $m\geq 2$ throughout.

Throughout, $\Omega'$ denotes a domain such that $\Omega \Subset \Omega'$. First, once a formula for the requisite extension operator $\mathscr{E}\colon\sobo^{1,1}(\Omega;\R^{n})\to\sobo^{1,1}(\Omega';\R^{n})$ is established, it is substantially easier to verify $\mathrm{div}(\mathscr{E}u)=0$. Indeed, in this case, $\mathscr{E}u\in\sobo^{1,1}(\Omega';\R^{n})$ allows to reduce the claimed solenoidality of $\mathscr{E}u$ to establishing $\mathrm{div}(\mathscr{E}u)=0$ $\mathscr{L}^{n}$-a.e. in $\Omega'$. On the contrary, the $\lebe^{1}$-case moreover required to establish $\mathrm{div}(\mathscr{E}u)\in\lebe^{1}(\Omega')$ so that, e.g., $\mathrm{div}(\mathscr{E}u)$ in particular is not a measure concentrated at $\partial\Omega$. On a formal level, this corresponds to verifying a (very weak form of a) Neumann-type condition at the boundary. From a technical perspective, this is linked to the fact that there is no boundary trace operator on $\lebe^{1}(\Omega;\R^{n})$.

In the higher regularity regime, however, there is a boundary trace operator indeed, and so the conditions for the extension to be matched at the boundary $\partial\Omega$ are stronger. Formally, this corresponds to additional (weak) Dirichlet boundary conditions to be fulfilled, and hence it is harder to come up with a formula for the extension. In fact, the construction employed for the $\lebe^1$-case satisfies a weak form of a Neumann boundary condition, but \emph{cannot} satisfy a Dirichlet boundary condition. Especially, the (relatively straightforward) extension formula \eqref{def:Eomega} fails in the $\sobo^{1,1}$ case.

Following the outline from Section \ref{sec:strategy}, the basic idea is to take a specific extension $\mathscr{E}_{0}$ of $\sobo^{1,1}$-functions to full space which does not  necessarily preserves solenoidality. We subsequently construct a bounded linear corrector $\mathscr{R}\colon \sobo^{1,1}(\Omega;\R^{n}) \to \sobo^{1,1}_0(\overline{\Omega'}\setminus \Omega;\R^n)$ that obeys
 \[
(u\in\sobo^{1,1}(\Omega;\R^{n})\;\text{and}\;\mathrm{div}(u)=0\;\text{in}\;\mathscr{D}'(\Omega))\Longrightarrow \diver(\mathscr{R} u) = - \diver(\mathscr{E}_{0}u)\;\text{in}\;\lebe^{1}(\Omega'). 
 \]
and define $\mathscr{E}= \mathscr{E}_{0} + \mathscr{R}$. Based on this construction, the main result of this section is as follows:
\begin{theorem}[Divergence-free extensions in the Sobolev case] \label{thm:W11}
    Let $\Omega \subset \R^n$ be an open and bounded set with Lipschitz boundary. 
    \begin{enumerate}
        \item If $\Omega$ is in addition convex, there exists a \emph{bounded and linear  extension operator} \[
        \mathscr{E}\colon \sobo^{1,1}(\Omega;\R^n) \longrightarrow \sobo^{1,1}(\R^n;\R^n)
        \]
    such that
        \[
        \diver(u) =0\;\text{$\mathscr{L}^{n}$-a.e. in } \Omega \quad \Longrightarrow \quad \diver(\mathscr{E}u) =0\; \text{$\mathscr{L}^{n}$-a.e. in } \R^n.
        \]
        \item There exists $\Omega'$ with $\Omega \Subset \Omega'$ and a \emph{bounded and linear extension operator} 
        \[
        \mathscr{E}\colon \sobo^{1,1}(\Omega;\R^n) \longrightarrow \sobo^{1,1}(\Omega';\R^n)
        \]
    such that
        \[
        \diver(u) =0\;\text{$\mathscr{L}^{n}$-a.e. in } \Omega \quad \Longrightarrow \quad \diver(\mathscr{E}u) =0\; \text{$\mathscr{L}^{n}$-a.e. in } \Omega'.
        \]
    \end{enumerate}
    
\end{theorem}
As a consequence of this theorem, a bound on the $\sobo^{1,1}$-norm away from the boundary and the result of \textsc{Kato et al.} \cite{HKMT}, we also get the counterpart to Theorem \ref{thm:global:extension}.
\begin{corollary}[Theorem \ref{thm:global:extension}, Sobolev case]\label{coro:global:extension}
     Let $\Omega \subset \R^n$ be an open and bounded set with Lipschitz boundary. Then there exists a finite dimensional space $X \subset \sobo^{1,1}(\Omega;\R^n)$ and a bounded and linear operator
        \[
        \widetilde{\mathscr{E}} \colon \sobo^{1,1}(\Omega;\R^n) \longrightarrow \sobo^{1,1}(\R^n;\R^n)
        \]
    such that $(\widetilde{\mathscr{E}} u - u)|_{\Omega}\in X$ and 
        \[ 
         \diver(u)=0\; \text{$\mathscr{L}^{n}$-a.e. in } \Omega \quad \Longrightarrow \quad \diver (\widetilde{\mathscr{E}} u) =0\; \text{$\mathscr{L}^{n}$-a.e. in } \R^n.
        \]
\end{corollary}

\subsection{Definition of the extension operator}
We now introduce the extension operator underlying Theorem \ref{thm:W11}. Here, we focus on the geometric set-up for sets with Lipschitz boundaries as displayed in Section \ref{sec:Lip}. The corresponding statements in the case of convex sets $\Omega\subset\R^{n}$ are  accomplished with the natural modifications subject to the set-up given in Section \ref{sec:L1}.

We remind the reader of the construction undertaken in the previous section including exterior and interior Whitney cover $\mathcal{W}_2$ and $\mathcal{W}_1$, a map $\Psi$ between those and a partition of unity $(\phi_i)_{i \in \N}$ with the properties \ref{item:POU1}--\ref{item:POU3}. 
\smallskip 

For the following, we recall that a robust way of extending $\sobo^{1,1}(\Omega;\R^{n})$-maps without taking care of the divergence constraint is given by \textsc{Jones}' extension operator
\begin{equation} \label{def:E0}
    \mathscr{E}_0 u = \rho\sum_{i \in \N} \phi_{i}(u)_{\widetilde{Q}_{i}},\qquad u\in\sobo^{1,1}(\Omega;\R^{n}),
\end{equation}
where $\rho\in\hold_{c}^{\infty}(\R^{n};[0,1])$ satisfies $\rho=1$ in a neighbourhood of $¸\overline{\Omega}$ and $(u)_{\tilde{Q}_i}$ is the average of $u$ on $\tilde{Q}_i=\tfrac{1}{2}Q_i$. For completeness, we record the following result:
\begin{proposition}[Jones {\cite{Jones}}]
    Let $\Omega \subset \R^n$ be a bounded Lipschitz domain and $1 \leq p \leq \infty$. Then the extension $\mathscr{E}_0 \colon \sobo^{1,p}(\Omega;\wedgeq^{n-1}(\R^n)^{\ast}) \to \sobo^{1,p}(\R^n;\wedgeq^{n-1}((\R^n)^{\ast}))$ is a bounded linear operator.
\end{proposition}
Based on the results from Section \ref{sec:simplices}, we now define $\Omega'$ to be a slightly larger domain, such that the reflection operator $\Psi$ is well-defined as well as simplices $M(\mathbf{x}_I)$ for $x_{i_j}$ in $\Omega'$  $\sup \vert x_{i_j} - x_{i_{j'}} \vert \leq \sup \dist(x_{i_j}, \partial \Omega)$.

For $k\in\{1,...,n-1\}$ we define correctors $\Rcal_k$ as follows. Recalling the notation $Du \cdot \nu$ and $[Du\cdot\nu]$ as introduced  \eqref{eq:Skmeaning} and \eqref{eq:inner}, we set
\begin{equation} \label{def:Rcal}
    \Rcal_k(y)= \sum_{I=(i_{1},...,i_{k+1}) \in \N^{k+1}} \phi_{i_{k+1}} \wedge \dif \phi_{i_k} \wedge \ldots \wedge \dif \phi_{i_1} \fint_{(\R^{n})^{k+1}} R(\mathbf{x}_I) \dif \mu^I,
\end{equation}
where $R(\mathbf{x}_I)$ is given by 
\begin{equation} \label{def:Rcal2}
    R(\mathbf{x}_I) := \int_{M(\mathbf{x}_{I})} [Du(z) \cdot \nu] (y-z) \dif\mathscr{H}^{k}(z).
\end{equation}
For future reference, we moreover define 
\begin{equation} \label{def:Scal}
    \Scal_k (y) = \sum_{i \in \N^{k+1}} \phi_{i_{k+1}} \wedge \dif \phi_{i_k} \wedge \ldots \wedge \dif \phi_{i_1} \fint_{(\R^{n})^{k+1}} S(\mathbf{x}_I) \dif \mu^I,
\end{equation}
where $S(\mathbf{x}_I)$ is given by 
\begin{equation} \label{def:Scal2}
    S(\mathbf{x}_I) := \int_{M(\mathbf{x}_{I})} [Du(z) \cdot \nu] \dif\mathscr{H}^{k}(z).
\end{equation}
Following our convention \eqref{eq:Skmeaning}ff., we note that $\Rcal_{k} u$ and $\Scal_{k}u$ take pointwise values in the correct space $\wedgeq^{n-1}(\R^n)^{\ast}$ and $\wedgeq^n(\R^n)^{\ast}$, respectively: We have $\phi_{i_{k+1}} \wedge \ldots \dif \phi_{i_1} \in \wedgeq^k(\R^n)^{\ast}$ and, by definition of $[Du \cdot \nu](y-z)$ and $Du \cdot \nu$ (see  \eqref{eq:inner} and \eqref{eq:Skmeaning}), we have that $R(\mathbf{x}_I) \in \wedgeq^{n-1-k}(\R^n)^{\ast}$
and $S(\mathbf{x}_I) \in \wedgeq^{n-k}(\R^n)^{\ast}$. Lastly, we put 
 \begin{align*}
            \widetilde{\mathcal{R}}_k u := \begin{cases} \Rcal_k u & \text{in } \Omega' \setminus \Omega, \\
            0 &  \text{in } \Omega,
            \end{cases}
\end{align*}
and then define 
\begin{align}\label{eq:TheSobolevAxe}
\mathscr{E}_{\Omega}[u] := \mathscr{E}_0 [u] + \mathscr{R}_{\Omega}[u] := \mathscr{E}_{0}[u] + \sum_{k=1}^{n-1} (-1)^k \frac{(n-k-1)!}{(n-1)!}\widetilde{\Rcal}_k u. 
\end{align}

\subsection{Boundedness of $\mathscr{E}_{\Omega}$} 
We now establish that $\mathscr{E}_{\Omega}$ defines a $\sobo^{1,1}$-bounded, solenoidality preserving extension operator on $\hold^{\infty}(\overline{\Omega};\wedgeq^{n-1}(\R^{n})^{*})$. In particular, we first show that $\E_{\Omega}$ is bounded from $\sobo^{1,1}(\Omega;\wedgeq^{n-1}(\R^{n})^{*})$ to $\sobo^{1,1}(\Omega';\wedgeq^{n-1}(\R^{n})^{*})$ for \emph{any} $u\in\hold^{\infty}(\Omega;\wedgeq^{n-1}(\R^{n})^{*})$, and satisfies $\di(\mathscr{E}_{\Omega}[u])=0$ in $\mathscr{D}'(\Omega')$ whenever $\di(u)=0$ in $\Omega$. Similar to the previous section, a density argument allows to extend $\mathscr{E}_{\Omega}$ to $\sobo^{1,1}(\Omega;\wedgeq^{n-1}(\R^{n})^{*})$ by continuity (also see Remark \ref{rem:techdetail} below). The key features for this programme are summarised in the following lemma: 
\begin{lemma} \label{lemma:w11:1}
    In the above situation, the following hold for each $k\in\{1,...,n-1\}$:
    \begin{enumerate}
        \item \label{w11:1:a} If $u \in \hold^1(\overline{\Omega};\wedgeq^{n-1}(\R^n)^{\ast})$, then $\widetilde{\Rcal}_k u \in \sobo^{1,\infty}(\Omega';\wedgeq^{n-1}(\R^n)^{\ast})$.
        \item \label{w11:1:b} There exists a constant $c>0$ such that we have 
        \begin{align*}
        \|\widetilde{\mathcal{R}}_{k}[u]\|_{\sobo^{1,1}(\Omega')}\leq c\|u\|_{\sobo^{1,1}(\Omega)}\qquad\text{for all}\;u\in\hold^{1}(\overline{\Omega};\R^{n}). 
        \end{align*}
        Hence, $\widetilde{\Rcal}_k$ extends to a (non-relabeled) bounded linear operator that maps $\sobo^{1,1}(\Omega;\wedgeq^{n-1}(\R^n)^{\ast})$ to $\sobo^{1,1}(\Omega';\wedgeq^{n-1}(\R^n)^{\ast})$.
    \end{enumerate}
\end{lemma}
The proof of Lemma \ref{lemma:w11:1} requires several steps, some of which are parallel to the $\lebe^{1}$-analogues from Sections \ref{sec:L1} and \ref{sec:Lip}.
To establish the requisite $\sobo^{1,1}$-bounds, we need to formulate the following counterpart of Lemma \ref{lem:convhullL1bound} or Lemma \ref{lem:LipschitzL1boonhund}, respectively:
\begin{lemma}[Boundedness] \label{lem:W11bound}
Let $r\in\{1,...,n\}$. There exists a constant $c>0$ such that for any $I=(i_{1},...,i_{r})\in\mathbb{N}^{r}$ and any $v\in\hold^{\infty}(\overline{\Omega};\bigwedge^{n-1}(\R^{n})^{*})$ there holds 
\begin{align}\label{eq:W11bound}
\left\vert \int_{(\R^{n})^{r}} \int_{M(\mathbf{x}_{I})} v(z) \nu \dHaus^{r-1}(z) ~\textup{d}\mu^I(\mathbf{x}_{I}) \right\vert \leq \frac{c}{\delta_{I}^{n+1-r}}\int_{\widetilde{M}_I}|v(x)|\dif x, 
\end{align}
where $\widetilde{M}_I$ is the union of all manifolds $M(\mathbf{x}_I)$ for $x_{i_j} \in Q^{i_j}$, $j\in\{1,...,n\}$, and
\begin{align}\label{eq:deltadef1}
\delta_{I}:=\min\{\ell(Q_{i_{j}})\colon\;j\in\{1,...,n\}\}.
\end{align}
\end{lemma} 
The proof of this lemma is the same as for Lemma \ref{lem:LipschitzL1boonhund}, just in some dimension lower (note that we had to adjust the scaling to fit the higher codimension).
Observe that the set $\widetilde{M}_I$ has the following properties: If $Q^{i_1},...,Q^{i_r}$ are contained in the set close to each other, {meaning that their distance is bounded by $c\delta_I$ for some suitably small, yet universal $c>0$ and thus have roughly the same diameter, then 
\begin{itemize}
\item $\widetilde{M}_{I} \subset\ball_{C\delta_{I}}(Q^{i_{1}})$ and 
\item $A^{\delta} = \{x \in \Omega \colon c^{-1} \delta \leq \dist(x,\partial \Omega) \leq c\delta\}$
\end{itemize} 
for some constants $c,C>0$ that only depend on $\Lip(\Omega,\mathcal{O})$. This allows to prove the following result by analogous means as those for Lemma \ref{lemma:L1bound:c1}:
 \begin{lemma}\label{lemma:ManuelNeuer}
    Let $u \in \hold^1(\overline{\Omega};{\wedgeq}^{n-1}(\R^n)^{\ast})$. There exist  constants $c_1,c_2>0$, only depending on the underlying space dimension $n\in\mathbb{N}$, the constructions made for simplices on $\Omega$ from Section \ref{sec:simplices}) and
    $\delta_{0}>0$ such that 
\begin{align}\label{eq:stripsum:v3}
\int_{\{x\in\Omega' \setminus\overline{\Omega}\colon\;\mathrm{dist}\delta/2 <(x,\partial\Omega)<\delta\}} |\widetilde{\mathcal{R}}_k[u]|+|\!\D \widetilde{\mathcal{R}}_k[u]|\dif x \leq c_1 \int_{c_2^{-1}\delta <\{x\in\Omega\colon\;\mathrm{dist}(x,\partial\Omega)<c_2 \delta\}}|\!\D u| + \vert u \vert \dif x.
\end{align}
holds for all $0<\delta<\delta_{0}$.
\end{lemma}
Since the argument is parallel to that for Lemma \ref{lemma:L1bound:c1}, we omit the proof here. Based on Lemma \ref{lemma:ManuelNeuer}, the proof of Lemma \ref{lem:W11bound} follows along the lines of Lemma \ref{lemma:L1bound:c1}/Lemma \ref{lemma:L1bound:c1:Lip}. We now come to the:
\begin{proof}[Proof of Lemma \ref{lemma:w11:1}] Ad \ref{w11:1:a}. 
    Recall that every summand in the definition of $\Rcal_k$ is an element of $\hold_c^{\infty}(\Omega'\setminus\overline{\Omega};{\wedgeq}^{n-1}(\R^{n})^{*})$. Moreover, suppose that $i_1,...,i_k,i_{k+1}$ are such that the (exterior) cubes $Q_{i_1},...,Q_{i_k},Q_{i_{k+1}}$ touch. Then, by \ref{item:Whitney2} and \ref{item:Whitney3}, their sidelengths are uniformly comparable. Together with \ref{item:POU3}, this implies for a universal dimensional constant $c>0$
    \begin{align} \label{est:1}
         \left \Vert \phi_{i_{k+1}} \wedge \dif \phi_{i_k} \wedge \ldots \wedge \dif \phi_{i_1} \right \Vert_{\lebe^{\infty}(\Omega')} \leq c \ell(Q_{i_1})^{-k}\;\;\text{and} \\
         \label{est:2}
         \left  \Vert \phi_{i_{k+1}} \wedge \dif \phi_{i_k} \wedge \ldots \wedge \dif \phi_{i_1} \right \Vert_{\sobo^{1,\infty}(\Omega)'} \leq c \ell(Q_{i_1})^{-k-1},
    \end{align}
   where we used that it is no loss of generality to assume that $\ell(Q_{i_{1}})\leq 1$. Moreover, we observe that the $\mathscr{H}^k$-measure of the manifold $M(\mathbf{x}_{I})$ is bounded by $c\ell(Q_{i_1})^{k}$ and that $|y-z|$ is bounded by $c\ell(Q_{i_1})$ for $y \in Q_{i_1}$ and $z \in M(\mathbf{x}_{I})$. As a consequence, we obtain the estimates
   \begin{align}
       &\left \Vert R_k(\mathbf{x}_I) \right \Vert_{\lebe^{\infty}(\Omega')} \leq C\ell(Q_{i_1})^{k+1} \Vert u \Vert_{\hold^{1}(\Omega)}, \label{est:3} \\
       &\left \Vert R_k(\mathbf{x}_I) \right \Vert_{\sobo^{1,\infty}(\Omega')} \leq C\ell(Q_{i_1})^k \Vert u \Vert_{\hold^{1}(\Omega)}. \label{est:4}
   \end{align}
   Based on \eqref{est:1}--\eqref{est:4}, we conclude that any the $\sobo^{1,\infty}$-norm of every summand is uniformly bounded, so 
    \begin{align} \label{w1infty}
        \left \Vert \phi_{i_{k+1}} \wedge \dif \phi_{i_k} \wedge \ldots \wedge \dif \phi_{i_1} \fint_{(\R^{n})^{k+1}} R(\mathbf{x}_I) \dif \mu^I \right \Vert_{\sobo^{1,\infty}} \leq c \Vert u \Vert_{\hold^{1}(\Omega)},
    \end{align}
    and, therefore,
     \begin{align*}
        \left \Vert \phi_{i_{k+1}} \wedge \dif \phi_{i_k} \wedge \ldots \wedge \dif \phi_{i_1} \fint_{(\R^{n})^{k+1}} R(\mathbf{x}_I) \dif \mu^I \right \Vert_{\sobo^{1,1}} \leq c\ell(Q_{i_1})^n \Vert u \Vert_{\hold^{1}(\Omega)}.
    \end{align*}
    Now fix $j \in \N$. As only finitely many cubes intersect with $Q_{j}$, we obtain that 
    \begin{align*}
        \sum_{I \in \N^{k+1} \colon i_1 =j}  \left \Vert \phi_{i_{k+1}} \wedge \dif \phi_{i_k} \wedge \ldots \wedge \dif \phi_{i_1} \fint_{(\R^{n})^{k+1}} R(\mathbf{x}_I) \dif \mu^I \right \Vert_{\sobo^{1,1}(\Omega')} \leq c \ell(Q_j)^{n}\Vert u \Vert_{\hold^{1}(\Omega)}.
    \end{align*}
    Summing over all $j \in \N$ then yields that 
    \begin{align*}
         \sum_{I \in \N^{k+1}}  \left \Vert \phi_{i_{k+1}} \wedge \dif \phi_{i_k} \wedge \ldots \wedge \dif \phi_{i_1} \fint_{(\R^{n})^{k+1}} R(\mathbf{x}_I) \dif \mu^I \right \Vert_{\sobo^{1,1}(\Omega')} & \leq c\, \sum_{j \in \N} \ell(Q_j)^n \Vert u \Vert_{\hold^{1}(\Omega)} \\ & \leq c\, \mathscr{L}^n(\Omega' \setminus \overline{\Omega}) \Vert u \Vert_{\hold^{1}(\Omega)}.
    \end{align*}
    Hence, the sum defining $\mathcal{R}_{k}$ converges absolutely in $\sobo^{1,1}(\Omega';\wedgeq^{n-1}(\R^{n})^{*})$, and therefore we also have $\widetilde{\Rcal}_k \in \sobo^{1,1}(\Omega';{\wedgeq}^{n-1}(\R^n)^{\ast})$. Recalling \eqref{w1infty}, we then find a pointwise bound on $\widetilde{\Rcal}_k$ and, thus, $\widetilde{R}_k \in \sobo^{1,\infty}(\Omega';{\wedgeq}^{n-1}(\R^n)^{\ast})$.

Ad \ref{w11:1:b}. By construction, $u\mapsto \widetilde{R}_{k}u$ is linear, and we need to show its boundedness. Assuming that $u \in \hold^1(\Omega;{\wedgeq}^{n-1}((\R^n)^{\ast}))$, one sees that due to the part \ref{w11:1:a}, $\E u \in \sobo^{1,1}(\Omega';{\wedgeq}^{n-1}((\R^n)^{\ast}))$. Moreover, due to Lemma \ref{lemma:ManuelNeuer} we get that the $\lebe^{1}$ norm of $Du$ on $\Omega'$ is bounded by the $\sobo^{1,1}$ norm of $u$ in $\Omega$, i.e. 
    \begin{equation*}
        \Vert D\tilde{\Rcal}_k u \Vert_{\lebe^{1}(\Omega')} \leq \Vert u \Vert_{\sobo^{1,1}(\Omega)}. 
    \end{equation*}
    The fact that $\tilde{\Rcal}_k \equiv 0$ on $\Omega$ allows us to use Poincar\'es inequality and conclude that $\Vert \tilde{\Rcal}_k \Vert_{\sobo^{1,1}(\Omega')} \leq C \Vert u \Vert_{\sobo^{1,1}}$. Using a density arugment one then can show this for any $u \in \sobo^{1,1}(\Omega;{\wedgeq}^{n-1}(\R^n)^{\ast})$.
\end{proof}

\subsection{Solenoidality of the extension}

Lemma \ref{lemma:w11:1} allows us to verify solenoidality by pointwise considerations. We have:
\begin{lemma} \label{lemma:w11:2}
Let $u \in \sobo^{1,1}(\Omega;\wedgeq^{n-1}(\R^n)^{\ast})$. Then, the following identities hold for $\mathscr{L}^{n}$-a.e. $y \in \Omega' \setminus \bar{\Omega}$:
    \begin{enumerate}
        \item \label{w:11:2a} We have 
            \begin{equation} \label{difE0}
                \dif \E_0 u(y) = \Scal_1(y);
            \end{equation}
        \item \label{w:11:2b} for $k=1,\ldots,n-1$ we have 
            \begin{equation} \label{difEk}
                \dif \E_k u(y) =(n-k) \left[ \Scal_k(y) - \Scal_{k+1}(y) \right];
            \end{equation}
        \item \label{w:11:2c} if $u \in \sobo^{1,1}(\R^n;{\wedgeq}^{n-1}(\R^n)^{\ast})$ obeys $\dif u=0$, then 
            \begin{equation} \label{Sn}
                \Scal_n(y) =0.
            \end{equation}
    \end{enumerate}
\end{lemma}
Combining these two lemmas \ref{lemma:w11:1} and \ref{lemma:w11:2} together yields Theorem \ref{thm:W11}.

Observe that $\Rcal_k$ is given by a sum, that is locally finite and consists of smooth summands. Therefore, $\Rcal_k \in \hold^{\infty}(\Omega' \setminus \bar{\Omega})$. This allows a pointwise computation of the exterior derivative.

We proceed with the proof of Lemma \ref{lemma:w11:2}.
\begin{proof}[Proof of Lemma \ref{lemma:w11:2}]
    \textbf{\ref{w:11:2a}:} We use the trick of introducing a new index $i_2$ that we already used in the $\lebe^1$-case and the fundamental theorem of calculus (i.e. Stokes' theorem for one-dimensional manifolds with boundary) \footnote{In this case, the vector $\nu$ may be seen as the tangent to the 1D manifold $M(x_{i_1},x_{i_2})$, i.e. the path from $x_{i_1}$ to $x_{i_2}$. The letter $\nu$ just unifies the notation.}
    \begin{align*}
        \dif \E_0 u &= \sum_{i \in \N} \dif \phi_i \fint u(z) \dif \mu^i \\
        &= \sum_{i_1,i_2 \in \N} \phi_{i_2} \wedge \dif \phi_{i_1} \left( \fint u(z) \dif \mu^i - \fint u(\tilde{z}) \dif \mu^j(\tilde{z}) \right) \\
        &= \sum_{i_1,i_2 \in \N} \phi_{i_2} \wedge \dif \phi_{i_1}  \left( \fint \int_{M(x_{i_1},x_{i_2})} 
        D u(z) \cdot \nu \dHaus^1(z) \mu^{i_1}(x_{i_1}) \mu^{i_2}(x_{i_2}) \right) \\
        &= \Scal_1(y).
    \end{align*}
~
\noindent \textbf{\ref{w:11:2b}: } Observe that we calculate the exterior derivative with respect to the variable $y$, i.e. the only 'interesting' term in the integrand is $(y-z)$. In particular,
    \begin{align*}
        \dif \Rcal_k u &= \sum_{I \in \N^{k+1}} \dif \phi_{i_{k+1}} \wedge \dif \phi_{i_k} \wedge \ldots \wedge \dif \phi_{i_1} \fint R(\mathbf{x}_I) \dif \mu^I
         \\
        & + \sum_{I \in \N^{k+1}} \phi_{i_{k+1}} \wedge \dif \phi_{i_k} \wedge \ldots \wedge \dif \phi_{i_1} \fint \dif_y( R(\mathbf{x}_I)) \dif \mu^I = \mathrm{(I)} + \mathrm{(II)}.
    \end{align*}
    First, let us handle the second term. Utilising a coordinate-wise computation one can show that 
    \begin{equation} \label{eq:dify}
    \dif_y(R(x_I)) = \int_{M(\mathbf{x}_I)} \dif_y[Du(z) \cdot \nu](y-z) \dif z = (n-k)\int_{M(\mathbf{x}_I)} [Du(z) \cdot \nu] \dif z.
    \end{equation}
   Indeed, let us give a short explicit argument why \eqref{eq:dify} is true. For simplicity, assume that for $\omega \colon \Omega \to \R$
    \[
    u(z) = \omega(z) \dif x_1 \wedge \ldots \wedge \dif x_{n-1}.
    \]
    Let $\alpha \in \{1,...,n\}$ and $\beta \in \{1,...,n-1\}^{k-1}$ with $\alpha \neq \beta_j$ and $\beta_j < \beta_{j+1}$, $j=1,...,k-1$. Denote by $\beta^{\complement} \in \{1,...,n-1\}^{n-k}$ the coordinate directions which are \emph{not} in $\beta$. Denote by 
    \[
    \nu_{\alpha \beta} = \nu \cdot (e_{\alpha} \wedge e_{\beta_1} \wedge \ldots \wedge e_{\beta_n}).
    \]
    Then, (for $m(\alpha,\beta) \in \N$ appropriately chosen), 
    \[
    Du(z) \cdot \nu = \sum_{\alpha=1}^n \sum_{\beta} (-1)^{m(\alpha,\beta)} \omega(z) \nu_{\alpha \beta} \dif x_{\beta^{\complement}_1} \wedge \ldots \wedge \dif x_{\beta^{\complement}_{n-k}}.
    \]
    Therefore, in a coordinate-wise representation
    \begin{align}
    (Du(z) \cdot \nu)(y-z) = \sum_{\alpha=1}^n \sum_{\beta} \sum_{j=1}^{n-k}
    (-1)^{m(\alpha,\beta) + (j-1)} \omega(z) \nu_{\alpha \beta} (y-z)_{\beta^{\complement}_{j}} \dif x^j_{\beta^{\complement}} \label{coord} \\
    \dif x^j_{\beta^{\complement}}  = \dif x_{\beta^{\complement}_1} \wedge \ldots \wedge \dif x_{\beta^{\complement}_{j-1}}\wedge \dif x_{\beta^{\complement}_{j+1}} \wedge \ldots \wedge \dif x_{\beta^{\complement}_{n-k}}. \nonumber
     \end{align}
    When computing the exterior derivative in the variable $y$, the only reasonable relevant term is $(y-z)_{\beta^{\complement}_j}$. Summing over all possible $j$ yields
    \[
    \dif_y \bigl (Du(z) \cdot \nu)(y-z) \bigr)= (n-k) \sum_{\alpha=1}^n \sum_{\beta} (-1)^{m(\alpha,\beta)} \omega(z) \dif x_{\beta^{\complement}_1} \wedge \ldots \wedge \dif x_{\beta^{\complement}_{n-k}} =(n-k) Du(z) \cdot \nu.
    \]
    The first term is more complicated as it involves an additional application of Stokes' theorem. In particular, introducing a multiindex $\tilde{I} \in \N^{k+2}$ and $\tilde{I}_r = (i_1,...,i_{r-1},i_{r+1},...,i_{k+1})$, $\tilde{I}_{k+2} = I$, we get by
    \begin{align*}
        \mathrm{(I)} = \sum_{\tilde{I} \in \N^{k+2}} \phi_{i_{k+2}} \wedge \dif \phi_{i_{k+1}} \wedge \dif \phi_{i_k} \wedge \ldots \wedge \dif \phi_{i_1} \fint R(\mathbf{x}_{\tilde{I}_{k+2}})  \dif \mu^{\tilde{I}_{k+2}} - \sum_{r=1}^{k+1} \fint R(\mathbf{x}_{\tilde{I}_r}) \dif \mu^{\tilde{I}_r}
    \end{align*}
    Observe that $M_{\tilde{I}_{k+2}},\ldots,M_{\tilde{I}_1}$ form the boundary of the $k+1$-dimensional simplex $M_{\tilde{I}}$ and that we may apply Stokes' theorem, i.e. noting that 
    \[
    [Du(z) \cdot \nu](y-z) =F(y) \left[\nu \right],
    \]
    for some appropriately chosen $F$ (cf. \eqref{coord})  we obtain 
    \begin{align*}
        \mathrm{(I)} = \sum_{\tilde{I} \in \N^{k+2}} \phi_{i_{k+2}} \wedge \dif \phi_{i_{k+1}} \wedge \dif \phi_{i_k} \wedge \ldots \wedge \dif \phi_{i_1} \fint \int_{M_{\mathbf{x}_{\tilde{I}}}} \dif F [\tilde{\nu}]  \dif \mu^{\tilde{I}},
    \end{align*}
    where $\tilde{\nu}$ is the outer normal w.r.t. $M_{\mathbf{x}_{\tilde{I}}}$. A coordinate-wise computation (if the derivative in $z$ is applied to $Du$ it cancels out due to antisymmetry and if its applied to $(y-z)$ we get the same as in \eqref{coord} ff.) gives
    \[
    \dif F[\tilde{\nu}] =  -(n-k) Du(z) \cdot \tilde{\nu},
    \]
    and therefore we conclude
    \begin{equation}
        \dif \Rcal_k u(y)= (n-k) (S_k(y)-S_{k+1}(y)).
    \end{equation}

    \noindent \textbf{\ref{w:11:2c}: } $S_n$ reads as follows:
    \[
   S_n(\mathbf{x}_I) = \int_{M(\mathbf{x}_I)} Du(z) \cdot \nu \dif z.
    \]
    We simply show that $Du(z) \cdot \nu$ is the same object as $\dif u(\nu)$. $u \colon \Omega \to {\wedgeq}^{n-1}((\R^n)^{\ast})$ and $\nu$ actually is a vector in $\wedgeq^{n}(\R^n)$ (which is isomorphic to $\R$). Therefore, up to sign, we might assume $\nu = e_1 \wedge e_2 \wedge \ldots \wedge e_n$. Then
    \begin{align*}
    Du(z) \cdot \nu = \sum_{\alpha=1}^n (-1)^{\alpha+1} \partial_{\alpha} u \cdot (e_1 \wedge \ldots e_{\alpha-1} \wedge e_{\alpha+1} \wedge \ldots e_n),
    \end{align*}
    which coincides with the coordinate-wise definition of $\dif u(\nu)$. Now, using solenoidality of $u$ yields
    \[
    S_n(\mathbf{x}_I) =0 \quad \Longrightarrow \quad \mathcal{S}_n =0.
    \]
 \end{proof}
\begin{remark} 
This proves the second part of Theorem \ref{thm:W11}. The special case of \emph{convex} $\Omega$ might be treated similarily to the convex $\lebe^1$ case, cf. Section \ref{sec:L1}. In particular, one might show that there is an $R>0$, such that the extension $\E$ obeys $\E u= \mathrm{const} = (u)_{\tilde{Q}_0}$ outside of $B_R(0)$. Solving the boundary value problem
\[
\begin{cases}
    \diver u=0 & x \in B_{2R}(0) \setminus B_R(0), \\
    u= \mathrm{const} & x \in \partial B_R(0), \\
    u=0 & x \in \partial B_{2R}(0).
\end{cases}
\]
then gives a global extension in the convex case.
\end{remark}
\begin{remark}[A technical detail]\label{rem:techdetail} In the $\lebe^1$-case it was necessary to establish a density result for $\hold^1$-functions which are divergence-free. While it is possible to argue in a similar way here, it now however suffices to a construct a sequence of functions $(u_{j}) \subset \sobo^{1,1}(\Omega;\R^n) \cap \hold^2(\overline{\Omega};\R^n)$ converging to a solenoidal field $u$ in $\sobo^{1,1}$ and then argue that the $\lebe^{1}$-norms of $\diver(\mathscr{E}(u_{j}))$ tend to $0$ as $u_j \to u$ in $\sobo^{1,1}(\Omega;\R^{n})$, see the proof of Lemma \ref{lemma:w11:1}.
 \end{remark}
\subsection{Global extension and higher order} \label{sec:higherorder}
Finding a \emph{global} extension for a topologically non-trivial domain is the same story as in the $\lebe^1$-case. As in Lemma \ref{lemma:higherorder} one may now show
 \begin{lemma}
        Let $\Omega \subset \R^n$ be a Lipschitz domain and $\E_{\Omega}$ be the extension operator mapping the space $\sobo^{1,1}(\Omega;{\wedgeq}^{n-1}(\R^n)^{\ast})$ into $\sobo^{1,1}(\Omega;{\wedgeq}^{n-1}(\R^n)^{\ast})$.  Consider the strip $A^{\delta} = \{ x \in \Omega' \colon \delta < \dist(x, \Omega) < \varepsilon\}$. Then, for any $m \in \N$, $\E_{\Omega} u \in \sobo^{m,1}(A^{\delta};\R^n)$ and 
        \[
            \Vert \E_{\Omega} u \Vert_{\sobo^{m,1}(A^{\delta})} \leq C \delta^{-k+1} \Vert u \Vert_{\sobo^{m,1}(\Omega)}.
        \]
    \end{lemma}
\noindent Consequently, one can also show Corollary \ref{coro:global:extension} in the same fashion as Theorem \ref{thm:global:extension}.
\begin{proof}[Proof of Corollary \ref{coro:global:extension}]
    Let $\Omega \subset \R^n$ be Lipschitz bounded. By extension theorem Theorem \ref{thm:W11} there exists a linear extension to some $\Omega'$; moreover close to the boundary of $\Omega'$ we actually have $\E u \in \sobo^{m,1}$, i.e. in some $\sobo^{1,p}$. We then may apply the extension of Kato et. al. to further extend to the full space by ignoring a finite dimensional subspace \cite[Corollary 3.2]{HKMT}.
\end{proof}

    We now further comment on the higher-regularity case.
    Given the formula for the divergence-free extension and the corrector terms for the Sobolev-case $\sobo^{1,p}$, we shortly discuss the analogous setup for $m > 1$. In that case, $\E_0 u$ is given through averaged Taylor-polynomials of $u$ of order $(m-1)$, i.e.
    \[
    \E_0 u = \sum_{i \in \N} \phi_i T^{m-1}_i u, \quad T^{m-1}_i u = \fint \sum_{\kappa=0}^{m-1} D^{\kappa} u(z) [(y-z)^{\kappa}] \dif \mu^i.
    \] 
    Now using that $\diver u=0$, one obtains that $\diver T^{m-1}_i u=0$. Therefore, when calculating the divergence of $\E_0 u$, the only nonzero terms are those, where the derivative falls on $\phi_i$. One can now construct suitable correctors as in \eqref{def:Rcal} ff. to obtain solenoidality of $\E_0 u + \mathcal{R}u$. These correctors are simimlar to $\mathcal{R}_k$ with featuring $D^m u(z)$ instead of $D u$. The tedious computational details are left to the reader. We remind the reader that a key argument lies in the scaling analysis \eqref{est:1}--\eqref{est:4}. On the other hand, as we have more derivatives at play, counterparts to \eqref{est:1} and \eqref{est:2} features an even higher degree of degeneracy if $l(Q_{i_1}) \to 0$. On the other hand, due to the higher order Taylor expansion, we are (also for the correctors) able to get better versions of \eqref{est:3} and \eqref{est:4}.

\section{$\lebe^p$ extensions and cuspidal domains} \label{sec:further}
In this section, we discuss some consequences of the construction undertaken in Sections \ref{sec:L1} and \ref{sec:Lip}.

First of all, we show that the extension constructed before does not only work for $\lebe^1$, but for any Lebesgue space.
We then consider domains $\Omega \subset \R^n$, which do \emph{not} have boundary of Lipschitz regularity and show, that a solenoidal $\lebe^p$-extension result is not possible for some $p$.  More precisely, we show that the assumption that $\Omega$ is Lipschitz bounded \emph{cannot} easily be replaced by slightly worse regularity assumptions, i.e. H\"older regularity.

\subsection{$\lebe^p$ extension of solenoidal fields} \label{sec:plargerone}
Revisiting Section \ref{sec:L1} and \ref{sec:Lip}, it is not challenging to repeat the proof of the extension theorem up to the $\lebe^p$-bound, i.e. the extension operator constructed definitely maps $\lebe^p_{\diver}(\Omega;\R^n) \to \lebe^1_{\diver}(\Omega^{\varepsilon};\R^n)$. The only issue is to show an $\lebe^p$ bound.

For $p= \infty$, this is quite simple. For the manifold $M_I$ we have the immediate bound 
\[ 
\int_{M_I} u(z) \nu(x_I) \dif \mathscr{H}^{n-1}(z) \leq \Vert u \Vert_{\lebe^{\infty}},
\quad
\mathscr{H}^{n-1}(M_I) \leq C \Vert u \Vert_{\lebe^{\infty}} \delta^{n-1}
\]
whenever cubes $Q_{i_1}$,...,$Q_{i_j}$ obey
\[
    \dist(Q_{i_k},Q_{i_j}) \leq C \delta, \quad \ell(Q_{i_k}) \leq C \delta.
\]
Therefore, using that $\Vert d \phi_i \Vert_{\lebe^{\infty}} \leq C \delta^{-1}$ and the formula for $\E_{\Omega}$ we directly arrive at the following statement.
\begin{corollary}
    Let $\Omega \subset \R^n$ be a Lipschitz domain. 
    The operator $\E_{\Omega}$ constructed in Section \ref{sec:Lip} is bounded and linear from $\lebe^{\infty}_{\diver}(\Omega;\R^n)$ to $\lebe^{\infty}_{\diver}(\Omega^{\varepsilon};\R^n)$.
\end{corollary}
Due to interpolation theorems (or by doing the calculation associated to Lemma \ref{lem:convhullL1bound}) we may also show
\begin{corollary} \label{Coro:Lpextension}
     Let $\Omega \subset \R^n$ be a Lipschitz domain, $1<p<\infty$. The operator $\E_{\Omega}$ constructed in Section \ref{sec:Lip} is bounded and linear from $\lebe^{p}_{\diver}(\Omega;\R^n)$ to $\lebe^{p}_{\diver}(\Omega^{\varepsilon};\R^n)$.
\end{corollary}

The same arguments also works in the higher regularity case, cf. the paragraph before lemma \ref{lem:W11bound}. Therefore, one can also show
\begin{corollary} \label{coro:W1pextension}
Let $\Omega \subset \R^n$ be a Lipschitz domain, $1<p \leq \infty$. The operator $\E$ constructed in Section \ref{sec:sobolev} is bounded from $\sobo^{1,p}_{\diver}(\Omega)$ to $\sobo^{1,p}_{\diver}(\Omega')$.
\end{corollary} 
\subsection{On the regularity of $\partial\Omega$} \label{sec:regularityofOmega}
Until now, we assumed $\Omega\subset\R^{n}$ to be a bounded domain with Lipschitz boundary.  Theorem \ref{thm:main:Lip} establishes that this suffices to conclude the existence of a solendoiality-preserving $\lebe^{p}$-bounded extension operator to a superset of $\Omega$. Different from arbitrary $\lebe^{p}$-functions, which can be trivially extended to the entire $\R^{n}$ regardless of the regularity of $\partial\Omega$, preserving solenoidality comes with certain restrictions on $\partial\Omega$. We do not attempt to determine the optimal conditions on $\partial\Omega$ for a solenoidality-preserving extension operator, but will compare the underlying conditions below. In order to focus on the regularity of the boundary and to avoid topological intricacies, the domains considered in the sequel are homeomorphic to the open unit ball. 

The way in which we show below theorem, reveals that there is a fundamental regularity flaw of $\Omega$ which is not connected to its topology. In particular, as suggested by Theorem \ref{thm:Lp}, we could hope for a finite dimensional vector space $X$, such that an extension for all $u \in \lebe^p_{\diver}(\Omega;\R^n) /X$ in the quotient is possible. We, however, construct a whole class of functions that does not allow for an extension (and one may show that no nonzero finite linear combination of those may be extended).
\smallskip

Let $0< \gamma<1$. We write $x=(x',y) \in \R^n$ with $x' \in \R^{n-1}$ and $y \in \R$ for elements of $\R^{n}$ and define 
\[
    \phi(s) = \begin{cases}
                s^{\gamma} & \text{if } s >0 \\
                0 & \text{otherwise.}
            \end{cases}
\]
Based on $\varphi$, we then introduce the sets $\Omega_{+},\Omega_{-}\subset\R^{n}$ by
\begin{align}\label{eq:omegaplusminusdef}
\begin{split}
\Omega_+ =\{ x \in (-1,1)^n \colon y > \phi(\vert x' \vert)\}, \\
\Omega_- =\{ x \in (-1,1)^n \colon y < \phi(\vert x' \vert)\},
\end{split}
\end{align}
so that $\Omega_{+}$ has an outward and $\Omega_{-}$ has an inward cusp at the origin, see Figure \ref{fig:Omegapm}. Both domains have boundaries of class $\hold^{\gamma}$. Denoting 
\begin{align*}
\Omega_{\pm}^{\eta}:=\{x \in \R^n \colon \dist(x,\Omega_{\pm})< \eta\},\qquad\eta>0,
\end{align*}
the main result of the present section then is as follows: 

\begin{theorem}[Limitations on the regularity of $\partial\Omega$] \label{thm:counterexamples}
    Let $\Omega_+$ and $\Omega_-$ be as in \eqref{eq:omegaplusminusdef}. Then the following hold:
    \begin{enumerate} [label=(\alph*)]
        \item \label{counter:1} Let $n=2$. Given $0<\gamma<1$, there exists $p \in [1,\infty)$ such that for any $\eta>0$ there is no bounded linear extension operator $\E \colon \lebe^p_{\diver}(\Omega_+;\R^n) \to \lebe^p_{\diver}(\Omega_+^{\eta};\R^n)$.
        \item \label{counter:2} Let $n \geq 2$. Given $0<\gamma<1$, there is some $p \in (1,\infty)$ such that for any $\eta >0$ there is no bounded linear extension operator $\E \colon \lebe^p_{\diver}(\Omega_-;\R^n) \to \lebe^p_{\diver}(\Omega_-^{\eta};\R^n)$.
        \item \label{counter:3} Let $n \geq 3$. There is a $(\varepsilon,\delta)$-domain  $\Omega \subset \R^n$ and $1<p<\infty$ (see Remark \ref{remark:Lipschitzepsdel} for this terminology) such that for any $\eta>0$ there is no bounded linear extension operator $\E \colon \lebe^p_{\diver}(\Omega;\R^n) \to  \lebe^p_{\diver}(\Omega^{\eta};\R^n)$.
    \end{enumerate}
\end{theorem}
 The counterexamples constructed for above theorem are quite elementary, but we have some restrictions on the exponents $p$. In particular, the case $p=1$ is not covered by every counterexample.
\begin{proof}[Proof of Theorem \ref{thm:counterexamples}
] {On \ref{counter:1}.}
We start by defining the requisite exponents. First, since $0<\gamma<1$, we have that $2<\frac{\gamma+1}{\gamma}$ and thus may choose $2<\alpha<\frac{\gamma+1}{\gamma}$. Then we have that $1<\frac{\gamma+1}{\alpha\gamma}$, and so may choose $1\leq p <\frac{\gamma+1}{\alpha\gamma}$. Because of $\alpha>2$, we have that $\frac{\gamma+1}{\gamma-1+\alpha}<1$ and so arrive at the overall estimate 
\begin{align}\label{eq:barschler}
\frac{\gamma+1}{\gamma-1+\alpha} < p <\frac{\gamma+1}{\alpha\gamma},\;\;\;\text{whereby}\;\;\;(1-p)\gamma + p (1-\alpha)<-1.
\end{align}
Subject to these exponent choices, we  consider the solenoidal field
    \[
    u(x',y) = (y^{-\alpha},0)
    \]
    Then we have that $u \in \lebe^p(\Omega_+;\R^2)$ because of $\alpha p<\frac{\gamma+1}{\gamma}$. For $0<s<1$, we now consider the rectangular domain $A_s=(s^{1/\gamma},s) \times (s, s^{\gamma})$ with two opposite vertices on $\partial\Omega_{+}$, see Figure \ref{fig:Omegapm}. Suppose towards a contradiction that there exists some $0<\eta<1$ such that there is a divergence-free extension  $\overline{u}$ of $u$ to some open neighbourhood of $\Omega_{+}$ which contains $\ball_{2\eta^\gamma}(0)$. It is no loss of generality to assume that 
\begin{align}\label{eq:etachoose}
0<\eta\leq 2^{-\frac{1}{(1-\alpha)(\gamma-1)}},\;\;\;\text{whereby}\;\;\;(0<s<\eta \Longrightarrow s^{(1-\alpha)\gamma}-s^{1-\alpha}\leq -\eta^{(\gamma-1)(1-\alpha)}s^{1-\alpha}).
\end{align}
For all sufficiently small $s>0$, $A_{s}\subset\ball_{2\eta^{\gamma}}(0)$. By Gauss' theorem, we then have for such $s$ that  
\begin{align}\label{eq:solenoidalex}
    \int_{\partial A_s} \overline{u} \cdot \nu_{\partial A_{s}} \dif\mathscr{H}^{1}= \int_{A_{s}}\di(\overline{u})\dif x = 0.
    \end{align}
Note that, even though $\overline{u}$ is a priori only an $\lebe^{p}$-map, the use of \eqref{eq:solenoidalex} in the following can be formally justified by routine mollification,  especially as we further integrate in $s$ at later stage. To keep our exposition at a reasonable length, however, we directly employ  \eqref{eq:solenoidalex}. Since we do not know $\overline{u}$, we will now derive the requisite contradiction based on \eqref{eq:solenoidalex}. 

Because of $\alpha>1$ and $\gamma <1$, we have that $(\gamma-1)(1-\alpha)>0$, and so the definition of $\overline{u}$ and $\nu_{\partial A_{s}}$ gives us
    \begin{align}\label{eq:dangerousgame}
    \begin{split}
    \int_{(\partial A_s ) \cap \Omega_+} \overline{u} \cdot \nu_{\partial A_{s}} \dif\mathscr{H}^{1} & =  \int_{(\partial A_s )\cap \Omega_+} u \cdot \nu_{\partial A_{s}} \dif\mathscr{H}^{1}= -\int_{s}^{s^{\gamma}} y^{-\alpha} \dif y \\ & = \frac{1}{\alpha-1}\Big(s^{(1-\alpha)\gamma}-s^{1-\alpha}\Big) \\ & \!\!\stackrel{\eqref{eq:etachoose}}{\leq} -\frac{1}{\alpha-1}\Big( \eta^{(\gamma-1)(1-\alpha)} \Big)s^{1-\alpha} =: - C s^{1-\alpha}
    \end{split}
    \end{align}
with an obvious definition of $C=C(\eta,\gamma,\alpha)>0$. Therefore, Gauss' theorem implies that 
     \begin{equation} 
     \label{estimate:As} 
     \begin{split} \int_{\partial A_s \cap \Omega_-} \vert \overline{u} \vert \dif \mathscr{H}^{1} & \geq \int_{\partial A_{s}\cap \Omega_{-}}\overline{u}\cdot\nu_{\partial A_{s}}\dif \mathscr{H}^{1} \\ 
     & =  \int_{\partial A_{s}}\overline{u}\cdot\nu_{\partial A_{s}}\dif \mathscr{H}^{1} - \int_{\partial A_{s}\cap \Omega_{+}}u\cdot\nu_{\partial A_{s}}\dif\mathscr{H}^{1} \stackrel{\eqref{eq:solenoidalex},\,\eqref{eq:dangerousgame}}{\geq} Cs^{1-\alpha}.
     \end{split}
     \end{equation}
    We now estimate the $\lebe^p$-norm of $\overline u$ on the rectangle $[0,\eta]\times[0,\eta^{\gamma}]$ by use of \eqref{estimate:As} and Fubini's theorem. To this end, we observe that
    \begin{align*}
\int_{[0,\eta]\times[0,\eta^{\gamma}]} \vert \overline {u}(x) \vert^p \dif x & = \frac{1}{2} \int_0^{\eta} \int_0^{\eta^{\gamma}} \vert \overline u(x',y) \vert^p \dif y \dif x' + \frac{1}{2} \int_0^{\eta^{\gamma}} \int_0^{\eta} \vert \overline u(x',y) \vert^p \dif x' \dif y \\
        & \geq \frac{1}{2} \int_0^{\eta} \int_{s}^{s^{\gamma}} \vert \overline u(s,y) \vert^p \dif y \dif s + \frac{1}{2} \int_0^{\eta} \int_{s^{1/\gamma}}^{s} \vert \overline u(x',s) \vert^p \dif x' \dif s \\
        & \geq  \frac{1}{2} \int_0^{\eta} \int_{\partial A_s \cap \Omega_-} \vert \overline u\vert^p \dif \mathscr{H}^{1} \dif s
    \end{align*}
    because of $0<\eta<1$ and $0<\gamma<1$. Since we have 
\begin{align*}
\mathscr{H}^{1}(\partial A_{s} \cap \Omega_-) = (s^{\gamma}-s^{1/\gamma}) < s^{{\gamma}},\qquad 0<s<1,
\end{align*}
 Jensen's inequality and \eqref{estimate:As} then combine to the lower bound
    \begin{align*}
2\int_{[0,\eta]\times[0,\eta^{\gamma}]} \vert \overline {u}(x)\vert^p \dif x &\geq \int_0^{\eta} \mathscr{H}^{1}(\partial A_s \cap \Omega_-)^{1-p}  \left( \int_{\partial A_s \cap\Omega_{-}} \vert u \vert \dif \mathscr{H}^1 \right)^p \dif s\\
        &\geq C^{p}\int_0^\eta s^{(1-p){\gamma}}s^{p(1-\alpha)} \dif s, 
     \end{align*}
where $C=C(\eta,\gamma,\alpha)>0$ is as in \eqref{estimate:As}. By \eqref{eq:barschler}, this lower bound is infinite, and this is the desired contradiction.

\begin{figure}
\begin{tikzpicture}[domain=-4:4]
\begin{axis}
[axis lines= none]
 \coordinate[label=above:\Large{\textcolor{black}{$\Omega_+$}}] (D) at (-0.1,0.4);
  \coordinate[label=above:\Large{\textcolor{black}{$\Omega_-$}}] (E) at (0.62,-0.2);
   \filldraw[color=Gump, pattern=north east lines, pattern color=Gump] (0.09,0.264) rectangle (0.4,0.473);
   \coordinate[label=above:\Large{\textcolor{Gump}{$A_s$}}] (R) at (0.6,0.3);
   \coordinate[label=above:{\textcolor{black}{$\varphi(\vert x' \vert)$}}] (F) at (-0.6,0.3);
   \filldraw[color=blue, pattern=north west lines, pattern color=blue]
   (-0.16,-0.336) rectangle (0.16,0.336);
   \coordinate[label=above:\Large{\textcolor{blue}{$Z_r$}}] (S) at (-0.3,-0.2);
\addplot [
    name path=A,
    domain=0:1.5, 
    color=black,
    thick,
    samples=200,
]
{sqrt{ x} -0.4*x};
\addplot [
    name path=B,
    domain=-1.5:0, 
    color=black,
    thick,
    samples=200
]
{sqrt{ -x} +0.4*x};
\addplot [
    name path=C,
    domain=-1.5:0,
    color=Gump,
    opacity=0.1,
    samples=200,
]
{-0.4};
\addplot [
    name path=D,
    domain=0:1.5,
    color=Gump,
    opacity=0.1,
    samples=200,
]
{-0.4};
 \addplot[
    domain=0:1.5,
     color=Gump,
     opacity=0.3,
 ] 
 fill between[of=A and D];
\addplot[
    domain=-1.5:0,
     color=Gump,
     opacity=0.3,
 ] 
 fill between[of=B and C];
\end{axis}
\end{tikzpicture} 
\caption{$\Omega_+$ and $\Omega_-$ and the integration paths/domains $A_s$ and $Z_r$ in two dimensions}\label{fig:Omegapm}
\end{figure}
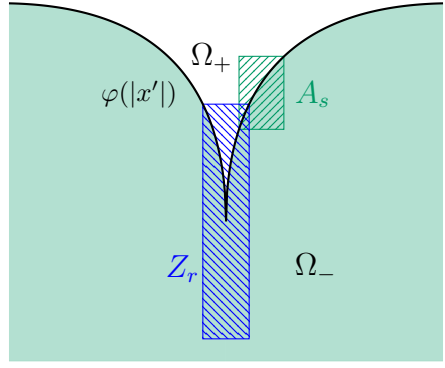
On \ref{counter:2}. Let $1 < p < \infty$, and let $\alpha \in (0,1)$. For simplicity, we only deal with the case $p <\tfrac{n-1}{n-2}$.

    Again, write $x=(x',y) \in \R^{n-1}\times \R$ and consider
    \[
    u(x',y) = \begin{cases}
        \left(\tfrac{x'}{\vert x' \vert^{n-1}} y^{-\alpha},0 \right) & \text{if } y >0 \\
        0 &\text{if } y \leq 0
    \end{cases}
    \]
    Using cylindrical coordinates, one can show that $u \in \lebe^p(\Omega_-;\R^n)$ if and only if 
    \begin{align*}
        (\ast)=\int_0^1 \int_{r^{1/\gamma}}^1 \int_{\partial B^{n-1}_\rho(0)} \vert u(x',r)\vert^p \dif x' \dif \rho \dif r < \infty,
    \end{align*}
    where $B^{n-1}_\rho(0)$ denotes the $(n-1)$-dimensional ball with centre $0$ and radius $\rho$. We may calculate:
    \begin{align*}
        (\ast) &= c_n \int_0^1 r^{-\alpha p} \int_{r^{\gamma}}^1 \rho^{(n-2)(1-p)} \dif \rho \dif r \leq c \int_0^1 r^{-\alpha p} \dif r 
    \end{align*}
    which is smaller than infinity if and only if $p < \alpha^{-1}$.
    In addition, one can show that $\diver u=0$ in $\Omega_+ \cap \{y>0\}$. As $u$ satisfies a jump condition at $\{y=0\}$ we already conclude $\diver u=0$ in $\mathcal{D}'(\Omega_-)$. Now consider the cylinder
    \[
    Z_r= B^{n-1}_{r^{1/\gamma}}(0) \times (-r,r).
    \]
    Suppose that there is a divergence-free $\lebe^p$ extension $\bar{u}$ of $u$. Then there is some $\eta>0$, such that it is an extension onto $Z_\eta$. For any $r< \eta$ we then have
    \[
    \int_{\partial Z_r} \bar{u} \cdot \nu =0.
    \]
    Observe that 
    \begin{align*}
        \int_{\Omega_- \cap \partial Z_r} \bar{u} \cdot \nu &= \int_0^{r} \int_{\partial B^{n-1}_{s^{1/\gamma}}(0)} \tfrac{x'}{\vert x' \vert^{n-1}} s^{-\alpha} \cdot \tfrac{x'}{\vert x' \vert} \dif x' \dif s = c_n\int_0^r s^{-\alpha} \dif s= c_nr^{1-\alpha}.
     \end{align*}
     Therefore, using Gau{\ss}' theorem as in \eqref{estimate:As}, we obtain
     \begin{equation} \label{est:omegaplus}
     \int_{\Omega_+ \cap \partial Z_r} \vert \bar{u} \vert \geq r^{1-\alpha}.
    \end{equation}
    Using that $
    \mathscr{H}^{n-1}(\partial Z_r \cap \Omega_+) = r^{1/\gamma} \mathscr{L}^{n-1}(B^{n-1}(0))$, Fubini's theorem and \eqref{est:omegaplus} imply that 
    \begin{align*}
    \int_{\Omega_+ \cap Z_{\eta}} \vert \bar{u} \vert^p \dif x &= \int_0^{\eta} \int_{\partial Z_r \cap \Omega_+} \vert \bar u \vert^p \dif \mathscr{H}^{n-1} \dif r \\
   & \geq \int_0^\eta (\mathscr{H}^{n-1}(\partial Z_r \cap \Omega_+))^{1-p} \left(\int_{\partial Z_r \cap \Omega_+} \vert  \bar u \vert \dif \mathscr{H}^{n-1} \right)^p \\
   &\geq c_n \int_0^{\eta} r^{(n-1)(1-p)/\gamma}\cdot r^{(1-\alpha)p},
    \end{align*}
    which is infinite if 
    \begin{align} \label{calc:1}
    (n-1)(1-p)/\gamma + (1-\alpha) p <-1.
    \end{align}
    Let now $p < \tfrac{n-1}{n-2}$. Suppose that $u \in \lebe^p(\Omega_-;\R^n)$, i.e. $\delta:= (1-\alpha p) >0$, i.e. we choose $\alpha>0$ to satisfy this. We then have $\bar{u} \notin \lebe^p(\Omega_+)$, i.e. \eqref{calc:1}, if 
    \[
    1/\gamma (n-1)(1-p)+ (1-\tfrac{1}{p})p <-1-\delta,
    \]
    for some small $\delta=(1-\alpha p)>0$, i.e.
    \begin{align*}
        (p-1)(1-(n-1)/\gamma) < -1 -\delta
    \end{align*}
    If we choose
    \[
    1+ \frac{1}{n-2} >p > 1 + \frac{1}{(n-1)/\gamma -1}
    \]
    we obtain $u \in \lebe^p(\Omega_-;\R^n)$ and $\bar{u} \notin \lebe^p(\Omega;\R^n)$.
    
    \noindent \textbf{\ref{counter:3}: } This is a direct consequence of \ref{counter:2} and the fact that in dimension $n>2$, $\Omega_-$ is actually a $(\varepsilon,\delta)$ domain.

\end{proof}

\begin{remark}\begin{enumerate} [label=(\alph*)]
    \item The method of using Gau{\ss}' theorem to obtain counterexamples does not really work for $p=1$ in case \ref{counter:2}. A crucial step is estimate \eqref{est:omegaplus} and the (reverse) use of Jensen's inequality, which does not work for $p=1$.
    \item The examples may be generalised to a larger range of exponents $p$ by taking slightly more care. For instance, in the proof of \ref{counter:2} one may take $\tilde{u}_h = u(x,y-h)$ for some small parameter $h>0$. Then choosing exponents $\alpha>\tfrac{1}{p}$ we can ensure that $\tilde{u}_h \in \lebe^p(\Omega_{-})$ (with an $\lebe^p$ bound depending on $h$). On the other hand, in the same fashion as in the proof one may show a lower bound for the $\lebe^p$-norm of a possible extension, thus giving a lower bound on the operator norm of the extension. This lower bound explodes as $h \to 0$.
    \item Using Theorem \ref{thm:counterexamples} by the interpolation argument of Corollary \ref{Coro:Lpextension} we directly get that there \emph{cannot} exist an extension operator $\lebe^p_{\diver}(\Omega) \to \lebe^p_{\diver}(\Omega)$ for either $p=1$ or $p=\infty$, if the existence fails for some $1<p<\infty$. 
\end{enumerate}
\end{remark}

\begin{remark}
Constructing counterexamples akin to the $\lebe^p$-case, i.e. Theorem \ref{thm:counterexamples}, is more involved in the   $\sobo^{1,p}$ case. We first comment on the \emph{unconstrained} $\sobo^{1,p}$ case. A remarkable work by \textsc{Jones} \cite{Jones} shows that for $(\epsilon,\infty)$-domain (see Remark \ref{remark:Lipschitzepsdel} for this notion) there is such an extension theorem. In particular, some wild, fractal domains fulfil that criterion, but domains with external or internal cusps do not.
For comparison, we recall from \cite{PM,Pob} that for $\Omega_-$ and $\Omega_+$ as in Section \ref{sec:further}
\begin{itemize}
    \item if $p<\infty$, there is no extension operator $\sobo^{1,p}(\Omega_+) \to \sobo^{1,p}(\R^n)$, whereas there is such an extension operator for $p=\infty$, 
    \item if $n=2$, there is no extension operator $\sobo^{1,p}(\Omega_-) \to \sobo^{1,p}(\R^n)$ for any $1<p\leq\infty$, and
    \item if $n>2$, then $\Omega_-$ is an $(\varepsilon,\infty)$-domain and, consequently, there is an extension operator from $\sobo^{1,p}(\Omega_-)$ to $\sobo^{1,p}(\R^n)$ for any $1\leq p \leq \infty$.
\end{itemize}
On the other hand, for $n=2$ being $(\varepsilon,\delta)$ is equivalent to the existence of an extension operator for all $1 \leq p \leq \infty$, \cite{Mazya1,VGL}. It does not seem to be clear as in the $\lebe^p$ case, whether the divergence constraint changes any statement on extendability of Sobolev functions or not.
\end{remark}
\begin{remark} \label{rem:ontheapproach}
    The existence of a counterexample in the $\lebe^p$ case is quite revealing in its own. In the $\sobo^{1,p}$-case (where such a direct counterexample fails and might not exist), instead of using our correctors, it might be imaginable to use a different method of achieving solenoidality, e.g. by using a corrector based on the  Bogovskii-operator (also cf. \cite{BDF12}). In particular, using that corrector one might be able to extend the results of this paper to domains with weaker boundary conditions.

    The counterexample however reveals, that a $\lebe^p$-reflection in the style of \textsc{Jones} \cite{Jones} and then taking a \emph{local} correction \emph{does not} work, as then an extension would be possible for $(\varepsilon,\delta)$-domains. It therefore seems that the construction using simplices, which we obtained for Lipschitz domains but may also be extended to an 'abstract' condition on the domain, is at least close to optimal.
\end{remark}

\section{Applications}\label{sec:applications}
In this concluding section we give some sample applications of the main theorems gathered in Section \ref{sec:intro}. We split this section into different parts according to different applications of the results and techniques used before. First, in Subsections \ref{sec:81} \& \ref{sec:W1curl}, we comment on cases where we replace the differential operator $\diver$ by some general, constant coefficient linear operator $\A$. 
In Subsections \ref{sec:83} \& \ref{sec:84} we shortly show some consequences for regularity theory and inequalities in the borderline case $p=1$.

\subsection{Differential forms and general differential operators} \label{sec:81}

In Section \ref{sec:L1} and \ref{sec:Lip} we used the structure of the differential operator $\diver$ quite heavily. Indeed, recall that we identified $\diver$ with $\dif \colon \hold^{\infty}(\Omega;{\wedgeq}^{n-1}(\R^n)^{\ast}) \to \hold^{\infty}(\Omega;{\wedgeq}^n(\R^n)^{\ast})$. The proof then heavily relies on the use of Stokes' theorem. 

As demonstrated in \cite{Schiffer}, this technique is not only limited to the divergence and we might formulate the following. Denote by $\lebe^p_{\dif}(\Omega;{\wedgeq}^{k}((\R^n)^{\ast}))$ the space of all $u \in \lebe^p(\Omega;{\wedgeq}^{k}((\R^n)^{\ast}))$ that obey $\dif u=0$.
\begin{proposition} \label{prop:diffforms}
    Let $\Omega \subset \R^n$ be a Lipschitz domain and let $1\leq p \leq \infty$. There exists $\Omega^{\varepsilon}$ and a bounded linear extension operator $\E \colon \lebe^p_{\dif}(\Omega;{\wedgeq}^{k}(\R^n)^{\ast}) \to \lebe^p_{\dif}(\Omega^{\varepsilon};{\wedgeq}^{k}(\R^n)^{\ast})$.
\end{proposition}
The proof of this is a mixture of the proofs of Section \ref{sec:L1} and \cite{Schiffer}. Moreover, this extends the result of \cite{MMS} about extension of differential forms to the borderline cases $p=1$ and $p=\infty$.

This result covers the geometrically natural differential operator of exterior differentiation, but one may also pose the question, whether an extension result is possible under \emph{any} linear differential constraint. To fix notation, let $W$ and $Z$ be finite dimensional inner product spaces and let $\Acal_{\beta} \colon W \to Z$ be linear maps for $\beta \in \N_0^n$ $\vert \beta \vert=l$, such that we define the differential operator $\Acal$ as 
\begin{equation} \label{def:Acal}
\Acal v := \sum_{\vert \beta \vert =l} \Acal_{\beta} \partial^{\beta} v, \qquad v \in \hold^{\infty}(\R^n;W).
\end{equation}
Similar as before, we may define 
\[
\lebe^1_{\Acal}(\Omega) := \{ u \in \lebe^1(\Omega;W) \colon \Acal u = 0 \text{ in } \Dcal'(\Omega;Z) \}
\]
and pose the question, whether for Lipschitz bounded $\Omega$ there is a bounded linear extension operator for $\lebe^1$, i.e. $\E_{\Acal} \colon \lebe^1_{\Acal}(\Omega) \to \lebe^1_{\Acal}(\Omega')$.

Proposition \ref{prop:diffforms} settles this question for $\Acal$ being the operator of exterior differentiation (which falls in the framework of \eqref{def:Acal}, leaving the general case open. While the methods of \cite{BGS} for another specific differential operator (the divergence acting on symmetric ($3 \times 3$-matrices) are likely transferable to the present setting, a treatment for \emph{general} differential operators seems to be quite challenging.

Parallel to \cite{BGS}, we conjecture that such an extension is possible for any $\Acal$ that satisfies the (complex) constant rank property. That is, the Fourier symbol of $\Acal$ acting on the complexification of $W$
\begin{equation} \label{def:Acal:FS}
    \Acal [\xi] = \sum_{\vert \beta \vert =l} \Acal_{\beta} \xi^{\beta}, \qquad \xi \in \mathbb C^n \setminus \{0\}
\end{equation}
has constant rank, i.e. the dimension of the kernel does not depend on $\xi$. This condition for instance is satisfied by the operation of exterior differentiation.
\smallskip

On the contrary, it is clear that apart from the homogeneity of the operator, some structural assumption needs to be satisfied. Indeed, take a minimalistic example and $\Acal \colon \hold^{\infty}(\R^2) \to \hold^{\infty}(\R^2)$ as
\[
\Acal \colon u \longmapsto \partial_2 u,
\]
i.e. functions in the kernel of $\Acal$ essentially only depend on the first coordinate.
it is clear that an extension to the full space, i.e. $\E \colon \lebe^1_{\Acal}(\Omega) \to \lebe^1_{\Acal}(\R^2)$ (or even only $\lebe^1_{\Acal,\mathrm{loc}}$), is not possible for topologically trivial domains, for instance take $\Omega= (0,3)^2 \setminus [2,3] \times [1,2]$. Taking the same operator $\Acal$, we can also think of domains where a local extension, i.e. $\E \colon \lebe^1_{\Acal}(\Omega) \to \lebe^1_{\Acal}(\Omega')$ is impossible. For instance, take 
\[
\Omega= \{(x,y) \in \R^2 \colon y \in [0,1], 0 \leq x \leq 1 + \tfrac{1}{2}y \cos(y) \}.
\]
\subsection{Another regularity viewpoint on the extension} \label{sec:W1curl}
Going back to the present setting, which
comprises of divergence-free fields and differential forms. As an example consider the case of closed 1-forms.
\begin{example}\label{ex:curl}
If $\Omega\subset\R^{n}$ is a bounded, topologically trivial domain with Lipschitz boundary $\partial\Omega$, then any field $F=(F_{1},...,F_{n})\in\lebe^{1}(\Omega;\R^{n})$ with vanishing distributional $\curl$ can be written as $F=\D v$ for some $v\in\sobo^{1,1}(\Omega)$. Here, $F$ is said to have \emph{vanishing curl} provided $\partial_{i}F_{j}-\partial_{j}F_{i}=0$ in $\mathscr{D}(\Omega)$ for all $i,j\in\{1,...,n\}$, and we define $\lebe_{\curl}^{1}(\Omega)$ to be the linear space of all $\R^{n}$-valued $\lebe^{1}$-fields on $\Omega$ with vanishing curl. This can be seen by mollification and the classical Poincar\'{e} lemma, and we refer the reader to \cite[Lem. 2.4]{BeckBulicekGmeineder} for more detail. Most importantly, the assignment $T\colon\lebe^{1}(\Omega;\R^{n})\ni F\mapsto v\in\sobo^{1,1}(\Omega)$ is well-defined and linear if we additionally require $(v)_{\Omega}=0$. Letting $E\colon\sobo^{1,1}(\Omega)\to\sobo^{1,1}(\R^{n})$ be an arbitrary but fixed extension operator, we may then put $
\mathscr{E}_{\Omega}[F]:=\D\big(E (TF)\big)$ for $F\in\lebe_{\curl}^{1}(\Omega)$. 
This defines an extension operator 
\[
\E \colon \lebe^1_{\curl}(\Omega) \to \lebe^1_{\curl}(\R^n).
\]
This operator is  a priori not defined on functions with $\curl F \neq 0$, but may extend $\E$ with Hahn-Banach to be defined on $\lebe^1_{\curl}(\Omega)$.
\end{example}

This strategy can be generalised for a certain class of differential operators (which does \emph{not} include the divergence). 

\begin{example}[A generalisation of Example \ref{ex:curl}] \label{ex:curlgen}
Let $\Omega\subset\R^{n}$ be a topologically trivial, open and bounded domain with Lipschitz boundary. If $V,W,Z$ are finite dimensional inner product spaces and $\mathbb{A}_{\alpha}\colon V\to W$, $\mathscr{A}_{\beta}\colon W\to Z$ are linear maps for $\alpha\in\mathbb{N}_{0}^{n}$ with $|\alpha|=k$ and $|\beta|=l$, we consider differential operators 
\begin{align*}
\mathbb{A}u:=\sum_{|\alpha|=k}\mathbb{A}_{\alpha}\partial^{\alpha}u,\;\;\;\mathscr{A}v:=\sum_{|\beta|=l}\mathscr{A}_{\beta}\partial^{\beta}v,\qquad u\in\hold^{\infty}(\R^{n};V),\;v\in\hold^{\infty}(\R^{n};W).
\end{align*}
We denote by $\mathbb{A}[\xi]\colon V\to W$ and $\mathscr{A}[\xi]\colon W\to Z$ the corresponding Fourier symbols, and assume that the short sequence of Fourier symbols 
\begin{align}\label{eq:complex}
V\;\;\stackrel{\mathbb{A}[\xi]}{\longrightarrow}\;\;W\;\;\stackrel{\mathscr{A}[\xi]}{\longrightarrow}\;\; Z\qquad \text{is exact at $W$ for every $\xi\in\R^{n}\setminus\{0\}$}, 
\end{align}
meaning that we have $\ker(\mathscr{A}[\xi])=\mathbb{A}[\xi](V)$ for any $\xi\in\R^{n}\setminus\{0\}$. If $\mathbb{A}$ is $\mathbb{C}$-elliptic, by which we understand that the complexified Fourier symbol map $\mathbb{A}[\xi]\colon V+\mathrm{i}V\to W+\mathrm{i}W$ is a monomorphism for all $\xi\in\mathbb{C}^{n}\setminus\{0\}$, then every $F\in\lebe^{1}(\Omega;W)$ with $\mathscr{A}F=0$ in $\Omega$ can be written as $F=\mathbb{A}TF$ for some $TF\in{\sobo}{^{\mathbb{A},1}}(\Omega)$. As in Example \ref{ex:curl}, the assignment $F\to TF$ can be achieved to be well-defined and linear. Now, subject to the $\mathbb{C}$-ellipticity hypothesis on $\mathbb{A}$, it is established in \cite{GR19} that there exists a bounded linear extension operator $E\colon{\sobo}{^{\mathbb{A},1}}(\Omega)\to{\sobo}{^{\mathbb{A},1}}(\R^{n})$. As in Example \ref{ex:curl}, an operator extending $\mathscr{A}$-free $\lebe^{1}$-fields to $\mathscr{A}$-free $\lebe^{1}$-fields then can be obtained by setting $\mathscr{E}_{\Omega}[F]:=\mathbb{A}(E(TF))$. This strategy e.g. applies to the following situations: 
\begin{itemize} 
\item If $\mathscr{A}=\curl$, then $\mathbb{A}=\D$ is the usual gradient. This operator is $\mathbb{C}$-elliptic, and we recover Example \ref{ex:curl}. 
\item If $\mathscr{A}=\curl\curl$ with $V=\R^{n}$ and $W=\R_{\mathrm{sym}}^{n\times n}$, then $\mathbb{A}$ is the symmetric gradient given by $\mathbb{A}u=\frac{1}{2}(\D u +\D u^{\top})$.
\end{itemize}
\end{example}
However, if we take $\mathscr{A}$ to be the divergence in \eqref{eq:complex}, then $\mathbb{A}=\curl^{\ast}$ (the adjoint of $\curl$) and this operator not only fails to be $\mathbb{C}$-elliptic but even elliptic (meaning that $\mathbb{A}[\xi]$ is not injective on the corresponding real vector space $V$). Since our Theorem \ref{thm:W1} still show that suitable extension operator which map divergence-free to divergence-free fields exist, this case is not entirely hopeless.

Still, one may adapt the viewpoint taken in previous examples and instead of taking 
\[
u \in \lebe^1_{\diver}(\Omega) = \{ u \in \lebe^1(\Omega;\R^n) \colon \diver u =0 \}
\]
one may consider functions in the space $\sobo^{1,\curl^{\ast}}$ (in 3D: $\curl^{\ast}= \curl$, i.e. $\sobo^{1,\curl}$), that is
\[
\sobo^{1,\curl^{\ast}}(\Omega) = \{ v \in \lebe^1(\Omega; \R^{n \times n}_{\mathrm{skew}}) \colon \curl^{\ast} v \in \lebe^1(\Omega;\R^n) \}.
\]
Asking for an extension of $\lebe^{1}_{\diver}(\Omega)$ then is almost the same as searching for an extension result for the space $\sobo^{1,\curl^{\ast}}(\Omega)$, as (for a topologically trivial $\Omega$) any $u \in \lebe^{1}_{\diver}(\Omega)$ can be written as $u=\curl^{\ast} v$.

Contrary to the previously discussed example, adapting this viewpoint mathematically bears no benefits: Both approaches are equally mathematically challenging. On the other hand, closely following the proofs in Sections \ref{sec:L1} \& \ref{sec:Lip}, it is possible to show the following:
\begin{proposition}
Let $\Omega \subset \R^n$ be a bounded Lipschitz domain and $\Omega \Subset \Omega'$.
\begin{enumerate}[label=(\roman*)]
    \item Let $\E_{\diver} \colon \lebe^1_{\diver}(\Omega) \to \lebe^1_{\diver}(\Omega)$ be the extension operator constructed in Section \ref{sec:Lip}. Then the extension $\E_{\diver} \colon \sobo^{1,\diver}(\Omega) \to \sobo^{1,\diver}(\Omega')$ also provides a linear and bounded extension operator.
    \item Let $\dif \colon \hold^{\infty}(\R^n;{\wedgeq}^r((\R^n)^{\ast})) \to \hold^{\infty}(\R^n;{\wedgeq}^{r+1}((\R^n)^{\ast}))$ be the operator of exterior differentiation. Then there exists a linear and bounded extension operator $\E_d \colon \sobo^{1,\dif}(\Omega) \to \sobo^{1,\dif}(\Omega')$.
\end{enumerate}
\end{proposition}
\begin{proof}[Sketch of Proof:]
  The proof of showing that $\diver \E u$ is an $\lebe^1$ function stays the same, as there we have never used the assumption that $\diver u=0$. The only worry is to show that the map actually is bounded.For this, first recall Lemma \ref{lemma:pointwise:solenoidal}, where we computed the divergence pointwisely. In particular (adopting the correspondence of $\diver$-free functions to differential forms), we obtained  
\begin{align*}
    \diver \E u = (-1)^{n-1} \sum_{I\in\mathbb{N}^{n+1}} \Big(\varphi_{i_{n+1}}\dif \phi_{i_n} \wedge \dif \phi_{i_{n-1}} \wedge \dots \wedge \dif \phi_{i_1} \Big) \int \int_{\conv_{I}} d u(z) \nu ~\textup{d}z \mu^{I}.
\end{align*} 
In Lemma \ref{lemma:pointwise:solenoidal} we could argue that the integral vanishes due to Stokes' or Gauss' theorem- this is not possible. Instead observe that this term almost has the structure as $u$ itself, i.e. it features $n$ copies of $\dif \phi_{i_r}$ and an integral over a convex hull of $(n+1)$ points. $\lebe^1$ bounds for this object can now be achieved via the same means as for $u$ itself, for instance Lemma \ref{lemma:L1bound:c1} ff.
The general differential form case then uses the same structural assumptions, but with lower codimension.
\end{proof}

\subsection{Bourgain-Brezis estimate on domains} \label{sec:83}
For this subsection we roughly follow the works \cite{BB,BVS,VS04}. A remarkable result due to \textsc{Bourgain \& Brezis} \cite{BB} is the following inequality: There is a dimensional constant $C>0$, such that for all $u \in \lebe^1(\R^n;\R^n)$ that obey $\diver u =0$ and all $\psi \in (\sobo^{1,n} \cap \lebe^{\infty})(\R^n;\R^n)$ we have
\begin{equation} \label{eq:BB}
    \left \vert \int_{\R^n} u \cdot \psi \dy \right \vert \leq C \Vert u \Vert_{\lebe^1(\R^n)} \Vert D \psi \Vert_{\lebe^n(\R^n)}.
\end{equation}
This estimate \emph{does not} hold if the constraint $\diver u=0$ is removed, as the critical Sobolev embedding $\sobo^{1,n} \hookrightarrow \lebe^{\infty}$ \emph{does not} hold. The proofs of this estimate unfortunately quite heavily rely on the structure of the full space (cf. \cite{BB,VS04}) and cannot be easily translated to the framework of an open set $\Omega$. With the extension result for $\lebe^1$, a result on domains $\Omega$ is however an easy consequence of \eqref{eq:BB}, also see \cite{BVS}: 
\begin{lemma} \label{BB:domains}
    Let $\Omega \subset \R^n$ be a bounded Lipschitz domain. There exists a constant $C=C(n,\Omega)$, such that for all $u \in \lebe^1_{\diver}(\Omega)$ and all $\psi \in (\sobo^{1,n}_0 \cap \lebe^{\infty})(\Omega;\R^n)$ we have 
    \[
    \left \vert \int_{\Omega} u \cdot \psi \dif y \right \vert \leq C \Vert u \Vert_{\lebe^1(\Omega)} \Vert D \psi \Vert_{\lebe^n(\R^n)}.
    \]
\end{lemma}
\begin{proof}
    Let $u \in \lebe^1_{\diver }(\Omega)$. Utilising the \emph{global} extension result Theorem \ref{thm:global:extension}, there exists a finite dimensional space $X$ of smooth functions on $\Omega$, a linear map $L \colon \lebe^1_{\diver}(\Omega) \to X$ and a linear and bounded extension operator $\E_{\Omega} \colon \lebe^1_{\diver}(\Omega) \to \lebe^1_{\diver}(\Omega)$, such that 
    \[
    \E_{\Omega} u \equiv u - Lu \quad \text{in } \Omega.
    \]
    Then \[
    \left \vert \int_{\Omega} u \cdot \psi \dif y \right \vert \leq
    \left \vert \int_{\Omega} Lu \cdot \psi \dif y \right \vert 
    + 
    \left \vert \int_{\Omega} (u-Lu) \cdot \psi \dif y \right \vert. 
    \]
    For the former summand in this inequality we may use that $X$ is finite dimensional, i.e. all norms are comparable, and thus
    \begin{align*}
    \left \vert \int_{\Omega} Lu \cdot \psi \dif y \right \vert &\leq \Vert L u \Vert_{\lebe^{n/(n-1)}(\Omega)} \Vert\psi \Vert_{\lebe^n(\Omega)} 
    \leq C(\Omega)  \Vert L u \Vert_{\lebe^{1}(\Omega)} \Vert D\psi \Vert_{\lebe^n(\Omega)}
    \\
    &
    \leq  C(\Omega) \Vert u \Vert_{\lebe^1(\Omega)} \Vert D \psi \Vert_{\lebe^n(\Omega)}.
     \end{align*}
    For the latter we may use the extension result and the Bourgain-Brezis inequality on the full space
    \begin{align*}
    \left \vert \int_{\Omega} (u-Lu) \cdot \psi \dif y \right \vert &= \left \vert \int_{\R^n} \E_{\Omega} u \psi \dif y \right \vert
     \leq \Vert \E_{\Omega} u \Vert_{\lebe^1(\R^n)} \Vert D \psi \Vert_{\lebe^n(\R^n)}
     \\
     & 
     \leq C(\Omega) \Vert u \Vert_{\lebe^1(\Omega)} \Vert D \psi \Vert_{\lebe^n(\Omega)}.
    \end{align*}
    Combining both estimates directly gives the lemma.
\end{proof}
We remark that with the algebraic observations of \cite{VS13} one may generalise Lemma \ref{BB:domains} to other constraints of first-order.
As a direct consequence of Lemma \ref{BB:domains}, following \cite{BVS}, one obtains.
\begin{corollary}
Let $\Omega \subset \R^n$ be a bounded Lipschitz domain. \begin{enumerate} [label=(\alph*)]
    \item  We have the embedding $\lebe^1_{\diver}(\Omega) \hookrightarrow \sobo^{-1,\tfrac{n}{n-1}}(\Omega;\R^n) = \left(\sobo^{1,n}_0(\Omega;\R^n)\right)'$.
    \item The unique solution $u$ to the vectorial Laplace equation
    \begin{equation*}
        \left\{ \begin{array}{rcll}
            - \Delta u &=& f & \text{in } \Omega, \\
            u&=& 0 & \text{on } \partial \Omega,
        \end{array} \right.
    \end{equation*}
    is in $\sobo^{1,\tfrac{n}{n-1}}(\Omega;\R^n)$ whenever $f \in \lebe^1_{\diver}(\Omega)$.
    \end{enumerate}
\end{corollary}

\subsection{Korn-Maxwell-Sobolev estimate on domains} \label{sec:84}
As a final application of the main results of the present paper, we discuss a borderline case for the so-called \emph{Korn-Maxwell-Sobolev inequalities}. Such inequalities provide a generalisation of the usual Korn-type inequalities to the so-called incompatible framework, meaning that the admissible competitors are not a priori assumed to be gradients. Given an open and bounded connected set $\Omega\subset\R^{3}$ with Lipschitz boundary, such inequalities are of the form 
\begin{align}\label{eq:KMSmain1}
\inf_{Y\in \mathfrak{Y}}\|X-Y\|_{\lebe^{p^{*}}(\Omega)} \leq c\Big(\|\mathscr{A}[X]\|_{\lebe^{p^{*}}(\Omega)} + \|\mathrm{Curl}(X)\|_{\lebe^{p}(\Omega)} \Big),\qquad X\in\hold^{\infty}(\overline{\Omega};\R^{3\times 3}), 
\end{align}
with a suitable space $\mathfrak{Y}$ of correctors.  
Here, $\mathscr{A}\colon \R^{3\times 3}\to\R^{3\times 3}$ is a \emph{part map}. We single out two examples of particular physical relevance: If, e.g., 
\begin{itemize} 
\item $\mathscr{A}[X]=X^{\mathrm{sym}}$ and $\mathfrak{Y}=\mathfrak{so}(3)$, then inequality \eqref{eq:KMSmain1} gives us a variant of the usual Korn inequality \cite{Friedrichs}: Letting $u\in\hold^{\infty}(\overline{\Omega};\R^{3})$ and putting $X=\D u$, \eqref{eq:KMSmain1} becomes 
\begin{align*}
\inf_{Y\in\mathfrak{so}(3)}\|\D u -Y\|_{\lebe^{p^{*}}(\Omega)}\leq c \|\varepsilon(u)\|_{\lebe^{p^{*}}(\Omega)}. 
\end{align*}
\item $\mathscr{A}[X]:= X^{\mathrm{devsym}}:=X^{\mathrm{sym}}-\frac{1}{3}\mathrm{tr}(X)\mathbbm{1}_{3}$ with the $(3\times 3)$-unit matrix $\mathbbm{1}_{3}=(\delta_{ij})_{1\leq i,j\leq 3}$ and
\begin{align*}
\mathfrak{Y}:=\D\mathcal{K}=\{\D\mathfrak{k}\colon\;\mathfrak{k}\;\text{conformal Killing vector}\}, 
\end{align*}
\eqref{eq:KMSmain1} gives us a variant of Korn inequalities involving the deviatoric symmetric gradient: Letting $u\in\hold^{\infty}(\overline{\Omega};\R^{3})$ and putting $X=\D u$, \eqref{eq:KMSmain1} becomes 
\begin{align*}
\inf_{Y\in\mathfrak{Y}}\|\D u -Y\|_{\lebe^{p^{*}}(\Omega)}\leq c\|\varepsilon^{D}u\|_{\lebe^{p^{*}}(\Omega)}. 
\end{align*}
\end{itemize}
Based on these examples, it is clear in how far Korn-Maxwell-Sobolev inequalities generalise the usual Korn-type inequalities to incompatible, so potentially non-curl-free fields. As for the usual Korn-type inequalities, estimates of this form can be obtained by harmonic analysis techniques \cite{GS21,GLN,GLN1,LMN,Schiffer3}. This especially makes inequality \eqref{eq:KMSmain1} difficult to be established for $p=1$. Yet, on full space, the specific structure of the $\mathrm{Curl}$-operator allows to employ a Bourgain-Brezis-type estimate of Van Schaftingen \cite{VS13} to approach \eqref{eq:KMSmain1}, see \cite{CG21,GS21}. Based on the divergence-free extension operator from Theorem \ref{thm:main:Lip}, we now establish how the critical inequalities for $p=1$ can be reduced to the requisite full space estimates:
\begin{proposition}
Let $\Omega\subset\R^{3}$ be open, bounded and connected with Lipschitz boundary. Then there exists a constant $c>0$ solely depending on $\Omega$ such that we have 
\begin{align}\label{eq:KMSsso}
\inf_{Y\in\D\mathcal{K}}\|X-Y\|_{\lebe^{3/2}(\Omega)}\leq c\Big(\|X^{\mathrm{devsym}}\|_{\lebe^{3/2}(\Omega)} + \|\mathrm{Curl}(X)\|_{\lebe^{1}(\Omega)} \Big)\qquad\text{for all}\;X\in\hold^{\infty}(\overline{\Omega};\R^{3\times 3}).
\end{align}
\end{proposition}
\begin{proof}
We first record from \cite[eq. (42)]{LN1} that there exists a constant $c>0$ such that 
\begin{align}\label{eq:NecasLions}
\|X\|_{\lebe^{3/2}(\Omega)}\leq c\Big(\|X^{\mathrm{devsym}}\|_{\lebe^{3/2}(\Omega)}+\|\mathrm{Curl}(F)\|_{\sobo^{-1,3/2}(\Omega)} + \sum_{j=1}^{m}\left\vert \int_{\Omega}\mathbf{e}_{j}\cdot F\dif x\right\vert \Big), 
\end{align}
where $\{\mathbf{e}_{1},...,\mathbf{e}_{m}\}$ is an $\lebe^{2}$-orthonormal basis of $\mathfrak{Y}=\D\mathcal{K}$. Inequality \eqref{eq:NecasLions} can be obtained by Ne\v{c}as-Lions-type techniques, see \cite{LN1,LN2}. We put $g:=\mathrm{Curl}(X)$ so that, in particular $g=(g_{1},g_{2},g_{3})^{\top}\in\lebe^{1}(\Omega;\R^{3\times 3})$ and the row-wise divergence satisfies $\mathrm{div}(g)=0$. By a componentwise application of Theorem \ref{thm:main:Lip}, we may extend $g$ to a row-wise distributionally divergence-free field $\widetilde{g}\in\lebe^{1}(\Omega_{\varepsilon};\R^{3})$ on an open neighbourhood $\Omega_{\varepsilon}$ of $\Omega$ by use of an $\lebe^{1}$-bounded extension operator. 

For future reference, we recall from \cite{BB} that there exists a constant $c>0$ such that 
\begin{align}\label{eq:VShelp}
\int_{\R^{3}}\Phi\cdot\varphi\dif x \leq c\Big(\|\Phi\|_{\lebe^{1}(\R^{3})}\|\nabla\varphi\|_{\lebe^{3}(\R^{n})} + \|\mathrm{div}(\Phi)\|_{\lebe^{1}(\R^{3})}\|\varphi\|_{\lebe^{3}(\R^{3})} \Big) 
\end{align}
holds for all fields $\Phi,\varphi\in\hold_{c}^{\infty}(\R^{3};\R^{3})$. Smooth approximation then yields that \eqref{eq:VShelp} extends to all $\Phi\in\lebe^{1}(\R^{3};\R^{3})$ such that $\mathrm{div}(\Phi)\in\lebe^{1}(\R^{3})$ and all $\varphi\in\sobo^{1,3}(\R^{3};\R^{3})$. 

We choose a smooth cut-off function $\rho\in\hold_{c}^{\infty}(\Omega_{\varepsilon};[0,1])$ such that $\rho=1$ in $\Omega$, and define $\Phi$ to be the extension of $\rho\widetilde{g}\in\lebe^{1}(\Omega_{\varepsilon};\R^{3\times 3})$ to $\R^{3}$ by zero. For $\psi\in\hold_{c}^{\infty}(\Omega;\R^{3\times 3})$, we then denote by $\overline{\psi}$ its trivial extension to $\R^{3}$. Hence, applying \eqref{eq:VShelp} componentwisely and using that $\mathrm{div}(\widetilde{g})=0$ in $\Omega_{\varepsilon}$, we conclude that 
\begin{align}\label{eq:mainboundKMS}
\begin{split}
\int_{\R^{3}}\Phi\cdot\overline{\psi}\dif x = \int_{\R^{3}}(\rho\widetilde{g})\cdot\overline{\psi}\dif x  & \leq c\Big(\|\rho\widetilde{g}\|_{\lebe^{1}(\R^{3})}\|\nabla\overline{\psi}\|_{\lebe^{3}(\R^{n})} + \|\mathrm{div}(\rho\widetilde{g})\|_{\lebe^{1}(\R^{3})}\|\overline{\psi}\|_{\lebe^{3}(\R^{3})} \Big) \\ 
& \leq c(\varepsilon,\rho)\,\|\widetilde{g}\|_{\lebe^{1}(\Omega_{\varepsilon})}\|{\psi}\|_{\sobo^{1,3}(\Omega)} \leq c(\varepsilon,\rho,\Omega)\,\|g\|_{\lebe^{1}(\Omega)}\|{\psi}\|_{\sobo^{1,3}(\Omega)}. 
\end{split}
\end{align}
We then have that 
\begin{align*}
\|g\|_{\sobo^{-1,3/2}(\Omega)} & = \sup_{\substack{\psi\in\hold_{c}^{\infty}(\Omega;\R^{3\times 3})\\ \|\psi\|_{\sobo^{1,3}(\Omega)}\leq 1}}\int_{\Omega}g\cdot\psi\dif x \leq  \sup_{\substack{\psi\in\hold_{c}^{\infty}(\Omega;\R^{3\times 3})\\ \|\psi\|_{\sobo^{1,3}(\Omega)}\leq 1}}\int_{\R^{3}}\Phi\cdot\overline{\psi}\dif x \stackrel{\eqref{eq:mainboundKMS}}{\leq} c\,\|g\|_{\lebe^{1}(\Omega)}.
\end{align*}
We then insert this bound into \eqref{eq:NecasLions} and apply the resulting inequality to $X-\sum_{j=1}^{N}\langle\mathbf{e}_{j},X\rangle\mathbf{e}_{j}$. This yields the claimed inequality \eqref{eq:KMSsso}, and the proof is complete. 
\end{proof}

 \bibliography{literature.bib}
\bibliographystyle{abbrv}
\end{document}